\documentclass[10.5pt, a4paper]{amsart}
\usepackage{amssymb, amsmath, amscd, bm, amsthm, pdfpages, setspace, hyperref, mathtools, subcaption, graphicx, cite, xfrac, empheq}

\def\@Rref#1{\hbox{\rm \ref{#1}}}
\def\Rref#1{\@Rref{#1}}
\theoremstyle{plain}
\newtheorem{theorem}{Theorem}[section]
\newtheorem{proposition}[theorem]{Proposition}

\newtheorem{lemma}[theorem]{Lemma}

\theoremstyle{definition}
\newtheorem{definition}{Definition}[section]
\newtheorem{example}[definition]{Example}
\newtheorem{remark}[definition]{Remark}

\newcommand{\Sect}{\mathop{\rm Sect}}

\newcommand{\re}{\mathop{\rm Re}\nolimits}
\newcommand{\im}{\mathop{\rm Im}\nolimits}
\newcommand{\vertiii}[1]{{\vert\kern-0.25ex\vert\kern-0.25ex\vert #1 
		\vert\kern-0.25ex\vert\kern-0.25ex\vert}_{\mathcal{B}_0,q}}
\newcommand{\bigvertiii}[1]{{\big\vert\kern-0.25ex\big\vert\kern-0.25ex\big\vert #1 
		\big\vert\kern-0.25ex\big\vert\kern-0.25ex\big\vert}_{\mathcal{B}_0,q}}
\newcommand{\biggvertiii}[1]{{\bigg\vert\kern-0.35ex\bigg\vert\kern-0.35ex\bigg\vert #1 
		\bigg\vert\kern-0.35ex\bigg\vert\kern-0.35ex\bigg\vert}}
	
\begin{document}

\title[Characterizations of the Crandall--Pazy Class]{Characterizations of the Crandall--Pazy Class of 
	$C_0$-semigroups on Hilbert Spaces and Their Application to 
	Decay Estimates}

\thispagestyle{plain}

\author{Masashi Wakaiki}
\address{Graduate School of System Informatics, Kobe University, Nada, Kobe, Hyogo 657-8501, Japan}
 \email{wakaiki@ruby.kobe-u.ac.jp}

\begin{abstract}
	We investigate immediately differentiable 
	$C_0$-semigroups 
	$(e^{-tA})_{t \geq 0}$
	satisfying \linebreak  $\sup_{0 < t <1}
	t^{1/\beta}\|Ae^{-tA}\| < \infty$ for some $0 < \beta \leq 1$.
	Such $C_0$-semigroups are referred to as 
	the Crandall--Pazy class of $C_0$-semigroups.
	In the Hilbert space setting,
	we present two characterizations of the Crandall--Pazy class.
	We then apply
	these characterizations to estimate decay rates for
	Crank--Nicolson schemes with smooth initial data
	when
	the associated abstract Cauchy problem is governed by 
	an exponentially stable $C_0$-semigroup in the Crandall--Pazy class.
	The first approach is based on a functional 
	calculus called the $\mathcal{B}$-calculus.
	The second approach builds upon estimates derived from Lyapunov equations and improves the decay estimate
	obtained in the first approach,  under the additional assumption that
	$-A^{-1}$ generates a bounded $C_0$-semigroup.
\end{abstract}


\maketitle

	\section{Introduction}
Consider the abstract Cauchy problem
\begin{equation}
	\label{eq:ACP}
	\begin{cases}
		\dot u (t) = -Au(t), & t \geq 0, \\
		u(0) = x,& x \in X,
	\end{cases}
\end{equation}
where $-A$ is the generator of a $C_0$-semigroup
$(e^{-tA})_{t \geq 0}$ on a Banach space $X$.
We say that $(e^{-tA})_{t \geq 0}$ is {\em 
	immediately differentiable} if
the orbit map $t \mapsto u(t) = e^{-tA}x$ is differentiable on $(0,\infty)$
for all $x \in X$.
Recall that $(e^{-tA})_{t \geq 0}$ is holomorphic if
and only if 
$\|Ae^{-tA}\| = O(t^{-1})$ as $t \to 0+$, i.e.,
there exist constants $M,t_0>0$ such that
$\|Ae^{-tA}\| \leq Mt^{-1}$ for all $0 < t < t_0$;
see, e.g., \cite[Theorem~3.7.19]{Arendt2001} and 
\cite[Theorem~II.4.6]{Engel2000}.
Crandall and Pazy \cite{Crandall1969} introduced 
the following class of $C_0$-semigroups, which lie
between immediately differentiable $C_0$-semigroups
and holomorphic $C_0$-semigroups.
\begin{definition}
		A
		$C_0$-semigroup $(e^{-tA})_{t \geq 0}$
		on a Banach space $X$
		is in the 
		{\em Crandall--Pazy class with parameter $\beta \in (0,1]$}
		if
		$(e^{-tA})_{t \geq 0}$ is immediately differentiable and 
		\begin{equation}
			\label{eq:CP_cond}
			\big\|Ae^{-tA}\big\| = O\left(
			\frac{1}{
				t^{1/\beta}}\right)\quad \text{as $t \to 0+$}.
		\end{equation}
\end{definition}
The Crandall--Pazy class with some parameter $\beta \in(0,1]$
is simply referred to as  the {\em Crandall--Pazy class}.
Note that $C_0$-semigroups in the Crandall--Pazy class with parameter
$\beta = 1$ corresponds to holomorphic $C_0$-semigroups.
We encounter
the Crandall--Pazy class of $C_0$-semigroups when examining the  differentiability of mild solutions of 
inhomogeneous Cauchy problems \cite{Crandall1969} and delay differential equations \cite{Batty2004,Batty2007}.
This class also arises in a necessary condition
for the existence and uniqueness of classical solutions of 
semilinear Cauchy problems \cite{Eberhardt1994}.

If the $C_0$-semigroup
$(e^{-tA})_{t \geq 0}$  on a Banach space is in the Crandall--Pazy class with parameter $\beta \in (0,1]$,
then the spectrum
$\sigma(A)$ of $A$ is contained in
the set $\{ z \in \mathbb{C} : \re z \geq -a + b |\im z|^{\beta} \}$
for some $a,b>0$ and
\begin{equation}
	\label{eq:CP_resolvent}
	\| (i\eta + A)^{-1}\| = O\left(
	\frac{1}{
		|\eta|^{\beta}}\right)\quad \text{as $|\eta| \to \infty$};
\end{equation}
see \cite{Crandall1969}.
Conversely,	it was also shown in \cite{Crandall1969} that 
the estimate \eqref{eq:CP_resolvent} for some 
$\beta \in (0,1]$ implies 
$\|Ae^{-tA}\| = O(t^{-2\beta+1})$ as $t \to 0+$.
This estimate was improved  to
$\|Ae^{-tA}\| = O(t^{-\beta + \varepsilon})$ as $t \to 0+$
for all $\varepsilon>0$ in \cite{Eberhardt1994}.
These results provide a characterization of the Crandall--Pazy
class without specifying the parameter $\beta$, 
which was applied to
perturbation theory for $C_0$-semigroups in
\cite{Arendt2006,Chen2020JEE}. 
Another characterization based on
resolvents 
can be found in \cite{Iley2007}.

In this paper, we focus on $C_0$-semigroups on Hilbert spaces.
Our first main 
result is stated as follows in the simplified case
where the $C_0$-semigroup $(e^{-tA})_{t \geq 0}$ is bounded, i.e., $\sup_{t \geq 0}\|e^{-tA}\| < \infty$; 
see Theorem~\ref{thm:decay_growth}
for the case where $(e^{-tA})_{t \geq 0}$ is not bounded.
\begin{theorem}
	\label{thm:simplified}
	Let $-A$ be the generator of a bounded $C_0$-semigroup $(e^{-tA})_{t \geq 0}$ on a Hilbert space $H$ such that 
	$\sigma(A) \cap i \mathbb{R} = \emptyset$.
	Then the following statements are equivalent for a fixed $\beta \in (0,1]$:
	\begin{enumerate}
		\renewcommand{\labelenumi}{(\roman{enumi})}
		\item $\|Ae^{-tA}\| = O(
		t^{-1/\beta})$ as $t \to 0+$.
		\item $\| (i\eta + A)^{-1}\| = O(
		|\eta|^{-\beta})$ as $|\eta| \to \infty$.
		\item For each $q \in (0,1/2)$, there exists 
		a constant $M>0$ such that for all $x \in H$,
		\begin{equation*}
			\sup_{\xi > 1} \xi^{1-2q}
			\int_{-\infty}^{\infty}
			\|A^{\beta q}(\xi + i \eta +A)^{-1}x \|^2 d\eta \leq M \|x\|^2.
		\end{equation*}
	\end{enumerate}
\end{theorem}
In the Hilbert space setting, the implication
(ii) $\Rightarrow$ (i) improves the existing estimates mentioned above.
On the other hand, statement 
(iii) is in a form similar to the characterization for
bounded $C_0$-semigroups established in \cite{Gomilko1999,Shi2000}. 
Theorem~\ref{thm:simplified} 
is analogous to the following characterizations of
bounded $C_0$-semigroups with polynomial decay
on Hilbert spaces (see \cite{Borichev2010,Wakaiki2024JEE}), and the  techniques developed for
such $C_0$-semigroups in \cite{Borichev2010,Batty2016,Wakaiki2024JEE} are employed 
in the proof of Theorem \ref{thm:simplified}.
\begin{theorem}
	\label{thm:polynomial_decay}
	Let $-A$ be the generator of a bounded $C_0$-semigroup $(e^{-tA})_{t \geq 0}$ on a Hilbert space $H$ such that 
	$\sigma(A) \cap i \mathbb{R} = \emptyset$.
	Then the following statements are equivalent for a fixed $\beta >0$:
	\begin{enumerate}
		\renewcommand{\labelenumi}{(\roman{enumi})}
		\item $\|e^{-tA}A^{-1}\| = O(
		t^{-1/\beta})$ as $t \to \infty$.
		\item $\| (i\eta + A)^{-1}\| = O(
		|\eta|^{\beta})$  as $|\eta| \to \infty$.
		\item For each $q \in (0,1/2)$, there exists 
		a constant $M>0$ such that for all $x \in H$,
		\begin{equation*}
			\sup_{0 < \xi < 1} \xi^{1-2q}
			\int_{-\infty}^{\infty}
			\|(\xi + i \eta +A)^{-1}A^{-\beta q}x \|^2 d\eta \leq M \|x\|^2.
		\end{equation*}
	\end{enumerate}
\end{theorem}

As an application of the above characterizations of the 
Crandall--Pazy class of $C_0$-semigroups, we study
the transference of decay rates in time-discretization of 
the abstract Cauchy problem~\eqref{eq:ACP}.
For time-discretization, we consider the following approximation 
with a fixed time-step $\tau >0$:
\begin{equation}
	\dot u(t) \approx \frac{u(t+\tau)- u(t)}{\tau}\quad \text{and} \quad 
	u(t) \approx \frac{u(t+\tau)+ u(t)}{2}.
\end{equation}
Then the discrete
approximation $u_{\textrm d}(n)$ of $u(n \tau)$ 
satisfies the difference equation
\[
\begin{cases}
	u_{\textrm d}(n+1)= -V_{2/\tau} (A) u_{\textrm d}(n), & n \in \mathbb{N}_0
	\coloneqq \{0,1,2,\dots \}, \\
	u_{\textrm d}(0) = x,& x \in X,
\end{cases}
\]
where 
\[
V_{2/\tau} (A) \coloneqq 
\left(A - 
\frac{2}{\tau}
\right)\left(
A + 
\frac{2}{\tau} 
\right)^{-1}.
\]
This time-discretization is known as the
{\em Crank--Nicolson scheme}.
The time-step $\tau$ is constant in the above setting for simplicity,
but it may vary depending on $n$. 
The operator $V_{2/\tau} (A) $, called 
the {\em Cayley transform of $A$}, is also widely used in system
theory for the transformation of
continuous-time systems into
discrete-time systems; see, e.g., \cite[Section~12.3]{Staffans2005}.
Let the 
$C_0$-semigroup $(e^{-tA})_{t \geq 0}$ on a Hilbert space 
be exponentially stable, i.e., $\lim_{t\to \infty}e^{ct}\|e^{-tA}\| =0$
for some $c >0$.
Our second goal is to estimate the rate of decay of
$u_{\textrm d}$ 
for smooth initial data $x$ in the domain $D(A)$ of $A$,
assuming that $(e^{-tA})_{t \geq 0}$ is  in
the Crandall--Pazy class.
To this end, we examine the quantitative behavior of 
the operator norm $\|(\prod_{k=1}^n
V_{\omega_k}(A)) A^{-1}\|$, where $0 < \omega_{\min}
\leq \omega_k \leq \omega_{\max} < \infty$ for $k \in \mathbb{N}$. 

There is a substantial body of literature on
the asymptotic behavior of the approximate solution 
$u_{\textrm d}$.
First, we briefly review known results for the constant case $
\omega_k \equiv 1$
and refer to the survey article \cite{Gomilko2017}
and the book \cite[Chapter~5]{Eisner2010}
for further information. Let $V(A) \coloneqq V_1(A)$.
It is clear that if $X$ is a finite-dimensional space, then 
boundedness is preserved, namely,
$\sup_{t \geq 0} \|e^{-tA}\| < \infty$ implies
$\sup_{n \in \mathbb{N}_0} \|V(A)^n\| < \infty$.
This preservation also holds for
contraction $C_0$-semigroups on Hilbert spaces \cite{Phillips1959} and
sectorially bounded 
holomorphic $C_0$-semigroups on Banach spaces \cite{Crouzeix1993}.
However, boundedness is not preserved in general
when
only the assumption of bounded or even exponentially stable
$C_0$-semigroups on Banach spaces is made; see, e.g., \cite{Brenner1979,Esiner2008JEE, Gomilko2011, Piskarev2007}.

For a bounded $C_0$-semigroup $(e^{-tA})_{t \geq 0}$ on a Hilbert space $H$,
it remains unknown
whether boundedness is preserved, but
the following estimates are known:
$\|V(A)^n\| = O(\log n)$ as $n \to \infty$ 
(see \cite{Gomilko2004}); and   
$\sup_{n \in \mathbb{N}_0} \|V(A)^n (1+A)^{-\alpha}\| < \infty$ for all $\alpha >0$ (see \cite[Lemma~2.3]{Gomilko2024}).
If we further assume that 
$-A^{-1}$ generates a bounded $C_0$-semigroup, 
then $\sup_{t \geq 0} \|e^{-tA}\| < \infty$ implies
$\sup_{n \in \mathbb{N}_0} \|V(A)^n\| < \infty$, which
was proved independently in \cite{Azizov2004,Gomilko2004,Guo2006}.
Under this assumption on $-A^{-1}$, it was  shown in 
\cite{Guo2006} that strong stability is preserved, i.e.,
if $e^{-tA}x \to 0$ as $t \to \infty$ for all $x \in H$, then
$V(A)^nx \to 0$ as $n \to \infty$ for all $x \in H$.
The semigroup-generation property of $-A^{-1}$ 
was also used in \cite{Wakaiki2021JEE}
to prove 
that if the bounded $C_0$-semigroup $(e^{-tA})_{t \geq 0}$ on a Hilbert space satisfies 
$\|e^{-tA}A^{-1}\| = O(
t^{-1/\beta})$ as $t \to \infty$ for some $\beta >0$, then
$\|V(A)^nA^{-1}\| = O(n^{-1/(2+\beta)} (\log n)^{1/(2+\beta)})$ as $n \to \infty$.
This estimate was refined  to 
$\|V(A)^nA^{-1}\| = O(n^{-1/(2+\beta)})$ as $n \to \infty$ in \cite{Pritchard2024}.
Moreover, without assuming that $-A^{-1}$ generates a bounded $C_0$-semigroup, the estimate
$\|V(A)^nA^{-1}\| = O(n^{-1/(2+\beta)} \log n)$ as $n \to \infty$
was obtained in \cite{Wakaiki2024JEE}.

Some of the estimates above were extended to the variable case $(\omega_k)_{k\in\mathbb{N}_0}$, where $0 < \omega_{\min} \leq \omega_k \leq 
\omega_{\max} < \infty$ for all $k \in \mathbb{N}$. 
If  $(e^{-tA})_{t \geq 0}$ is a sectorially bounded 
holomorphic $C_0$-semigroup on a Banach space, then
$\sup_{n \in \mathbb{N}_0} \|(\prod_{k=1}^n
V_{\omega_k}(A))\| < \infty$; see \cite{Palencia1993,Bakaev1995, Piskarev2007, Casteren2011, Batty2025}.
Boundedness is also preserved
if both $-A$ and $-A^{-1}$ generate a bounded $C_0$-semigroup on a Hilbert space $H$, as proven in \cite{Piskarev2007}.
Moreover, it was shown in \cite{Piskarev2007} that 
if $e^{-tA}x \to 0$ and $e^{-tA^{-1}}x \to 0$ 
as $t \to \infty$ for all $x \in H$, then
$(\prod_{k=1}^n
V_{\omega_k}(A))x \to 0$ as $n \to \infty$ for all $x \in H$.
For 
every exponentially stable $C_0$-semigroup  $(e^{-tA})_{t \geq 0}$
on a Hilbert space, the estimate 
\begin{equation}
	\label{eq:exp_case}
	\left\|\left(\prod_{k=1}^n
	V_{\omega_k}(A)\right) A^{-1}
	\right\| =  O
	\left( \frac{1}{n^{1/2}} \right)\quad \text{as $n \to \infty$}
\end{equation}
was derived in \cite{Wakaiki2024JEE}.
It was also proved in \cite{Wakaiki2024JEE} that 
if the negative generator $A$ of a bounded $C_0$-semigroup $(e^{-tA})_{t \geq 0}$ on a Hilbert space is a normal operator and satisfies
$\|e^{-tA}A^{-1}\| = O(
t^{-1/\beta})$ as $t \to \infty$ for some $\beta >0$,
then $\|(\prod_{k=1}^n
V_{\omega_k}(A)) A^{-1}\| = O(n^{-1/(2+\beta)})$ as $n \to \infty$.

It is natural to ask whether additional properties can
improve the decay estimate given in \eqref{eq:exp_case}
for exponentially stable $C_0$-semigroups on Hilbert spaces.
This question motivates the application of the characterizations
of the Crandall--Pazy class.
In this paper, 
we begin by briefly examining exponentially stable holomorphic $C_0$-semigroups
on Banach spaces and obtain the estimate
\begin{equation}
	\label{eq:holomorphic_case}
	\left\|\left(\prod_{k=1}^n
	V_{\omega_k}(A)\right) A^{-1}
	\right\| = O
	\left( \frac{1}{n} \right)\quad \text{as $n \to \infty$}.
\end{equation}
When comparing the estimates \eqref{eq:exp_case} and 
\eqref{eq:holomorphic_case}, there is a gap
between the decay rates $n^{-1/2}$ and $n^{-1}$.
We show that 
this gap is bridged by
the Crandall--Pazy class in the Hilbert space setting.
Specifically,
if the
exponentially stable $C_0$-semigroup 
$(e^{-tA})_{t \geq 0}$ on a Hilbert space $H$ is in
the Crandall--Pazy class with parameter $\beta \in (0,1]$, then
for all $\varepsilon >0$,
\begin{equation}
	\label{eq:CP_case1}
	\left\|\left(\prod_{k=1}^n
	V_{\omega_k}(A)\right) A^{-1}
	\right\| = O
	\left( \frac{(\log n)^{2/(2-\beta) + \varepsilon}}{n^{1/(2-\beta)}} \right)\quad \text{as $n \to \infty$}.
\end{equation}
It is still open whether the logarithmic term in the estimate
\eqref{eq:CP_case1} may be dropped in general.
However, we show that if 
$-A^{-1}$ generates
a bounded $C_0$-semigroup on $H$, then
\begin{equation}
	\label{eq:CP_case2}
	\left\|\left(\prod_{k=1}^n
	V_{\omega_k}(A)\right) A^{-1}
	\right\| = O
	\left( \frac{1}{n^{1/(2-\beta)}} \right)\quad \text{as $n \to \infty$}.
\end{equation}
We obtain the first estimate 
\eqref{eq:CP_case1} through the implication 
(i) $\Rightarrow$ (iii) in Theorem~\ref{thm:simplified} and 
the $\mathcal{B}$-calculus established in \cite{Batty2021,Batty2021JFA}.
On the other hand,
the second estimate 
\eqref{eq:CP_case2} is proved by combining the implication 
(i) $\Rightarrow$ (ii) in Theorem~\ref{thm:simplified} with
the estimate  derived from Lyapunov equations in \cite{Piskarev2007}.
Conversely, if $-A$ is the generator of a bounded $C_0$-semigroup
on a Hilbert space such that $\sigma(A) \cap i \mathbb{R} =\emptyset$ and if $\sup_{n \in \mathbb{N}_0} \|V(A)^n\|< \infty$, then
$\|V(A)^n A^{-1}\| = O(n^{-1/(2-\beta)})$ as $n \to \infty$
for some $\beta \in (0,1]$ implies $\|Ae^{-tA}\| = 
O(t^{-1/\beta})$ as $t \to 0+$.
This converse statement is obtained from 
the implication 
(ii) $\Rightarrow$ (i) in Theorem~\ref{thm:simplified}
and is analogous to the one obtained in \cite{Pritchard2024}
for bounded $C_0$-semigroups with polynomial decay
on Hilbert spaces.

This paper is organized as follows.
In Section~\ref{sec:charactrization}, 
we give characterizations of the Crandall--Pazy class of $C_0$-semigroups
on Hilbert spaces.
In Section~\ref{sec:Decay_estimate1},
we first develop a new operator-norm estimate for the $\mathcal{B}$-calculus,
which is tailored to the Crandall--Pazy class.
Next, we estimate the rate of decay for the Crank--Nicolson scheme
with smooth initial data.
In Section~\ref{sec:Decay_estimate2},
we assume that $-A^{-1}$ generates a bounded $C_0$-semigroup.
Then we 
improve the decay estimate given in Section~\ref{sec:Decay_estimate1}
and discuss transference from
decay estimates for $\|V(A)^nA^{-1}\|$ as $n \to \infty$ to
growth estimates for $\|Ae^{-tA}\|$ as $t\to 0+$.

\subsubsection*{Notation}

Let $\mathbb{N}$ and $\mathbb{N}_0$ denote 
the set of positive integers and the set of nonnegative integers,
respectively.
We write $\mathbb{C}_+ \coloneqq \{\lambda \in \mathbb{C}:
\re \lambda >0 \}$ and $\mathbb{R}_+ \coloneqq [0,\infty)$.
Let $\Sigma_{\theta} \coloneqq 
\{ z \in \mathbb{C} \setminus \{ 0\}: 
|\arg z| < \theta\}$ for $0 < \theta < \pi$
and let $\Sigma_0 \coloneqq (0,\infty)$.
For a subset $\Omega$ of 
$\mathbb{C}$, we denote by $\overline \Omega$
the closure of $\Omega$.
Given functions $f,g \colon [a,\infty) \to (0,\infty)$ for some $a \geq 0$, 
we write 
$f(t) = O(g(t))$ 
as $t \to \infty$
if there exist constants $M>0$ and $t_0 \geq a$ such that 
$f(t) \leq Mg(t)$ for all $t \geq t_0$.
Analogous notation is used for other variants of asymptotic behavior.
For $0 < \omega_{\min} \leq \omega_{\max} < \infty$,
let $\mathcal{S}(\omega_{\min},\omega_{\max})$ denote
the set of sequences $(\omega_k)_{k \in \mathbb{N}}$
of positive real numbers satisfying $\omega_{\min} \leq 
\omega_k \leq \omega_{\max}$ for all $k \in \mathbb{N}$.
The 
gamma function is denoted  by $\Gamma$.

Throughout this paper, all Banach spaces
are assumed to be complex.
Let $X$ be a Banach space. We write $\mathcal{L}(X)$ for
the Banach algebra of bounded linear operators on $X$.
Let $A$ be a linear operator on $X$. We denote 
the domain of $A$ by $D(A)$,
the spectrum of $A$ by
$\sigma(A)$, and 
the resolvent set of $A$ by
$\varrho(A)$.
For $\omega \in \varrho(-A) \cap (0,\infty)$, we define 
the Cayley transform $V_{\omega}(A)$ of $A$
with parameter $\omega$ by
$V_{\omega}(A) \coloneqq (A-\omega)(A+\omega)^{-1}$.
Let $H$ be a Hilbert space.
We denote the inner product on $H$ by
$\langle \cdot , \cdot \rangle$.
If $A$ is a densely defined linear operator on $H$, then
the Hilbert space adjoint of $A$ is denoted by $A^*$.

\section{Characterizations of Crandall--Pazy class of semigroups}
\label{sec:charactrization}
In this section, 
we start by recalling the notion of sectorial operators and 
the moment inequality.
We then present our first main result in 
Theorem~\ref{thm:decay_growth}, which provides two characterizations of 
the Crandall--Pazy class of $C_0$-semigroups on
Hilbert spaces. 
Both characterizations are used 
to estimate the rate of decay for
the Crank--Nicolson scheme in the subsequent sections.
\subsection{Moment inequality}
A densely defined closed operator $A$ on a Banach space $X$ is called 
{\em sectorial with angle $\theta \in [0,\pi)$} if 
$\sigma(A) \subset \overline{\Sigma_{\theta}}$ and
for each $\theta' \in (\theta,\pi)$,
\[
\sup\Big\{
\|\lambda (\lambda-A)^{-1}\| : 
\lambda \in \mathbb{C} \setminus \overline{\Sigma_{\theta'}}
\Big\} < \infty.
\]
An operator $A$ is simply called {\em sectorial} if it is sectorial
of angle $\theta$ for some $\theta \in [0,\pi)$.
For $\theta \in [0,\pi)$,
we denote by $\Sect(\theta)$ the set of
sectorial operators of angle $\theta$ on a Banach space $X$.
Define 
\[
\Sect(\pi/2-) \coloneqq \bigcup_{0\leq \theta < \pi/2} \Sect(\theta).
\]
For a sectorial operator $A$ and a constant $\alpha>0$, 
the fractional power $A^{\alpha}$
of $A$ is defined by the sectorial functional calculus; see \cite[Chapter~3]{Haase2006}.

We recall the moment inequality for sectorial operators; see, e.g., \cite[Proposition~6.6.4]{Haase2006}
for the proof.
\begin{lemma}
	\label{lem:moment_ineq}
	Let $A$ be a sectorial operator on a Banach space $X$, and let $0 \leq \alpha < \beta < \gamma$. There exists a constant $C>0$ such that 
	\[
	\|A^{\beta}x\| \leq C\|A^{\alpha} x\|^{(\gamma - \beta)/(\gamma - \alpha)}
	\|A^{\gamma} x\|^{(\beta - \alpha)/(\gamma - \alpha)}
	\]
	for all $x \in D(A^{\gamma})$.
\end{lemma}

Let $-A$ be the generator of
a bounded $C_0$-semigroup $(e^{-tA})_{t \geq 0}$ on a Banach space $X$, and
let $\alpha,t >0$.
Then $A$ is sectorial of angle $\pi/2$.
Since the fractional power $A^{\alpha}$ is a closed operator,
the closed graph theorem  shows that if
$e^{-tA} X \subset D(A^{\alpha})$, then $A^{\alpha} e^{-tA} \in \mathcal{L}(X)$.
The following lemma on fractional powers 
can be proved in the same way as \cite[Lemma~4.2]{Batty2016}
and \cite[Theorem~2, $(\mathrm{a}_{\alpha}) \Leftrightarrow
(\mathrm{b}_{\alpha})
$]{Eberhardt1994},
but to make the paper self-contained, we give a short argument.
\begin{lemma}
	\label{lem:domain_decay}
	Let $-A$ be the generator of a bounded
	$C_0$-semigroup $(e^{-tA})_{t \geq 0}$ on a Banach space $X$.
	Let $\alpha >0$ and assume that $e^{-tA}X \subset D(A^{\alpha})$
	for all $t>0$. Then 
	the following statements hold for all $\beta >0$:
	\begin{enumerate}
		\renewcommand{\labelenumi}{\alph{enumi})}
		\item $e^{-tA}X \subset D(A^\beta)$ for all $t>0$.
		\item There exist constants $C_1,C_2,c_1,c_2>0$ such that for all $t>0$,
		\[
		C_1 \big\|A^{\alpha} e^{-c_1tA}\big\|^{\beta} \leq 
		\big\|A^{\beta} e^{-tA}\big\|^{\alpha} \leq 
		C_2 \big\|A^{\alpha} e^{-c_2tA}\big\|^{\beta}.
		\]
	\end{enumerate}
\end{lemma}
\begin{proof}
	Proof of a):
	Let $n \in \mathbb{N}$ and $\delta \in [0,\alpha)$ satisfy 
	$\beta = n \alpha + \delta$.
	Let $x \in X$ and $t>0$. We have
	\[
	e^{-tA} x = \big(e^{-t A/(n+1)}\big)^n e^{-t A/(n+1)}x.
	\]
	Since
	\[
	e^{-t A/(n+1)} x \in D(A^{\alpha}) \subset D(A^{\delta}),
	\]
	it follows that $e^{-tA}x \in D(A^{n\alpha + \delta}) = D(A^{\beta})$.
	
	Proof of b):
	Let $n \in \mathbb{N}$ satisfy $n\alpha \geq \beta$, and let
	$K \coloneqq \sup_{t \geq 0} \|e^{-tA}\|$.
	By the moment inequality, there exists $C>0$ such that
	for all $x \in X$,
	\begin{align*}
		\big\|A^{\beta /n} e^{-tA/n}x\big\| &\leq C 
		\big\|e^{-tA/n}x\big\|^{1-\beta/(n\alpha)} \big\|A^{\alpha}e^{-tA/n}x\big\|^{\beta/(n\alpha)} \\
		&\leq C K^{1-\beta/(n\alpha)} \big\|A^{\alpha}e^{-tA/n}\big\|^{\beta/(n\alpha)}  \|x\|.
	\end{align*} 
	This implies that 
	\[
	\big\|A^{\beta} e^{-tA}\big\| =
	\big\|\big(A^{\beta /n} e^{-tA/n} \big)^n\big\| \leq 
	\left(C K^{1-\beta/(n\alpha)} \big\|A^{\alpha}e^{-tA/n}\big\|^{\beta/(n\alpha)}\right)^n.
	\]
	Hence,
	\[
	\big\|A^{\beta} e^{-tA}\big\|^{\alpha} \leq \left(C K^{1-\beta/(n\alpha)} \big\|A^{\alpha}e^{-tA/n}\big\|^{\beta/(n\alpha)}\right)^{n\alpha} \leq 
	\left(C^{n\alpha} K^{n\alpha-\beta}\right) \big\|A^{\alpha}e^{-tA/n}\big\|^{\beta}.
	\]
	Interchanging $\alpha$ and $\beta$, we obtain the other inequality.
\end{proof}

\subsection{Crandall--Pazy class of semigroups}
The following result characterizes 
the Crandall--Pazy class of $C_0$-semigroups on Hilbert spaces.
\begin{theorem}
	\label{thm:decay_growth}
	Let $-A$ be the generator of a  $C_0$-semigroup $(e^{-tA})_{t \geq 0}$ 
	on a Hilbert space $H$, and let $\omega_A \in \mathbb{R}$ be the exponential 
	growth bound of $(e^{-tA})_{t \geq 0}$.
	Then the following statements are equivalent for a fixed $\beta \in (0,1]$:
	\begin{enumerate}
		\renewcommand{\labelenumi}{(\roman{enumi})}
		\item 
		$(e^{-tA})_{t \geq 0}$ is in the Crandall--Pazy class with parameter $\beta$.
		\item 
		$\sigma(A) \cap i \mathbb{R}$ is bounded 
		and $
		\|(i\eta + A)^{-1}\| = O(|\eta|^{-\beta})
		$
		as $|\eta| \to \infty$.
		\item
		For each $c > \omega_A$ and $q \in (0,1/2)$, there
		exists a constant $M>0$ such that for all $x \in H$,
		\begin{equation}
			\label{eq:resol_int_unbounded}
			\sup_{\xi>c} \xi^{1-2q}
			\int_{-\infty}^{\infty}
			\|(A+c)^{\beta q}(\xi + i \eta +A)^{-1}x \|^2 d\eta \leq M \|x\|^2.
		\end{equation}
		\item[(iii)']
		There
		exist constants $c > \omega_A$, $q \in (0,1/2)$, and $M>0$ 
		such that the estimate \eqref{eq:resol_int_unbounded}
		holds for all $x \in H$.
	\end{enumerate}
\end{theorem}

The proof of 
the equivalence of (i) and (iii) in Theorem~\ref{thm:decay_growth}
is based on the following lemma.
\begin{lemma}
	\label{lem:decay2resol_int}
	Let $-A$ be the generator of a bounded $C_0$-semigroup $(e^{-tA})_{t \geq 0}$ 
	on a Hilbert space $H$. Then
	the following statements hold for a fixed $q \in (0,1/2)$:
	\begin{enumerate}
		\renewcommand{\labelenumi}{\alph{enumi})}
		\item
		If 
		$(e^{-tA})_{t \geq 0}$ is in the Crandall--Pazy class with parameter $\beta \in (0,1]$,
		then for each $\xi_0>0$,
		there exists a constant $M>0$ such that for all $x \in H$,
		\begin{equation}
			\label{eq:resol_int_small}
			\sup_{0 < \xi \leq \xi_0} \xi
			\int_{-\infty}^{\infty}
			\|A^{\beta q}(\xi + i \eta +A)^{-1}x \|^2 d\eta \leq M \|x\|^2
		\end{equation}
		and
		\begin{equation}
			\label{eq:resol_int_large}
			\sup_{\xi > \xi_0} \xi^{1-2q}
			\int_{-\infty}^{\infty}
			\|A^{\beta q}(\xi + i \eta +A)^{-1}x \|^2 d\eta \leq M \|x\|^2.
		\end{equation}
		\item
		If there exist constants 
		$\beta \in (0,1]$ and $\xi_0, M>0$ such that 
		the estimate \eqref{eq:resol_int_large} holds
		for all $x \in H$, then $(e^{-tA})_{t \geq 0}$ is in the Crandall--Pazy class with parameter $\beta$.
	\end{enumerate}
\end{lemma}
\begin{proof}
	Let $q \in (0,1/2)$.
	By Lemma~\ref{lem:domain_decay}, the bounded $C_0$-semigroup
	$(e^{-tA})_{t \geq 0}$ is in the Crandall--Pazy class with parameter $\beta \in (0,1]$ if and only if $e^{-tA}H \subset D(A^{\beta q})$ for all $t>0$ and
	there exist $M_0,t_0 >0$ such that 
	for all $t \in (0, t_0]$,
	\begin{equation}
		\label{eq:AT_bound1}
		\big\|A^{\beta q}e^{-tA}\big\| \leq \frac{M_0}{t^{q}}.
	\end{equation}
	
	Proof of a).
	Let $K \coloneqq \sup_{t\geq 0} \|e^{-tA}\|$.
	Using \eqref{eq:AT_bound1}, we have that 
	for all $t>t_0$, 
	\begin{equation}
		\label{eq:AT_bound2}
		\big\|A^{\beta q}e^{-tA}\big\| \leq \big\|A^{\beta q}e^{-t_0A}\big\| \, \big\|e^{-(t-t_0)A}\big\| \leq 
		\frac{M_0 K}{t_0^q}.
	\end{equation}
	Let $x \in H$.
	For $\xi>0$, we define $f_{\xi}\colon \mathbb{R} \to H$ by
	\[
	f_{\xi}(t) \coloneqq 
	\begin{cases}
		e^{-\xi t}A^{\beta q}e^{-tA}x, & t > 0,\\
		0,& t \leq 0.
	\end{cases}
	\]
	From the estimates \eqref{eq:AT_bound1} and \eqref{eq:AT_bound2},
	we see that $f_{\xi} \in L^1(\mathbb{R},H) \cap L^2(\mathbb{R},H)$ for all $\xi>0$.
	For each $\xi >0$,
	the Fourier transform $\mathcal{F}[f_{\xi}]$ of $f_{\xi}$ is given by
	\[
	\mathcal{F}[f_{\xi}](\eta) =
	\int_0^{\infty} e^{-(\xi+i\eta)t}A^{\beta q }e^{-tA}x dt =
	A^{\beta q}(\xi+i\eta+A)^{-1}x,\quad \eta \in \mathbb{R}.
	\]
	
	By the Plancherel theorem,
	\begin{equation}
		\label{eq:Abq_Plancherel}
		\int_{-\infty}^{\infty}
		\|A^{\beta q}(\xi + i \eta +A)^{-1}x \|^2 d\eta = 2\pi
		\int_0^{\infty} e^{-2\xi t} \big\|A^{\beta q} e^{-tA}x\big\|^2 dt
	\end{equation}
	for all $\xi >0$.
	Using the estimates \eqref{eq:AT_bound1} and \eqref{eq:AT_bound2}, 
	we obtain
	\begin{align*}
		\int_0^{\infty} e^{-2\xi t} \big\|A^{\beta q} e^{-tA}x\big\|^2 dt 
		&=
		\int_0^{t_0} e^{-2\xi t} \big\|A^{\beta q} e^{-tA}x\big\|^2 dt +
		\int_{t_0}^{\infty} e^{-2\xi t} \big\|A^{\beta q} e^{-tA} x\big\|^2 dt \\
		&\leq
		M_0^2 \|x\|^2 \int_0^{t_0}  \frac{e^{-2\xi t} }{t^{2q }}dt +
		\frac{M_0^2K^2}{t_0^{2q}} \|x\|^2 \int_{t_0}^{\infty}  e^{-2\xi t} dt.
	\end{align*}
	By the definition of the gamma function $\Gamma$,
	\[
	\int_0^{t_0}  \frac{e^{-2\xi t} }{t^{2q }}dt \leq 
	\int_0^{\infty}  \frac{e^{-2\xi t} }{t^{2q }}dt 
	=\frac{\Gamma(1-2q)}{(2\xi)^{1-2q}},
	\]
	and hence
	\begin{align*}
		\int_0^{\infty} e^{-2\xi t} \big\|A^{\beta q} e^{-tA}x\big\|^2 dt 
		&\leq
		\frac{M_0^2}{2}
		\left(
		\frac{2^{2q}\Gamma(1-2q)}{\xi^{1-2q}} + 
		\frac{K^2}{t_0^{2q}\xi}  \right)
		\|x\|^2 
	\end{align*}
	for all $\xi >0$. 
	Let $\xi_0>0$ be given. Then
	\[
	\xi\int_{-\infty}^{\infty}
	\|A^{\beta q}(\xi + i \eta +A)^{-1}x \|^2 d\eta \leq 
	\pi M_0^2
	\left(
	(2\xi_0)^{2q}\Gamma(1-2q)+ 
	\frac{K^2}{t_0^{2q}}  \right) \|x\|^2 
	\]
	for all $\xi \in (0,\xi_0]$,	
	and
	\[
	\xi^{1-2q}\int_{-\infty}^{\infty}
	\|A^{\beta q}(\xi + i \eta +A)^{-1}x \|^2 d\eta \leq 
	\pi M_0^2
	\left(2^{2q} 
	\Gamma(1-2q)+ 
	\frac{K^2}{(t_0\xi_0)^{2q}}  \right) \|x\|^2 
	\]
	for all $\xi >\xi_0$.	
	
	Proof of b).
	By assumption,
	there exist $\beta \in (0,1]$ and $\xi_0, M>0$ such that 
	\begin{equation}
		\label{eq:resol_int_bound_CP}
		\int_{-\infty}^{\infty}
		\|A^{\beta q}
		(\xi+i \eta +A)^{-1}x \|^2 d\eta \leq \frac{M}{\xi^{1-2q}} \|x\|^2
	\end{equation}
	for all $x \in H$ and $\xi > \xi_0$. 
	The Plancherel theorem shows that 
	\begin{equation}
		\label{eq:resol_int_bound_bounded}
		\int_{-\infty}^{\infty}
		\|
		(\xi-i \eta +A^*)^{-1}y \|^2 d\eta = 
		2\pi 
		\int^{\infty}_0 e^{-2\xi t} \big \| e^{-tA^*} y  \big\|^2 dt
		\leq \frac{\pi K^2}{\xi} \|y\|^2
	\end{equation}
	for all $y \in H$ and $\xi > 0$, where $K \coloneqq \sup_{t\geq 0} \|e^{-tA}\|$.
	
	Let $x \in D(A^{2+\beta q})$ and $y \in H$.
	The inversion formula (see, e.g., \cite[Corollary~III.5.16]{Engel2000}
	and \cite[Corollary~8.14]{Batty2021JFA}) shows that for all
	$t,\xi >0$,
	\[
	\big|\big\langle
	e^{-tA} A^{\beta q}x,y
	\big\rangle\big|
	\leq 
	\frac{e^{\xi t}}{2\pi t}
	\int_{-\infty}^{\infty}
	|
	\langle
	(\xi+i \eta +A)^{-2}A^{\beta q}x,y
	\rangle| d\eta.
	\]
	By the Cauchy-Schwartz inequality,
	\begin{align*}
		&\big|\big\langle
		A^{\beta q}e^{-tA} x,y
		\big\rangle\big| \leq 
		\frac{e^{\xi t}}{2\pi t}
		\left(
		\int_{-\infty}^{\infty}
		\|A^{\beta q}
		(\xi+i \eta +A)^{-1}x \|^2 d\eta \right)^{1/2}
		\left(
		\int_{-\infty}^{\infty}
		\|(\xi-i \eta +A^*)^{-1}y \|^2 d\eta\right)^{1/2}
	\end{align*}
	for all
	$t,\xi >0$.
	Combining this with the estimates \eqref{eq:resol_int_bound_CP}
	and \eqref{eq:resol_int_bound_bounded}, we obtain
	\begin{equation}
		\label{eq:inner_product_bound}
		\big|\big\langle
		A^{\beta q}e^{-tA} x,y
		\big\rangle\big|
		\leq 
		\frac{M_1e^{\xi t}}{\xi^{1-q} \, t}\|x\|\, \|y\|,\quad 
		\text{where~}M_1 \coloneqq 
		\frac{K}{2}\sqrt{\frac{M}{\pi}}
	\end{equation}
	for all $t>0$ and $\xi > \xi_0$.
	
	Let $t\in (0,1/\xi_0)$, and set $\xi \coloneqq 1/t > \xi_0$. From the 
	estimate \eqref{eq:inner_product_bound}, we have that
	\begin{equation}
		\label{eq:AbpT_bound}
		\big\|A^{\beta q}e^{-tA} x\big\| \leq \frac{eM_1}{ t^q} \|x\|
	\end{equation}
	for all $x \in D(A^{2+\beta q})$.
	The estimate 
	\eqref{eq:AbpT_bound} and the closed graph theorem
	show that 
	$A^{\beta q}e^{-tA} \in \mathcal{L}(H)$. Moreover,
	\[
	\big\|A^{\beta q}e^{-tA}\big\| \leq \frac{eM_1}{t^q}.
	\]
	Thus, $(e^{-tA})_{t \geq 0}$ is in the Crandall--Pazy class with parameter $\beta$.
\end{proof}

We are now ready to prove Theorem~\ref{thm:decay_growth}.
\begin{proof}[Proof of Theorem~\ref{thm:decay_growth}.]
	Proof of (i) $\Leftrightarrow$ (ii).
	In \cite[Theorem~2.1]{Crandall1969},  the implication (i) $\Rightarrow $ (ii) was proved.
	To prove  the implication (ii) $\Rightarrow $ (i), we first see that
	there exist $K\geq 1$ and $c \in \mathbb{R}$ such that 
	$\|e^{-tA}\| \leq Ke^{ct}$ for all $t \geq 0$. 
	Define $B \coloneqq A+c+1$.
	Then $-B$ generates an exponentially stable $C_0$-semigroup $(e^{-tB})_{t \geq 0}$
	on $H$. It suffices to show that $(e^{-tB})_{t \geq 0}$ is in the Crandall--Pazy class
	with parameter $\beta$.
	
	We will prove that 
	\begin{equation}
		\label{eq:Bbeta_resol_bound}
		\sup_{\lambda \in \mathbb{C}_+}\|B^{\beta}(\lambda +B)^{-1} \| < \infty.
	\end{equation}
	If the estimate \eqref{eq:Bbeta_resol_bound} is true, then
	the argument in the proof of \cite[Theorem~4.7]{Batty2016} (or
	its variant, \cite[Proposition~3.5]{Wakaiki2024JMAA}, directly) shows that 
	$e^{-tB}H \subset D(B^{\beta})$ for all $t>0$, and moreover that
	there exists $M >0$ such that 
	\[
	\big\|B^{\beta} e^{-tB}\big\| \leq \frac{M}{t}
	\]
	for all $t >0$. Therefore, Lemma~\ref{lem:domain_decay}
	shows that $(e^{-tB})_{t \geq 0}$ is in the Crandall--Pazy class
	with parameter $\beta$.
	
	By assumption,
	there exist $M_1,\delta >0$ 
	such that for all $\eta \in \mathbb{R}$
	satisfying $|\eta | \geq  \delta $, 
	\[
	i \eta \in \varrho(-A)\quad \text{and} \quad 
	\|(i \eta + A)^{-1} \| \leq \frac{M_1}{|\eta|^{\beta}}.
	\]
	The Hille-Yosida theorem shows that 
	\begin{equation}
		\label{eq:HY_estimate_cA}
		\|(\lambda + c+ A)^{-1}\| \leq \frac{K}{\re \lambda}
	\end{equation}
	for all $\lambda \in \mathbb{C}_+$.
	The resolvent equation
	\[
	(i \eta + B)^{-1} =
	(i \eta + A)^{-1}  - (c+1)  (i \eta + A)^{-1} 
	\big( (1 + i \eta)+ c + A \big)^{-1} 
	\]
	yields
	\[
	\|(i \eta + B)^{-1} \| 
	\leq (1+(c+1)K) \|(i \eta + A)^{-1} \|
	\leq \frac{ (1+(c+1)K)M_1}{|\eta|^{\beta}}
	\]
	for all $\eta \in \mathbb{R}$ satisfying $|\eta | \geq \delta $.
	Since $(e^{-tB})_{t \geq 0}$ is exponentially stable, we have
	\[
	\sup_{-\delta < \eta < \delta } \|(i \eta + B)^{-1} \| < \infty.
	\]
	Therefore,
	there exists $M_2 >0$ such that 
	\begin{equation}
		\label{eq:B_resol_ia}
		\|(i \eta + B)^{-1} \| \leq \frac{M_2}{(|\eta|+ 1)^\beta }
	\end{equation}
	for all $\eta \in \mathbb{R}$.
	As shown in \cite[Theorem~2, $(\mathrm{f}_{\alpha}) \Rightarrow
	(\mathrm{h}_{\alpha})
	$]{Eberhardt1994}, 
	combining \eqref{eq:B_resol_ia} with the estimate
	\[
	\|B(i \eta + B)^{-1}\| \leq 1 + |\eta| \, \| (i \eta + B)^{-1}\|,\quad \eta \in \mathbb{R}
	\]
	and the moment inequality given in Lemma~\ref{lem:moment_ineq}, we derive
	\begin{equation}
		\label{eq:B_resol_ia_bounded}
		\|B^{\beta} (i \eta +B)^{-1}\| \leq M_3
	\end{equation}
	for all $\eta \in \mathbb{R}$ and some $M_3 >0$.
	
	The estimate \eqref{eq:HY_estimate_cA} gives
	\begin{equation}
		\label{eq:resol_rh}
		\| (\xi + i \eta + B)^{-1}\| \leq \frac{K}{\xi+1}
	\end{equation}
	for all $\xi >0$ and $\eta \in \mathbb{R}$. 
	From \eqref{eq:B_resol_ia_bounded}, \eqref{eq:resol_rh}, and
	the resolvent equation, it follows that 
	\begin{align*}
		\|B^{\beta}(\xi + i\eta +B)^{-1} \| &\leq 
		\|
		B^{\beta}(i \eta+B)^{-1}
		\| + 
		\xi \| B^{\beta}(i\eta + B)^{-1}\| \, \| 
		(\xi + i\eta +B)^{-1}
		\| \\
		&\leq M_3(1+K)
	\end{align*}
	for all $\xi >0$ and $\eta \in \mathbb{R}$.
	Thus, the desired estimate \eqref{eq:Bbeta_resol_bound} holds.
	
	Proof of (i) $\Leftrightarrow$ (iii) $\Leftrightarrow$ (iii)'.
	Since the implication  (iii) $\Rightarrow$ (iii)' is obvious, 
	we will prove the implications  (i) $\Rightarrow$ (iii) and
	(iii)' $\Rightarrow$ (i).
	
	For the proof of  the implication (i) $\Rightarrow$ (iii),
	let $c >  \omega_A$ and $q \in (0,1/2)$.
	Take $c_1 \in (\omega_A,c)$, and define
	$B \coloneqq A+c_1$.
	Then $-B$ generates an exponentially stable 
	$C_0$-semigroup on $H$.
	By Lemma~\ref{lem:decay2resol_int}.a), there exists $M_1 >0$ such that
	\[
	\sup_{\xi > c-c_1} \xi^{1-2q}
	\int_{-\infty}^{\infty}
	\|B^{\beta q}(\xi + i \eta +B)^{-1}x \|^2 d\eta \leq M_1 \|x\|^2
	\]
	for all $x \in H$. Setting $\zeta \coloneqq \xi+c_1$, we obtain
	\begin{equation}
		\label{eq:zeta_c1_int_bound}
		\sup_{\zeta > c} \,(\zeta - c_1)^{1-2q}
		\int_{-\infty}^{\infty}
		\|(A+c_1)^{\beta q}(\zeta + i \eta +A)^{-1}x \|^2 d\eta \leq M_1 \|x\|^2
	\end{equation}
	for all $x \in H$.
	Since
	\[
	(A+c)^{\beta q} (A+c_1)^{-\beta q},\,
	(A+c_1)^{\beta q} (A+c)^{-\beta q} \in \mathcal{L}(H),
	\]
	there exist $C_1,C_2 >0$ such that 
	\begin{equation}
		\label{eq:A_c_ineq}
		C_1\|(A+c)^{\beta q}x \|^2 \leq  \|(A+c_1)^{\beta q}x \|^2 \leq 
		C_2\|(A+c)^{\beta q}x \|^2
	\end{equation}
	for all $x \in D(A)$. 
	Moreover, 
	\begin{equation}
		\label{eq:zeta_c_ineq}
		\left(
		\frac{c-c_1}{c}
		\right)^{1-2q} 
		\zeta^{1-2q} \leq  (\zeta - c_1)^{1-2q} \leq \zeta^{1-2q}
	\end{equation}
	for all $\zeta > c$.
	Combining 
	the inequalities~\eqref{eq:zeta_c1_int_bound}--\eqref{eq:zeta_c_ineq},
	we obtain
	\[
	\sup_{\zeta > c} \zeta^{1-2q}
	\int_{-\infty}^{\infty}
	\|(A+c)^{\beta q}(\zeta + i \eta +A)^{-1}x \|^2 d\eta \leq \frac{M_1}{C_1}\left(
	\frac{c}{c-c_1}
	\right)^{1-2q} 
	\|x\|^2
	\]
	for all $x \in H$. Hence, statement~(iii) holds.
	
	Finally, we prove the implication  (iii)' $\Rightarrow$ (i).
	Suppose that there
	exist $c > \omega_A$, $q \in (0,1/2)$, and $M>0$ 
	such that the estimate \eqref{eq:resol_int_unbounded}
	holds for all $x \in H$.
	Let $c_1 \in (\omega_A,c)$.
	Combining \eqref{eq:resol_int_unbounded} with
	\eqref{eq:A_c_ineq} and  \eqref{eq:zeta_c_ineq},
	we see that 
	\[
	\sup_{\zeta > c} \, (\zeta - c_1)^{1-2q}
	\int_{-\infty}^{\infty}
	\|(A+c_1)^{\beta q}(\zeta + i \eta +A)^{-1}x \|^2 d\eta \leq C_2M \|x\|^2
	\]
	for all $x \in H$. Set $\xi \coloneqq \zeta- c_1$ 
	and $B \coloneqq A+c_1$. Then
	\[
	\sup_{\xi > c-c_1} \xi^{1-2q}
	\int_{-\infty}^{\infty}
	\|B^{\beta q}(\xi + i \eta +B)^{-1}x \|^2 d\eta \leq C_2M \|x\|^2
	\]
	for all $x \in H$. This and Lemma~\ref{lem:decay2resol_int}.b) show 
	that $-B$ generates a $C_0$-semigroup
	in the Crandell--Pazy class. Thus, statement~(i) holds.
\end{proof}

We conclude this section by making two remarks.
\begin{remark}
		In the proof of Theorem~\ref{thm:decay_growth},
		the Plancherel theorem plays an important role,
		which is the reason why we consider Hilbert space semigroups.
		The equivalence of (i) and (iii) follows 
		from Lemma~\ref{lem:decay2resol_int}, where the Plancherel theorem is applied twice; see
		\eqref{eq:Abq_Plancherel} and 
		\eqref{eq:resol_int_bound_bounded}.
		Moreover, the Plancherel theorem
		is implicitly used in the proof of the implication 
		(ii) $\Rightarrow $ (i), since
		the argument in the proof of
		\cite[Theorem~4.7]{Batty2016}
		relies on the Plancherel theorem.
\end{remark}
\begin{remark}
		The above proof of the implication  (ii) $\Rightarrow$ (i) is based on \cite[Theorem~4.7]{Batty2016}. An alternative approach
		is as follows. 
		Let $-B$ be the generator of an
		exponentially stable $C_0$-semigroup $(e^{-tB})_{t \geq 0}$ on a Hilbert space, and let $\beta \in (0,1]$. Define
		the operator matrix
		\[
		\mathbf{B} \coloneqq
		\begin{pmatrix}
			B & B^{\beta} \\
			0 & B
		\end{pmatrix}
		\]
		with domain $D(B) \times D(B)$.
		By \cite[Theorem~2, $(\mathrm{a}_{\alpha}) \Leftrightarrow
		(\mathrm{d}_{\alpha})
		$]{Eberhardt1994}, 
		the $C_0$-semigroup 
		$(e^{-tB})_{t \geq 0}$ is in the Crandall--Pazy class with parameter $\beta$
		if and only if $-\mathbf{B}$
		generates a $C_0$-semigroup on $H \times H$.
		From an argument  used in the proof of \cite[Theorem~2.4]{Borichev2010}, 
		we see  that 
		if the estimate \eqref{eq:Bbeta_resol_bound} holds,
		the Plancherel theorem yields
		\[
		\sup_{\xi >0} \xi \int_{-\infty}^{\infty}
		\left(
		\|(\xi+i\eta+ \mathbf{B})^{-1} \mathbf{x}\|^2 + 
		\|(\xi+i\eta+ \mathbf{B}^*)^{-1} \mathbf{x}\|^2
		\right) d\eta < \infty
		\]
		for all $\mathbf{x} \in H \times H$. 
		This implies that 
		$-\mathbf{B}$ generates a $C_0$-semigroup on $H \times H$; see
		\cite[Theorem~2]{Gomilko1999} or \cite[Theorem~1]{Shi2000}.
\end{remark}

\section{Decay estimates via the $\mathcal{B}$-calculus}
\label{sec:Decay_estimate1}
In this section, we first 
present background material on the $\mathcal{B}$-calculus.
Next, we provide a new estimate for operator norms 
within the framework of the $\mathcal{B}$-calculus, using the 
integral estimate for the resolvent of the generator 
of a $C_0$-semigroup in the Crandall--Pazy class.
Then we estimate decay rates for the Crank--Nicolson scheme
with smooth initial data when
the  $C_0$-semigroup
$(e^{-tA})_{t \geq 0}$ is exponentially stable and is
in the Crandall--Pazy class.
We also give an estimate for the rate of decay of $\|e^{-tA^{-1}}A^{-\alpha}\|$
with $\alpha >0$ by the same approach.
Finally, we give a simple example illustrating that 
the obtained decay estimates cannot in general be improved
except for omitting logarithmic terms.

\subsection{Background material on the $\mathcal{B}$-calculus}
\label{sec:preliminaries_B_calc}
We recall some basic facts on the $\mathcal{B}$-calculus and refer to \cite{Batty2021,Batty2021JFA,Gomilko2024} for further information.
Let $\mathcal{B}$ be the space of 
holomorphic functions $f$ on $\mathbb{C}_+$ such that 
\[
\|f\|_{\mathcal{B}_{0}} \coloneqq
\int_0^{\infty}
\sup_{\eta \in \mathbb{R}} |f'(\xi+i\eta)| d\xi < \infty.
\]
If $f \in \mathcal{B}$, then 
\[
f(\infty) \coloneqq 
\lim_{\re z \to \infty} f(z)
\] 
exists in $\mathbb{C}$.

Let $\textrm{M}(\mathbb{R}_+)$ be the Banach algebra of 
bounded Borel measures on $\mathbb{R}_+$ with 
multiplication given by convolution and with norm given by
the total variation norm
$\|\cdot\|_{\textrm{M}(\mathbb{R}_+)}$. 
We identify $L^1(\mathbb{R}_+)$ with a subalgebra of $\textrm{M}(\mathbb{R}_+)$.
Let 
$\mathcal{LM}$ be the Banach algebra of Laplace transforms
$\hat \mu$ of $\mu \in M(\mathbb{R}_+)$, equipped with the norm $
\|\hat \mu\|_{\mathrm{HP}} \coloneqq \|\mu\|_{\textrm{M}(\mathbb{R}_+)}$.
We call $\mathcal{LM}$ the {\em Hille-Phillips algebra}.
Let $H^{\infty}(\mathbb{C}_+)$ be the Banach algebra
of bounded holomorphic functions on $\mathbb{C}_+$,
equipped with the supremum norm
\[
\|f\|_{\infty} \coloneqq 
\sup_{z \in \mathbb{C}_+} |f(z)|.
\]

We have that $\mathcal{LM} \subset
\mathcal{B} \subset H^{\infty}(\mathbb{C}_+)$.
Moreover, 
\begin{equation}
	\label{eq:inf_norm_bound}
	\|f\|_{\infty} \leq |f(\infty)| + \|f\|_{\mathcal{B}_0},\quad 
	f \in \mathcal{B},
\end{equation}
and
\[
\|f\|_{\infty} + \|f\|_{\mathcal{B}_0} \leq 2 \|f\|_{\textrm{HP}},\quad
f \in \mathcal{LM}.
\]
The space $\mathcal{B}$ equipped with the norm
\[
\|f\|_{\mathcal{B}} \coloneqq \|f\|_{\infty} + \|f\|_{\mathcal{B}_0},
\quad f \in  \mathcal{B},
\]
is a Banach algebra. We define 
$\mathcal{B}_0 \coloneqq \{
f \in \mathcal{B}: f(\infty) = 0
\}$.

\begin{example}
	\label{ex:r_alpha_Ds}
		Let $\alpha,c >0$, and define
		\[
		w_{\alpha,c} (z) \coloneqq \frac{1}{(z+c)^{\alpha}},\quad 
		z \in \mathbb{C}_+.
		\]
		Then $w_{\alpha,c}$
		is the Laplace transform of the 
		$L^1(\mathbb{R}_+)$-function 
		\[
		t \mapsto \frac{t^{\alpha-1}e^{-ct}}{\Gamma(\alpha)};
		\]
		see, e.g., 
		\cite[Lemma~3.3.4]{Haase2006}.
		Hence, $w_{\alpha,c} \in \mathcal{LM}
		\subset  \mathcal{B}$. 
\end{example}

Let $-A$ be the generator of a bounded $C_0$-semigroup
$(e^{-tA})_{t \geq 0}$ on a Hilbert space $H$.
Let $f\in \mathcal{B}$, and define a linear operator $f(A)$ on $H$ by
\begin{equation}
	\label{eq:Bcalc_def}
	\langle
	f(A)x,y
	\rangle
	\coloneqq
	f(\infty)
	\langle
	x,y
	\rangle -
	\frac{2}{\pi}
	\int_0^{\infty} \xi
	\int_{-\infty}^{\infty}
	\langle
	(\xi - i  \eta +A)^{-2} x,y
	\rangle
	f'(\xi + i \eta) d\eta d\xi
\end{equation}
for $x,y \in H$. 
Using the Plancherel theorem, we obtain
\begin{equation}
	\label{eq:resol_int_bound_x}
	\sup_{\xi >0 }\xi \int_{-\infty}^{\infty} \|(\xi-i\eta+A)^{-1}x\|^2 d\eta  = 
	2\pi \sup_{\xi >0}\xi 
	\int_0^{\infty} e^{-2\xi t} \big\|e^{-tA}x\big\|^2 dt
	\leq \pi K^2 \|x\|^2
\end{equation}
for all $x \in H$, where $K \coloneqq \sup_{t \geq 0} \|e^{-tA}\|$.
Since $-A^*$ also generates a bounded $C_0$-semigroup
$(e^{-tA^*})_{t \geq 0}$ and since $K = \sup_{t \geq 0} \|e^{-tA^*}\|$,
the Plancherel theorem yields
\begin{equation}
	\label{eq:resol_int_bound_y}
	\sup_{\xi >0 }\xi \int_{-\infty}^{\infty} \|(\xi+i\eta+A^*)^{-1}y\|^2 d\eta 
	\leq \pi K^2 \|y\|^2
\end{equation}
for all $y \in H$. The inequalities 
\eqref{eq:resol_int_bound_x} and \eqref{eq:resol_int_bound_y},
together with the Cauchy-Schwartz inequality, give
\[
\sup_{\xi >0 }\xi \int_{-\infty}^{\infty} | \langle (\xi-i\eta+A)^{-2}x,y \rangle| d\eta 
\leq \pi K^2 \|x\|\, \|y\|
\]
for all $x,y \in H$.
Therefore, 
the operator $f(A)$ defined by \eqref{eq:Bcalc_def}
is bounded on $H$ and satisfies 
\begin{equation}
	\label{eq:B_calc_bound}
	\|f(A)\| \leq |f(\infty)| + 2K^2 \|f\|_{\mathcal{B}_0} \leq 2K^2 \|f\|_{\mathcal{B}}.
\end{equation}
The map 
\[
\Phi_A \colon \mathcal{B} \to \mathcal{L}(H),\quad f\mapsto f(A)
\]
is the unique bounded algebraic homomorphism
from $\mathcal{B}$ to $\mathcal{L}(H)$
such that the equality 
$\Phi_A((\lambda+ \cdot \,)^{-1}) = (\lambda + A)^{-1} $
holds for all $\lambda \in \mathbb{C}_+$.
We call $\Phi_A$ 
the {\em $\mathcal{B}$-calculus} for $A$. 

For $f = \hat \mu \in \mathcal{LM}$, we define
\[
\Pi_A (f)x \coloneqq \int_{\mathbb{R}_+} e^{-tA} x \mu(dt),\quad x \in X.
\]
Then $\Pi_A (f) \in \mathcal{L}(X)$ and
$\|\Pi_A (f)\| \leq K \|f\|_{\textrm{HP}}$. 
The map
$\Pi_A \colon \mathcal{LM} \to \mathcal{L}(X)$
is an algebra homomorphism and is known as the 
{\em Hille-Phillips calculus} for $A$.
If $f \in \mathcal{LM}$, then
$\Pi_A (f)=\Phi_A (f)$.
It follows from this that $w_{\alpha,c}$
given in Example~\ref{ex:r_alpha_Ds} satisfies
$w_{\alpha,c}(A) = 
(A+c)^{-\alpha}$, since 
$
\Pi_A(w_{\alpha,c}) = 
(A+c)^{-\alpha}
$
by \cite[Proposition~3.3.5]{Haase2006}.

\subsection{Norm estimate for the Crandall--Pazy class of semigroups}
The norm estimate \eqref{eq:B_calc_bound} is valid 
for every generator of a bounded $C_0$-semigroup on a Hilbert space.
Focusing on the Crandall--Pazy class of bounded 
$C_0$-semigroups on Hilbert spaces,
we present a new norm estimate.
To this end, 
we let $q >0$ and define
\begin{equation}
	\label{eq:vertiii}
	\vertiii{f} \coloneqq
	\int_0^{\infty} \varphi_q(\xi) \sup_{\eta \in \mathbb{R}} |f'(\xi+i\eta)| d\xi,
	\quad f \in \mathcal{B},
\end{equation}
where
\begin{equation}
	\label{eq:psi_def}
	\varphi_q(\xi) \coloneqq 
	\begin{cases}
		1, & 0<\xi < 1, \\
		\xi^{q }, & \xi \geq 1.
	\end{cases}
\end{equation}
Let $f \in \mathcal{B}_0$ and $g \in \mathcal{B}$ satisfy $\vertiii{f},\vertiii{g} < \infty$.
Since $\|f\|_{\mathcal{B}_0} \leq \vertiii{f}$ and 
$f(\infty) = 0$, 
the inequality \eqref{eq:inf_norm_bound} yields $\|f\|_{\infty}   \leq \vertiii{f}$. Therefore,
we obtain
\begin{equation}
	\label{eq:product_estimate}
	\vertiii{fg} \leq 
	\vertiii{f} \big 
	(\vertiii{g} +\|g\|_{\infty} \big).
\end{equation}

\begin{example}
	\label{ex:ra_B0q_norm}
		Let $\alpha >0$, $c\geq d > 0$, and $q \in (0,1)$.
		Define $v_{\alpha,c,d}$ by 
		\[
		v_{\alpha,c,d}(z) \coloneqq \left( \frac{z+d}{z+c} \right)^{\alpha},
		\quad z \in \mathbb{C}_+.
		\]
		Then $\vertiii{v_{\alpha,c,d}} < \infty$.
		Indeed, since
		\begin{align*}
			\sup_{\eta \in \mathbb{R}}
			\left|
			v_{\alpha,c,d}'(\xi + i\eta)
			\right| \leq 
			\sup_{\eta \in \mathbb{R}} \frac{\alpha(c-d)}{(\xi+c)^2+\eta^2} 
			\leq \frac{\alpha(c-d)}{(\xi+c)^2}
		\end{align*}
		for all $\xi >0$,
		we have that 
		\begin{align*}
			\vertiii{v_{\alpha,c,d}} &\leq 
			\int_0^{1} \frac{\alpha(c-d)}{(\xi+c)^2} d\xi +
			\int_1^{\infty} \frac{\alpha(c-d)\xi^q}{(\xi+c)^2} d\xi \\
			&\leq 
			\frac{\alpha(c-d)}{c^2} + \frac{\alpha(c-d)}{(1-q)(1+c)^{1-q}}.
		\end{align*}
\end{example}

Using the integral condition on the resolvent  given in 
Lemma~\ref{lem:decay2resol_int}, we obtain a norm estimate tailored 
to the Crandall--Pazy class of $C_0$-semigroups.
\begin{proposition}
	\label{prop:AbqfA_bound}
	Let $-A$ be the generator of a bounded  $C_0$-semigroup
	$(e^{-tA})_{t \geq 0}$ on a Hilbert space $H$. 
	If $(e^{-tA})_{t \geq 0}$ is in the Crandall--Pazy class with parameter $\beta\in (0,1]$, then
	the following statements hold for a fixed $q \in (0,1/2)$, where 
	$\vertiii{\cdot}$ is as in \eqref{eq:vertiii}:
	\begin{enumerate}
		\renewcommand{\labelenumi}{\alph{enumi})}
		\item $A^{\beta q}f(A) \in \mathcal{L}(H)$ 
		for all $f \in \mathcal{B}_0$ satisfying  $\vertiii{f} < \infty$.
		\item There exists a constant $M>0$ such that 
		\begin{equation}
			\label{eq:AbqfA_bound}
			\|A^{\beta q}f(A)\| \leq M \vertiii{f}
		\end{equation}
		for all $f \in \mathcal{B}_0$ satisfying $\vertiii{f} < \infty$.
	\end{enumerate}
\end{proposition}
\begin{proof}
	Let $x \in D(A^{\beta q} )$ and $y \in H$.
	From \eqref{eq:Bcalc_def}, we have that 
	for all $f \in \mathcal{B}_0$,  
	\[
	\langle
	f(A)A^{\beta q}x,y
	\rangle
	=
	-
	\frac{2}{\pi}
	\int_0^{\infty} \xi
	\int_{-\infty}^{\infty}
	\langle
	(\xi - i  \eta +A)^{-1} A^{\beta q}x,(\xi + i  \eta +A^*)^{-1}y
	\rangle
	f'(\xi + i \eta) d\eta d\xi.
	\]
	The Cauchy-Schwartz inequality yields
	\begin{align}
		&\int_{-\infty}^{\infty}
		|\langle
		(\xi - i  \eta +A)^{-1} A^{\beta q}x,(\xi + i  \eta +A^*)^{-1}y
		\rangle
		f'(\xi + i \eta)| d\eta \notag \\
		&\hspace{10pt} \leq\sup_{\eta \in \mathbb{R}} |f'(\xi+i
		\eta)| 
		\left(\int_{-\infty}^{\infty}
		\|A^{\beta q}(\xi - i \eta +A)^{-1}x \|^2 d\eta \right)^{1/2}
		\left(\int_{-\infty}^{\infty}
		\|(\xi + i \eta +A^*)^{-1}y \|^2 \right)^{1/2} \label{eq:CS_ineq}
	\end{align}
	for all $\xi >0$. 
	Lemma~\ref{lem:decay2resol_int} with $\xi_0 = 1$ 
	shows that 
	\begin{equation}
		\label{eq:W1x_bound}
		\sup_{\xi >0 }  \frac{\xi}{\varphi_q(\xi)^2}
		\int_{-\infty}^{\infty}
		\|A^{\beta q}(\xi - i \eta +A)^{-1}x \|^2 d\eta\leq M \|x\|^2
	\end{equation}
	for some $M>0$ independent of $x$.
	The Plancherel theorem  gives
	\begin{equation}
		\label{eq:W2y_bound}
		\sup_{\xi >0 } \xi \int_{-\infty}^{\infty}
		\|(\xi + i \eta +A^*)^{-1}y \|^2 d\eta \leq \pi K^2 \|y\|^2,
	\end{equation}
	where $K \coloneqq \sup_{t \geq 0} \|e^{-tA}\|$.
	
	Let $f\in \mathcal{B}_0$ satisfy $\vertiii{f} < \infty$. The
	estimates~\eqref{eq:CS_ineq}--\eqref{eq:W2y_bound} give
	\begin{align*}
		&\int_0^{\infty} \xi
		\int_{-\infty}^{\infty}
		|\langle
		(\xi - i  \eta +A)^{-1} A^{\beta q}x,(\xi + i  \eta +A^*)^{-1}y
		\rangle
		f'(\xi + i \eta)| d\eta d\xi\\
		&\quad \leq 
		\int_0^{\infty} \varphi_q(\xi) \sup_{\eta \in \mathbb{R}} |f'(\xi+i
		\eta)| \left(M \|x\|^2 \right)^{1/2}  \left(\pi K^2 \|y\|^2\right)^{1/2}  d\xi \\
		&\quad \leq  K \sqrt{\pi M}  \vertiii{f}  \|x\|\, \|y\|.
	\end{align*}
	Hence,
	\begin{equation}
		\label{eq:AbqfAx_bound}
		\|A^{\beta q}f(A)x\| \leq 2K \sqrt{\frac{M}{\pi}} \vertiii{f} \|x\|
	\end{equation}
	This, together with
	the closed graph theorem,
	shows that $A^{\beta q}f(A) \in \mathcal{L}(H)$. Moreover, 
	from the estimate \eqref{eq:AbqfAx_bound}, we obtain
	\[
	\|A^{\beta q}f(A)\| \leq 2K \sqrt{\frac{M}{\pi}} \vertiii{f},
	\]
	and the constant $2K\sqrt{M/\pi}$ depends on $A$ and $q$ but not on 
	$f$.
\end{proof}

\subsection{Decay estimate for Crank--Nicolson schemes}
\subsubsection{Holomorphic $C_0$-semigroups on Banach spaces}
\label{sec:CN_scheme_Holomorphic}
Now we study the rate of decay for the Crank--Nicolson scheme.
Before presenting a decay estimate in the case of the
Crandall--Pazy class of $C_0$-semigroups on Hilbert spaces,
we briefly discuss
the case of holomorphic $C_0$-semigroups on Banach spaces.
For an operator $A$ on a
Banach space $X$, it is well known that 
$-A$ is the generator of a sectorially bounded 
holomorphic $C_0$-semigroup on $X$ if and only if $A \in \Sect(\pi/2-)$; see, e.g.,  \cite[Theorem~3.7.11]{Arendt2001}
and 
\cite[Theorem~II.4.6]{Engel2000}.
Recall that for 
$0 < \omega_{\min} \leq \omega_{\max} < \infty$, we write
$(\omega_k)_{k \in \mathbb{N}} \in 
\mathcal{S}(\omega_{\min},\omega_{\max})$ if
$\omega_{\min} \leq 
\omega_k \leq \omega_{\max}$ for all $k \in \mathbb{N}$.
\begin{theorem}
	\label{thm:CN_holomorphic}
	Let an operator $A$ on a 
	Banach space $X$ satisfy $A \in \Sect(\pi/2-)$
	and $0 \in \varrho(A)$.
	For each $\alpha \geq 0$ and $0 < \omega_{\min} \leq  \omega_{\max} < \infty$,
	there exists a constant $M>0$ such that 
	\begin{equation}
		\label{eq:VA_bound}
		\left\|
		\left(
		\prod_{k=1}^n
		V_{\omega_k}(A)
		\right)
		A^{-\alpha}
		\right\| 
		\leq \frac{M}{n^{\alpha}}
	\end{equation}
	for all $n \in \mathbb{N}$ and 
	$(\omega_k)_{k \in \mathbb{N}} \in \mathcal{S}(\omega_{\min},\omega_{\max})$.
\end{theorem}

\begin{proof}
	Define $B \coloneqq A^{-1}$.
	Since $B  \in \Sect(\pi/2-) \cap \mathcal{L}(X)$,
	it follows from 
	\cite[Theorem~4]{Saito2004} that 
	the following estimate holds:
	For each $\alpha \geq 0$ and $0 < \omega_{\min} \leq  \omega_{\max} < \infty$,
	there exists $ M>0$ such that 
	\[
	\left\|
	\left(
	\prod_{k=1}^n
	V_{1/\omega_k}(B)
	\right)
	B^{\alpha}
	\right\| 
	\leq \frac{M}{n^{\alpha}}
	\]
	for all $n \in \mathbb{N}$ and 
	$(\omega_k)_{k \in \mathbb{N}} \in \mathcal{S}(\omega_{\min},\omega_{\max})$.
	Since
	$
	V_{1/\omega}(B) = -V_{\omega}(A)
	$
	for all $\omega >0$, we obtain the desired estimate \eqref{eq:VA_bound}.
\end{proof}
\begin{remark}
		Theorem~\ref{thm:CN_holomorphic} can also be proved 
		via the $\mathcal{D}$-calculus and the
		$\mathcal{H}$-calculus introduced in \cite{Batty2023}.
		A detailed proof based on
		the $\mathcal{D}$-calculus is provided in Appendix for interested readers. 
		Rational approximation of holomorphic $C_0$-semigroups 
		via the
		$\mathcal{H}$-calculus can be found in \cite{Batty2025}.
\end{remark}

We show that the decay rate given in 
Theorem~\ref{thm:CN_holomorphic}
cannot in general be improved.
\begin{proposition}
	Let $A$ be a normal 
	operator on a Hilbert space $H$ such that 
	$\sigma(A) \subset \{
	\lambda \in \mathbb{C}: \re \lambda \geq \delta
	\}$
	for some $\delta >0$.
	Assume that
	there exist a sequence $(\lambda_k)_{k \in \mathbb{N}}$
	of complex numbers
	and a constant $\theta \in (0,\pi/2)$ such that 
	$\lambda_k \in \sigma(A) \cap \Sigma_{\theta}$
	for all $k \in \mathbb{N}$ and
	$|\lambda_k| \to \infty$ as $k \to \infty$.
	Then
	\begin{equation}
		\label{eq:inflim_normal}
		\limsup_{n \to \infty} n^{\alpha} \|V_{\omega} (A)^nA^{-\alpha}\| >0
	\end{equation}
	for all $\alpha,\omega >0$.
\end{proposition}

\begin{proof}
	Let $c \coloneqq \delta/2$.
	Define $B \coloneqq A- c$
	and $\mu_k \coloneqq \lambda_k - c \in \sigma(B)$. Then
	$\mu_k \in \Sigma_{\theta'}$ for all $k \in \mathbb{N}$ 
	and  some $\theta' \in (\theta,\pi/2)$.
	
	Let $n \in \mathbb{N}$ and $\alpha,\omega >0$.
	Define
	\[
	f_{n,\alpha}(z) \coloneqq 
	\frac{(z+c-\omega)^{n}}{(z+c+\omega)^{n+\alpha}},\quad 
	z \in \mathbb{C}_+.
	\]
	For all $r \geq |c-\omega|$ and $\phi \in (0,\theta')$, we obtain
	\begin{align}
		|f_{n,\alpha}(r e^{i\phi})| &=
		\frac{\big(r^2 + 2 (c-\omega) r\cos \phi  + (c-\omega)^2\big)^{n/2}}
		{\big(r^2 + 2  (c+\omega) r \cos \phi+ (c+\omega)^2\big)^{(n+\alpha)/2}} \notag \\
		&\geq 
		\frac{(r- |c-\omega|)^{n}}{(r+c+\omega)^{n+\alpha}} \eqqcolon
		g_{n,\alpha}(r).
		\label{eq:fna_lower_bound}
	\end{align}
	Moreover, $g_{n,\alpha}'(r) \leq 0$
	if 
	\[
	r \geq |c-\omega| + \frac{c+\omega + |c-\omega|}{\alpha}n \eqqcolon
	\tau(n)
	\]
	is satisfied. 
	Since $|\mu_k| \to \infty$ as $k \to \infty$, there exists $k_0 \in \mathbb{N}$
	such that $|\mu_k| \geq \tau(1)$ for all $k \geq k_0$.
	For each $k \geq k_0$,
	we take $n_k \in \mathbb{N}$ satisfying
	\[
	\tau(n_k)\leq |\mu_k|
	< \tau(n_k+1).
	\]
	Then $n_k \to \infty$ as $k \to \infty$.
	Since $g_{n_k,\alpha}'(r) \leq 0$
	for all $r \geq |\mu_k|$, we obtain
	\begin{equation}
		\label{eq:g_decrease}
		g_{n_k,\alpha} (|\mu_k|) \geq g_{n_k,\alpha}\big(\tau(n_k+1)\big)
	\end{equation}
	for all $k \geq k_0$. A simple calculation shows 
	that for all $k \geq k_0$,
	\begin{equation}
		\label{eq:g_nk_rep}
		g_{n_k,\alpha}\big(\tau(n_k+1)\big)=
		\frac{1}{  (c+\omega+|c-\omega|)^{\alpha}}
		\left(
		1+ \frac{\alpha}{n_k+1}
		\right)^{-n_k}  \left(
		\frac{\alpha}{n_k+\alpha+1}
		\right)^{\alpha}.
	\end{equation}
	Using the normality of $B$ and the estimate~\eqref{eq:fna_lower_bound},
	we have that
	\begin{align}
		\label{eq:V_lower_bound}
		\|V_{\omega} (A)^{n_k} (A+\omega)^{-\alpha} \| =	
		\sup_{\mu \in \sigma(B)} |f_{n_k,\alpha}(\mu)| 
		\geq g_{n_k,\alpha} (|\mu_k|) 
	\end{align}
	for all $k \geq k_0$.
	Combining \eqref{eq:g_decrease}--\eqref{eq:V_lower_bound},
	we obtain the desired estimate \eqref{eq:inflim_normal}.
\end{proof}

\subsubsection{Crandall--Pazy class of 
	$C_0$-semigroups on Hilbert spaces}
\label{sec:CP_CN_decay}
Let $-A$ be the generator of an
exponentially stable $C_0$-semigroup  
on a Hilbert space. Then, for each $\alpha > 0$ and 
$0 < \omega_{\min} \leq  \omega_{\max} < \infty$,
there exists $M>0$ such that 
\begin{equation}
	\label{eq:exp_case_alpha}
	\left\|\left(\prod_{k=1}^n
	V_{\omega_k}(A)\right) A^{-\alpha}
	\right\| \leq  \frac{M}{n^{\alpha/2}} 
\end{equation}
for all $n \in \mathbb{N}$ and 
$(\omega_k)_{k \in \mathbb{N}} \in \mathcal{S}(\omega_{\min},\omega_{\max})$; see \cite[Theorem~2]{Wakaiki2024JEE}.
Now we give a decay estimate for 
exponentially stable $C_0$-semigroups in the Crandall--Pazy class, which
bridges the gap between the estimates \eqref{eq:VA_bound} and 
\eqref{eq:exp_case_alpha}.
For $\alpha,\varepsilon>0$ and $\beta \in (0,1]$, 
define the function $L_{\alpha,\beta, \varepsilon}$
on $(0,\infty)$ by
\begin{equation}
	\label{eq:Labe}
	L_{\alpha,\beta, \varepsilon}(\tau) \coloneqq
	\begin{cases}
		\log (\tau+1), & 0 < \alpha < \dfrac{2-\beta}{2}, \vspace{3pt}\\
		\big(\log (\tau+1)\big)^{2\alpha/(2-\beta) + \varepsilon},
		& \alpha \geq \dfrac{2-\beta}{2}.\vspace{3pt}
	\end{cases}
\end{equation}
\begin{theorem}
	\label{thm:VA_bound_nonanalytic}
	Let $-A$ be the generator of 
	an exponentially stable $C_0$-semigroup $(e^{-tA})_{t \geq 0}$
	in the Crandall--Pazy class with parameter $\beta \in (0,1]$
	on a Hilbert space $H$.
	For each $\alpha,\varepsilon >0$ and $0 < \omega_{\min} \leq  \omega_{\max}< \infty$,
	there exists a constant $M>0$ such that
	\begin{equation}
		\label{eq:VA_bound_nonanalytic}
		\left\|
		\left(
		\prod_{k=1}^n
		V_{\omega_k}(A)
		\right)
		A^{-\alpha}
		\right\| 
		\leq M\frac{L_{\alpha,\beta, \varepsilon}(n)}{ n^{\alpha/(2-\beta)}}
	\end{equation}
	for all $n \in \mathbb{N}$ and 
	$(\omega_k)_{k \in \mathbb{N}} \in \mathcal{S}(\omega_{\min},\omega_{\max})$,
	where $L_{\alpha,\beta, \varepsilon}$ is as in \eqref{eq:Labe}.
\end{theorem}

We will prove Theorem~\ref{thm:VA_bound_nonanalytic} via the 
$\mathcal{B}$-calculus.
Fix $c >0$ and $0<\omega_{\min}\leq \omega_{\max} < \infty$. 
Define
\begin{equation}
	\label{eq:fnaw_def_CP}
	f_{n,\alpha,(\omega_k)}(z) \coloneqq
	\frac{1}{(z+c+\omega_{\min})^{\alpha}} 
	\prod_{k=1}^n
	\frac{z+c-\omega_k}{z+c+\omega_k},\quad z \in \mathbb{C}_+,
\end{equation}
where $n \in \mathbb{N}$, $\alpha>0$, and 
$(\omega_k)_{k\in \mathbb{N}} \in \mathcal{S}(\omega_{\min},\omega_{\max})$.
It was shown in \cite[Proposition~4]{Wakaiki2024JEE} that 
$\|f_{n,\alpha,(\omega_k)}\|_{\mathcal{B}_0} = O(n^{-\alpha/2})$ as $n \to \infty$
for $\alpha >0$.
To prove Theorem~\ref{thm:VA_bound_nonanalytic}, we here estimate
$\vertiii{f_{n,\alpha,(\omega_k)}}$ for $q \in (0,1)$ and $\alpha > q$.
For that purpose,
we define the function $F_{\alpha,q}\colon \mathbb{N} \to \mathbb{R}_+$ by
\begin{equation}
	\label{eq:Faq_def}
	F_{\alpha,q}(n) \coloneqq
	\begin{cases}
		\dfrac{1}{n^{\alpha-q}},& q < \alpha < 2q, \vspace{3pt}\\
		\dfrac{\log (n+1)}{n^{q}},& \alpha = 2q, \vspace{3pt}\\
		\dfrac{1}{n^{\alpha/2}},& \alpha > 2q. \vspace{3pt}\\
	\end{cases}
\end{equation}
\begin{proposition}
	\label{prop:fnaw_bound_nonanalytic}
	Let $c>0$, $q \in (0,1)$, and $0 < \omega_{\min} \leq \omega_{\max}< \infty$.
	If $\alpha > q$, then
	there exists a constant $M>0$  such that 
	the function $f_{n,\alpha,(\omega_k)}$ defined by \eqref{eq:fnaw_def_CP} satisfies
	\begin{equation}
		\label{eq:fnaw_bound_nonanalytic}
		\bigvertiii{f_{n,\alpha,(\omega_k)}} \leq M F_{\alpha,q}(n)
	\end{equation}
	for all $n \in \mathbb{N}$ and $(\omega_k)_{k \in \mathbb{N}} \in
	\mathcal{S}(\omega_{\min},\omega_{\max})$, where the function $F_{\alpha,q}$ is as in 
	\eqref{eq:Faq_def}.
\end{proposition}

Before proceeding to the proof of Proposition~\ref{prop:fnaw_bound_nonanalytic},
we recall some technical results from \cite{Wakaiki2024JEE}, where
$\|f_{n,\alpha,(\omega_k)}\|_{\mathcal{B}_0}$ was examined 
instead of $\vertiii{f_{n,\alpha,(\omega_k)}}$.
The following estimate was obtained in \cite[Lemma~2]{Wakaiki2024JEE}.
\begin{lemma}
	\label{lem:each_bound}
	Let 
	$c> 0$ and $0 < \omega_{\min} \leq \omega_{\max} < \infty$
	satisfy  $c^2 \geq \omega_{\min} \omega_{\max}$. 
	If $\omega_{\min} \leq \omega \leq \omega_{\max}$, then
	\[
	\frac{(\xi+c-\omega)^2+s}{(\xi+c+\omega)^2+s} \leq 
	\frac{(\xi+c-\omega_{\min})^2+s}{(\xi+c+\omega_{\min})^2+s} 
	\] 
	for all $\xi >0$ and $s \geq 0$. 
\end{lemma}

Let $n \in \mathbb{N}$ and $c,\alpha,\omega>0$. Define
\begin{align}
	g_{n,\alpha,\omega}(\xi,s) &\coloneqq
	\frac{\big( 
		(\xi + c - \omega)^2 + s
		\big)^{n/2}}{ \big( 
		(\xi +c + \omega)^2 + s
		\big)^{(n+\alpha+1)/2}}, \label{eq:gnaw_def}\\
	h_{n,\alpha,\omega}(\xi,s) &\coloneqq
	\frac{\big( 
		(\xi + c- \omega)^2 + s
		\big)^{(n-1)/2}}{ \big( 
		(\xi + c+ \omega)^2 + s
		\big)^{(n+\alpha+1)/2}}\label{eq:hnaw_def}
\end{align}
for $\xi >0$ and $s \geq 0$. 
Then $\sup_{s \geq 0} g_{n,\alpha,\omega}(\xi,s)$ can be written as follows; 
see the proof of \cite[Lemma~3]{Wakaiki2024JEE} for the detailed derivation.
There exists $n_1 \in \mathbb{N}$ such that 
the equation
\[
-\xi^2 + 2
\left(
\frac{2\omega n}{\alpha+1} -c+ \omega
\right) \xi
+
\frac{4c \omega n}{\alpha+1} - (c - \omega)^2 = 0
\]
has a unique positive  solution $\xi_1(n)$ for all 
$n \geq n_1$. For each $n \geq n_1$,
\begin{equation}
	\label{eq:xi1_expression}
	\xi_1(n) 
	=
	\left(
	\frac{2\omega n}{\alpha+1} - c + \omega
	\right) + \frac{2\omega \sqrt{n (n + \alpha + 1)}}{\alpha + 1}.
\end{equation}
Using the positive solution $\xi_1(n)$, 
we obtain
\begin{equation}
	\label{eq:g_sup}
	\sup_{s \geq 0} g_{n,\alpha,\omega}(\xi,s) =
	\begin{cases}
		\displaystyle 	\left(
		\frac{n}{n+\alpha+1}
		\right)^{n/2}
		\left(
		\frac{\alpha+1}{4\omega(n+\alpha+1)}
		\right)^{(\alpha+1)/2}\frac{1}{(\xi+c)^{(\alpha+1)/2}}  , &\hspace{-1pt} \xi \leq \xi_1(n), \vspace{5pt}\\
		\displaystyle \frac{(\xi+c-\omega)^{n}}{(\xi+c+\omega)^{n+\alpha +1}}, &\hspace{-1pt} \xi> \xi_1(n)
	\end{cases} 
\end{equation}
for all $n \geq n_1$.

Using the expression of $\sup_{s \geq 0} g_{n,\alpha,\omega}(\xi,s)$
given in \eqref{eq:g_sup},
we prove the following  lemma, which will be useful in the proof of Proposition~\ref{prop:fnaw_bound_nonanalytic}.

\begin{lemma}
	\label{lem:g_h_int_bound}
	Let $n \in \mathbb{N}$, $c,\omega >0$, and 
	$\alpha > q >0$.
	Define $\varphi_q$, $g_{n,\alpha,\omega}$, 
	and $h_{n,\alpha,\omega}$ by \eqref{eq:psi_def}, 
	\eqref{eq:gnaw_def}, and \eqref{eq:hnaw_def},
	respectively.
	Then
	\begin{equation}
		\label{eq:g_bound_sup}
		\sup_{n \in \mathbb{N}} \int_0^{\infty}
		\varphi_q(\xi) \sup_{s \geq 0} 
		g_{n,\alpha,\omega}(\xi,s) d\xi  < \infty
	\end{equation}
	and
	\begin{equation}
		\label{eq:h_bound_sup}
		\sup_{n \in \mathbb{N}} \int_0^{\infty}
		\varphi_q(\xi) \sup_{s \geq 0} 
		h_{n,\alpha,\omega}(\xi,s) d\xi  < \infty.
	\end{equation}
	Moreover,
	\begin{equation}
		\label{eq:g_bound} 
		\int_0^{\infty}
		\varphi_q(\xi) \sup_{s \geq 0} 
		g_{n,\alpha,\omega}(\xi,s) d\xi =
		\begin{cases}
			O\left(
			\dfrac{1}{n^{\alpha-q}}
			\right),& 
			q<\alpha < 2q+1, \vspace{3pt}\\
			O\left(
			\dfrac{\log n}{n^{q+1}}
			\right),& 
			\alpha = 2q+1, \vspace{3pt}\\
			O\left(
			\dfrac{1}{n^{(\alpha+1)/2}}
			\right),& 
			\alpha > 2q+1
		\end{cases}
	\end{equation}
	and
	\begin{equation}
		\label{eq:h_bound} 
		\int_0^{\infty}
		\varphi_q(\xi) \sup_{s \geq 0} 
		h_{n,\alpha,\omega}(\xi,s) d\xi =
		\begin{cases}
			O\left(
			\dfrac{1}{n^{\alpha-q+1}}
			\right),& 
			q< \alpha < 2q, \vspace{3pt}\\
			O\left(
			\dfrac{\log n}{n^{q+1}}
			\right),& 
			\alpha = 2q, \vspace{3pt}\\
			O\left(
			\dfrac{1}{n^{\alpha/2 + 1}}
			\right),& 
			\alpha > 2q
		\end{cases}
	\end{equation}
	as $n \to \infty$.
\end{lemma}
\begin{proof}
	For all $\xi>0$ and $s \geq 0$,
	\[
	g_{n,\alpha,\omega}(\xi,s)  \leq \frac{1}{\big((\xi+c+\omega)^2+s\big)^{(\alpha+1)/2}}.
	\]
	Hence,
	\[
	\sup_{s \geq 0} g_{n,\alpha,\omega}(\xi,s)  \leq \frac{1}{(\xi+c+\omega)^{\alpha+1}}
	\]
	for all $\xi >0$. Since $\alpha > q$, we obtain
	\begin{align*}
		\int_0^{\infty}
		\varphi_q(\xi) \sup_{s \geq 0} 
		g_{n,\alpha,\omega}(\xi,s) d\xi
		&\leq 
		\int_0^{1}
		\frac{1}{(\xi+c+\omega)^{\alpha+1}}
		d\xi 
		+
		\int_1^{\infty}
		\frac{\xi^q}{(\xi+c+\omega)^{\alpha+1}}
		d\xi \\
		&\leq 
		\frac{1}{(c+\omega)^{\alpha+1}} +
		\frac{1}{(\alpha-q)(1+c+\omega)^{\alpha - q}}
	\end{align*}
	for all $n \in \mathbb{N}$. Hence, the estimate \eqref{eq:g_bound_sup} 
	for $g_{n,\alpha,\omega}$ holds.
	The estimate \eqref{eq:h_bound_sup} 
	for $h_{n,\alpha,\omega}$
	can be proved in a similar way.
	Since the estimate \eqref{eq:h_bound} for $h_{n,\alpha,\omega}$
	can be obtained  by replacing $n$ with $n-1$ and $\alpha$
	with $\alpha+1$ in the estimate \eqref{eq:g_bound} for 
	$g_{n,\alpha,\omega}$,
	we will show only the estimate \eqref{eq:g_bound}.

	Let $n_1 \in \mathbb{N}$ and $\xi_1(n) >0$ be
	as in the paragraph following Lemma~\ref{lem:each_bound}.
	Using \eqref{eq:g_sup}, we obtain
	\begin{align}
		&\int_0^{\xi_1(n)} \varphi_q(\xi) \sup_{ s \geq 0}
		g_{n,\alpha,\omega}(\xi,x) d\xi =
		\left(
		\frac{n}{n+\alpha+1}
		\right)^{n/2}
		\left(
		\frac{\alpha+1}{4\omega(n+\alpha+1)}
		\right)^{(\alpha+1)/2}
		\int_0^{\xi_1(n)}
		\frac{\varphi_q(\xi)}{(\xi+c)^{(\alpha+1)/2}} d\xi
		\label{eq:vp_g_estimate1}
	\end{align}
	for all $n \geq n_1$.
	We see from \eqref{eq:xi1_expression}  that there exists $n_2 \geq n_1$
	such that $\xi_1(n) \geq 1$ for all $n \geq n_2$. If $n \geq n_2$, then
	\begin{align}
		\int_0^{\xi_1(n)}
		\frac{\varphi_q(\xi)}{(\xi+c)^{(\alpha+1)/2}} d\xi &=
		\int_0^{1}
		\frac{1}{(\xi+c)^{(\alpha+1)/2}} d\xi
		+ 
		\int_1^{\xi_1(n)}
		\frac{\xi^q}{(\xi+c)^{(\alpha+1)/2}} d\xi \notag \\
		&\leq 
		\frac{1}{c^{(\alpha+1)/2}}+ 
		\int_1^{\xi_1(n)}
		\frac{1}{(\xi+c)^{(\alpha-2q+1)/2}} d\xi.
		\label{eq:vp_g_estimate2}
	\end{align}
	Since $\xi_1(n)= O(n)$ as $n \to \infty$,
	it follows that 
	\begin{equation}
		\label{eq:vp_g_estimate3}
		\int_1^{\xi_1(n)}
		\frac{1}{(\xi+c)^{(\alpha-2q+1)/2}} d\xi
		=
		\begin{cases}
			O(n^{(1-\alpha+2q)/2}), & \alpha < 2q+1, \\
			O(\log n), & \alpha = 2q+1, \\
			O(1), & \alpha > 2q+1\\
		\end{cases}
	\end{equation}
	as $n \to \infty$.
	Combining \eqref{eq:vp_g_estimate1}--\eqref{eq:vp_g_estimate3},
	we obtain
	\begin{equation}
		\label{eq:xi1_below}
		\int_0^{\xi_1(n)} \varphi_q(\xi) \sup_{ s \geq 0}
		g_{n,\alpha,\omega}(\xi,x) d\xi =
		\begin{cases}
			O\left( \dfrac{1}{n^{\alpha-q}}\right), & \alpha < 2q+1, \vspace{3pt}\\
			O\left( \dfrac{\log n}{n^{q+1}}\right), & \alpha = 2q+1, \vspace{3pt}\\
			O\left( \dfrac{1}{n^{(\alpha+1)/2}}\right), & \alpha > 2q+1 \vspace{3pt}\\
		\end{cases}
	\end{equation}
	as $n \to \infty$.
	From  \eqref{eq:g_sup}, 
	we also have that 
	\begin{align*}
		\int_{\xi_1(n)}^{\infty} \varphi_q(\xi) \sup_{ s \geq 0}
		g_{n,\alpha,\omega}(\xi,x) d\xi &=
		\int_{\xi_1(n)}^{\infty} 
		\frac{\xi^q(\xi+c-\omega)^{n}}{(\xi+c+\omega)^{n+\alpha +1}} d\xi \\
		&\leq
		\int_{\xi_1(n)}^{\infty} \frac{1}{(\xi+c+\omega)^{\alpha-q+1}}d\xi
	\end{align*}
	for all $n \geq n_2$.
	This estimate and  the assumption $\alpha > q$ give
	\begin{equation}
		\label{eq:xi1_above}
		\int_{\xi_1(n)}^{\infty} \varphi_q(\xi) \sup_{ s \geq 0}
		g_{n,\alpha,\omega}(\xi,x) d\xi = O
		\left(
		\frac{1}{n^{\alpha - q}}
		\right)
	\end{equation}
	as $n \to \infty$. Thus, the estimate \eqref{eq:g_bound} is obtained
	from \eqref{eq:xi1_below} and  \eqref{eq:xi1_above}.
\end{proof}

We are now ready to prove
Proposition~\ref{prop:fnaw_bound_nonanalytic}.

\begin{proof}[Proof of Proposition~\ref{prop:fnaw_bound_nonanalytic}.]
	Define 
	\[
	\tilde \omega_{\min} \coloneqq \min 
	\left\{ \omega_{\min},\, \frac{c^2}{\omega_{\max}} \right\}. 
	\] 
	Then
	$\mathcal{S}(\omega_{\min},\omega_{\max}) \subset \mathcal{S}(\tilde \omega_{\min},\omega_{\max})$ 
	and 
	$c^2 \geq \tilde \omega_{\min} \omega_{\max}$. 
	The function $f_{n,\alpha,(\omega_k)}$ defined by \eqref{eq:fnaw_def_CP} satisfies
	\[
	f_{n,\alpha,(\omega_k)} =
	v_{\alpha}
	\tilde f_{n,\alpha,(\omega_k)}
	\]
	for all $(\omega_k)_{k \in \mathbb{N}} \in 
	\mathcal{S}(\omega_{\min},\omega_{\max})$,
	where 
	\begin{align*}
		v_{\alpha}(z) &\coloneqq
		\left(
		\frac{z+c+\tilde \omega_{\min}}{z+c+\omega_{\min}}
		\right)^{\alpha},\\
		\tilde f_{n,\alpha,(\omega_k)} (z) &\coloneqq 
		\frac{1}{(z+c+\tilde \omega_{\min})^{\alpha}} 
		\prod_{k=1}^n
		\frac{z+c-\omega_k}{z+c+\omega_k}
	\end{align*}
	for $z \in \mathbb{C}_+$.
	
	The derivative $\tilde f_{n,\alpha,(\omega_k)}'$ is given by
	\begin{align}
		\tilde f_{n,\alpha,(\omega_k)}'(z) 
		&=\frac{-\alpha}{(z+c+\tilde \omega_{\min})^{\alpha+1}}
		\prod_{k=1}^n
		\frac{z+c-\omega_k}{z+c+\omega_k} +
		\frac{2}{(z+c+\tilde \omega_{\min})^{\alpha}}
		\sum_{\ell=1}^n
		\frac{\omega_{\ell}}{(z+c+\omega_{\ell})^2}
		\prod_{k=1,\,k\not=\ell }^n
		\frac{z+c-\omega_k}{z+c+\omega_k}.
		\label{eq:f_deriv_expression}
	\end{align}
	Let $\xi >0$, $\eta \in \mathbb{R}$, and 
	$s \coloneqq \eta^2 \geq 0$.
	Define 
	$g_{n,\alpha,\tilde\omega_{\min}}$ and $h_{n,\alpha,\tilde\omega_{\min}}$ by \eqref{eq:gnaw_def} and \eqref{eq:hnaw_def} with
	$\omega=\tilde \omega_{\min}$, respectively.
	Then we obtain
	\begin{align*}
		\big|\tilde f_{n,\alpha,(\omega_k)}'(\xi+i \eta)\big| \leq 
		\alpha g_{n,\alpha,\tilde\omega_{\min}}(\xi,s)  + 2\omega_{\max} n h_{n,\alpha,\tilde\omega_{\min}}(\xi,s) 
	\end{align*}
	for all $(\omega_k)_{k \in \mathbb{N}} \in 
	\mathcal{S}(\tilde\omega_{\min},\omega_{\max})$,
	by using Lemma~\ref{lem:each_bound} with $s = \eta^2$.
	From Lemma~\ref{lem:g_h_int_bound}, we see that there exists
	$\tilde M>0$  such that 
	\[
	\bigvertiii{\tilde f_{n,\alpha,(\omega_k)}} \leq \tilde M F_{\alpha,q}(n)
	\]
	for all $n \in \mathbb{N}$ and $(\omega_k)_{k \in \mathbb{N}} \in 
	\mathcal{S}(\tilde \omega_{\min},\omega_{\max})$, where
	$F_{\alpha,q}$ is as in \eqref{eq:Faq_def}.
	
	Since 
	$\vertiii{v_{\alpha}} < \infty$ for $q \in (0,1)$ as shown in Example~\ref{ex:ra_B0q_norm},
	it follows from \eqref{eq:product_estimate} that
	\[
	\bigvertiii{f_{n,\alpha,(\omega_k)}} \leq 
	\big(
	\vertiii{v_{\alpha}} + \|v_{\alpha}\|_{\infty} 
	\big) \bigvertiii{\tilde f_{n,\alpha,(\omega_k)}}
	\]
	for all $(\omega_k)_{k \in \mathbb{N}} \in 
	\mathcal{S}(\omega_{\min},\omega_{\max})$.
	Thus, if we define 
	$M \coloneqq 
	(\vertiii{v_{\alpha}} + \|v_{\alpha}\|_{\infty} 
	)
	\tilde M $, then 
	the desired estimate \eqref{eq:fnaw_bound_nonanalytic}
	holds
	for all $n \in \mathbb{N}$ and $(\omega_k)_{k \in \mathbb{N}} \in 
	\mathcal{S}(\omega_{\min},\omega_{\max})$.
\end{proof}

We finally 
prove Theorem~\ref{thm:VA_bound_nonanalytic}, 
using the estimate \eqref{eq:fnaw_bound_nonanalytic} 
in the case $\alpha = 2q$.
\begin{proof}[Proof of Theorem~\ref{thm:VA_bound_nonanalytic}.]
	
	By assumption, there exists $c>0$ such that $-A+c$ 
	generates
	an exponentially stable $C_0$-semigroup on $H$.
	This $C_0$-semigroup
	is also in the Crandall--Pazy class
	with parameter $\beta$.
	Define $B \coloneqq A - c$, and
	let $q \in (0,1/2)$.
	The function $f_{n,2q,(\omega_k)}$ 
	defined by \eqref{eq:fnaw_def_CP}
	satisfies
	\begin{align}
		\label{eq:fB_V_CP}
		f_{n,2q,(\omega_k)}(B) &=
		\left(
		\prod_{k=1}^n
		V_{\omega_k}(A)
		\right) (A+ \omega_{\min} )^{-2q}
	\end{align}
	for all $n \in \mathbb{N}$ and $(\omega_k)_{k \in \mathbb{N}} \in
	\mathcal{S}(\omega_{\min},\omega_{\max})$.
	
	By Proposition~\ref{prop:fnaw_bound_nonanalytic},
	there exists $M_1 >0$ such that
	\begin{equation}
		\label{eq:fn2qo_bound1}
		\bigvertiii{f_{n,2q,(\omega_k)}} \leq M_1
		\frac{\log (n+1)}{n^{q}}
	\end{equation}
	for all $n \in \mathbb{N}$ and $(\omega_k)_{k \in \mathbb{N}} \in
	\mathcal{S}(\omega_{\min},\omega_{\max})$.
	From Proposition~\ref{prop:AbqfA_bound},
	we obtain 
	$B^{\beta q}f_{n,2q,(\omega_k)}(B)
	\in \mathcal{L}(H)$, and 
	there exists $M_2>0$ such that  
	\begin{equation}
		\label{eq:fn2qo_bound2}
		\big\|B^{\beta q} f_{n,2q,(\omega_k)}(B)\big\| \leq M_2 \bigvertiii{f_{n,2q,(\omega_k)}}
	\end{equation}
	for $n \in \mathbb{N}$ and $(\omega_k)_{k \in \mathbb{N}} \in
	\mathcal{S}(\omega_{\min},\omega_{\max})$.
	Since
	\[
	(A+\omega_{\min})^{2q} A^{-2q},\,
	(B+c)^{\beta q} B^{-\beta q}
	\in \mathcal{L}(H),
	\]
	it follows that 
	\[
	M_3 \coloneqq \|
	(B+c)^{\beta q} B^{-\beta q} (A+\omega_{\min})^{2q} A^{-2q}
	\| < \infty.
	\]
	From
	\eqref{eq:fB_V_CP}, we obtain
	\begin{align}
		\left\|
		\left(
		\prod_{k=1}^n
		V_{\omega_k}(A)
		\right)A^{-(2-\beta)q}\right\| 
		&\leq M_3 \big\|B^{\beta q} f_{n,2q,(\omega_k)}(B)\big\|
		\label{eq:fn2qo_bound3}
	\end{align}
	for all $n \in \mathbb{N}$ 
	and $(\omega_k)_{k \in \mathbb{N}} \in
	\mathcal{S}(\omega_{\min},\omega_{\max})$.
	By
	the estimates \eqref{eq:fn2qo_bound1}--\eqref{eq:fn2qo_bound3},
	\begin{equation}
		\label{eq:V_bound_ell0}
		\left\|
		\left(
		\prod_{k=1}^n
		V_{\omega_k}(A)
		\right)A^{-(2-\beta)q }
		\right\| \leq M_4 \frac{\log (n+1)}{n^q}
	\end{equation}
	for all $n \in \mathbb{N}$ and $(\omega_k)_{k \in \mathbb{N}} \in
	\mathcal{S}(\omega_{\min},\omega_{\max})$,
	where $M_4 \coloneqq M_1M_2M_3$.
	
	Let $0 < \alpha < (2-\beta)/2$, and set $q \coloneqq \alpha/(2-\beta)$.
	Then $0 < q < 1/2$ and $(2-\beta)q = \alpha$.
	By the definition \eqref{eq:Labe} of $L_{\alpha,\beta,\varepsilon}$,
	\[
	\frac{\log (n+1)}{n^q} = \frac{
		L_{\alpha,\beta,\varepsilon}(n)}{n^{\alpha/(2-\beta)}}
	\] 
	for all $n \in \mathbb{N}$.
	Hence, we conclude from \eqref{eq:V_bound_ell0}
	that the desired estimate
	\eqref{eq:VA_bound_nonanalytic} holds 
	in the case $0 < \alpha < (2-\beta)/2$.

	Take 
	$\alpha \geq (2-\beta)/2$ arbitrarily. Let  $\ell \in \mathbb{N}$ and
	$0 \leq \vartheta_0 < (2-\beta)/2$
	satisfy
	\[
	\alpha = \frac{\ell(2-\beta)}{2} + \vartheta_0.
	\]
	Choose $\delta  >0$ small enough such that
	\[
	0 <\vartheta  \coloneqq \vartheta_0 + \frac{\delta \ell}{2} < \frac{2-\beta}{2}.
	\]
	By definition, 
	\[
	\alpha = \ell \gamma+ \vartheta, \quad \text{where~}
	\gamma  \coloneqq \frac{2-\beta - \delta }{2} > 0.
	\]
	Define $\tau(n) \coloneqq \lfloor n/(\ell+1)\rfloor$ for $n \in \mathbb{N}$.
	For $n \in \mathbb{N}$ and $(\omega_k)_{k \in \mathbb{N}} \in
	\mathcal{S}(\omega_{\min},\omega_{\max})$, we obtain
	\[
	\left( 
	\prod_{k=1}^n
	V_{\omega_k}(A)
	\right)A^{-\alpha} =
	\prod_{j=0}^{\ell+1}
	W_{j,n,(\omega_k)},
	\]
	where
	\begin{align*}
		W_{0,n,(\omega_k)} &\coloneqq 
		\left( \prod_{k=1}^{\tau(n)}
		V_{\omega_k}(A)
		\right)A^{-\vartheta}, \\
		W_{j,n,(\omega_k)} &\coloneqq 
		\left( \prod_{k=j  \tau(n)+1}^{(j + 1) \tau(n)}
		V_{\omega_k}(A)
		\right)A^{-\gamma}, \quad j=1,2,\dots,\ell,\\
		W_{\ell+1,n,(\omega_k)} &\coloneqq 	\prod_{k=(\ell+1)\tau(n)+1}^{n}
		V_{\omega_k}(A). 
	\end{align*}
	Here, for $n \in \mathbb{N}$
	the empty product 
	$\prod_{k=n}^{n-1} V_{\omega_k}(A)$ is interpreted as the identity operator.
	Note that  
	\[
	\gamma, \vartheta < \frac{2-\beta}{2}.
	\]
	By the estimate \eqref{eq:VA_bound_nonanalytic} in the case $0 < \alpha < (2-\beta)/2$,
	there exists $M_5>0$
	such that for all $n \in \mathbb{N}$ and $(\omega_k)_{k \in \mathbb{N}} \in
	\mathcal{S}(\omega_{\min},\omega_{\max})$,
	\begin{align*}
		\big\|
		W_{0,n,(\omega_k)}
		\big\|\leq 
		M_5
		\frac{\log (n+1)}{n^{\vartheta/(2-\beta)} } 
	\end{align*}
	and
	\begin{align*}
		\big\|
		W_{j,n,(\omega_k)}
		\big\|\leq 
		M_5
		\frac{\log (n+1)}{n^{\gamma/(2-\beta)}},\quad j=1,2,\dots,\ell.
	\end{align*}
	Moreover,
	\[
	\big\|
	W_{\ell+1,n,(\omega_k)}
	\big\| \leq 
	\max \big\{ 1,\,
	\max\{
	\|V_{\omega}(A)\|^\ell: \omega_{\min} \leq \omega \leq \omega_{\max} \} \big\}
	\]
	for all $n \in \mathbb{N}$ and $(\omega_k)_{k \in \mathbb{N}} \in
	\mathcal{S}(\omega_{\min},\omega_{\max})$. Hence,
	there exists $M_6>0$ such that
	\begin{equation}
		\label{eq:VA_bound_nonanalytic_ell}
		\left\|
		\left(
		\prod_{k=1}^n
		V_{\omega_k}(A)
		\right)
		A^{-\alpha}
		\right\| 
		\leq M_6\frac{\big(\log (n+1)\big)^{\ell +1}}{ n^{\alpha/(2-\beta)}},
		\quad \text{where~}
		\ell = \left\lfloor
		\frac{2\alpha}{2-\beta}
		\right\rfloor,
	\end{equation}
	for all $n \in \mathbb{N}$ and 
	$(\omega_k)_{k \in \mathbb{N}} \in \mathcal{S}(\omega_{\min},\omega_{\max})$.
	
	Let $\varepsilon >0$ be given.
	If $\delta>0$ is sufficiently small, then $\gamma = (2-\beta - \delta)/2$
	satisfies
	\begin{equation}
		\label{eq:delta_eps}
		\frac{\alpha}{\gamma} =
		\frac{2\alpha}{2-\beta - \delta} \leq 
		\frac{2\alpha}{2-\beta} + \varepsilon
	\end{equation}
	and 
	\begin{equation}
		\label{eq:ell_gamma_alpha}
		(\ell-1) \frac{2-\beta}{2} 
		\leq
		\ell \gamma <
		\alpha < (\ell + 1) \gamma.
	\end{equation}
	By the moment inequality given in Lemma~\ref{lem:moment_ineq}, there exists $C>0$
	such that 
	\begin{align}
		\label{eq:appli_MI}
		\left\|
		A^{-\alpha}
		\prod_{k=1}^n
		V_{\omega_k}(A)
		\right\| 
		&\leq 
		C
		\left\|
		A^{-\ell \gamma}
		\prod_{k=1}^n
		V_{\omega_k}(A)
		\right\|^{\ell+1- \alpha/\gamma}
		\left\|
		A^{-(\ell+1)\gamma}
		\prod_{k=1}^n
		V_{\omega_k}(A)
		\right\|^{\alpha/\gamma - \ell } 
	\end{align}
	for all $n \in \mathbb{N}$ and 
	$(\omega_k)_{k \in \mathbb{N}} \in \mathcal{S}(\omega_{\min},\omega_{\max})$.
	We see from \eqref{eq:ell_gamma_alpha} that 
	\[
	\left\lfloor
	\frac{2\ell \gamma}{2-\beta}
	\right\rfloor = \ell-1\quad 
	\text{and}\quad 
	\left\lfloor
	\frac{2(\ell+1) \gamma}{2-\beta}
	\right\rfloor = \ell.
	\]
	By the estimates \eqref{eq:VA_bound_nonanalytic_ell} and
	\eqref{eq:appli_MI}, 
	there exists $M_7>0$ such that 
	\begin{align}
		\label{eq:A_V_log_n}
		\left\|
		A^{-\alpha}
		\prod_{k=1}^n
		V_{\omega_k}(A)
		\right\| 
		&\leq 
		M_7\frac{\big(\log (n+1) \big)^{\alpha/\gamma}}{n^{\alpha/(2-\beta)}} 
	\end{align}
	for all $n \in \mathbb{N}$ and 
	$(\omega_k)_{k \in \mathbb{N}}\in \mathcal{S}(\omega_{\min},\omega_{\max})$.
	By \eqref{eq:delta_eps},
	\[
	\big(\log (n+1) \big)^{\alpha/\gamma} \leq 
	\big(\log (n+1) \big)^{2\alpha/(2-\beta) + \varepsilon} = L_{\alpha,\beta,\varepsilon}(n)
	\]
	for all $n \geq 2$.
	The desired estimate \eqref {eq:VA_bound_nonanalytic}
	in the case $\alpha \geq (2-\beta)/2$
	follows from \eqref{eq:A_V_log_n}.
\end{proof}

\begin{remark}
	\label{rem:other_a}
		In the proof of 
		Theorem~\ref{thm:VA_bound_nonanalytic},
		we used the 
		estimate \eqref{eq:fnaw_bound_nonanalytic}
		in the case $\alpha = 2	q$.
		When we apply \eqref{eq:fnaw_bound_nonanalytic}
		in the case $q<\alpha < 2 q$, the following 
		estimate is obtained by the same argument:
		There exists $M_1>0$ such that 
		\begin{equation}
			\label{eq:a_small}
			\left\|
			\left(
			\prod_{k=1}^n
			V_{\omega_k}(A)
			\right)
			A^{-1}
			\right\| 
			\leq \frac{M_1}{n^{(p-1)/(p-\beta)}}
		\end{equation}
		for all $n \in \mathbb{N}$ and 
		$(\omega_k)_{k \in \mathbb{N}} \in \mathcal{S}(\omega_{\min},\omega_{\max})$,
		where $p \coloneqq 
		\alpha /q \in (1,2)$.  
		If $\beta = 1$, then the upper bound for the decay rate given in 
		\eqref{eq:a_small}
		coincides with the one in
		Theorem~\ref{thm:CN_holomorphic} 
		for holomorphic $C_0$-semigroups.
		However, if $\beta \in (0,1)$, then
		\[
		\frac{1}{2-\beta} > \frac{p-1}{p-\beta}
		\]
		for all $p \in (1,2)$, and 
		the estimate \eqref{eq:a_small}
		is less sharp than the estimate 
		\eqref{eq:VA_bound_nonanalytic}.
		Similarly,  \eqref{eq:fnaw_bound_nonanalytic}
		in the case $\alpha > 2	q$ leads to the following estimate:
		There exists $M_2>0$ such that 
		\begin{equation}
			\label{eq:a_large}
			\left\|
			\left(
			\prod_{k=1}^n
			V_{\omega_k}(A)
			\right)
			A^{-1}
			\right\| 
			\leq \frac{M_2}{n^{p/(2p-2\beta)}}
		\end{equation}
		for all $n \in \mathbb{N}$ and 
		$(\omega_k)_{k \in \mathbb{N}} \in \mathcal{S}(\omega_{\min},\omega_{\max})$,
		where $p \coloneqq 
		\alpha /q >2$.  Since
		\[
		\frac{1}{2-\beta} > \frac{p}{2p-2\beta}
		\]
		for all $p > 2$ and $\beta \in (0,1]$,
		the estimate \eqref{eq:a_large}
		is also worse than the estimate 
		\eqref{eq:VA_bound_nonanalytic}
		for all $\beta \in (0,1]$.
\end{remark}

\subsection{Decay estimate for inverse generators}
Let $-A$ be the generator of a bounded $C_0$-semigroup
$(e^{-tA})_{t \geq 0}$ on a Banach space.
Assume that the algebraic inverse $A^{-1}$ exists and is densely defined.
The inverse generator problem, introduced in \cite{deLaubenfels1988}, asks whether
$-A^{-1}$ also generates a bounded 
$C_0$-semigroup.
The inverse generator $-A^{-1}$ naturally arises when applying the reciprocal transform in system theory; see,
e.g., \cite{Curtain2002} and \cite[Section~12.4]{Staffans2005}.
The $C_0$-semigroup $(e^{-tA^{-1}})_{t \geq 0}$, if it exists,
has numerous properties analogous to
the Cayley transform $V_\omega(A)$.
We review some existing studies
on the inverse generator problem in the next paragraph, and a more detailed discussion can be found in the survey \cite{Gomilko2017}.

The inverse generator problem has a negative answer
in the general Banach space setting, which was implicitly
given in \cite[pp.~343--344]{Komatsu1966}; see also \cite{Zwart2007,Gomilko2007MS, Fackler2016}.
However, positive results were given for
sectorially bounded holomorphic $C_0$-semigroups
on Banach spaces \cite{deLaubenfels1988} and for contraction $C_0$-semigroups
on Hilbert spaces \cite{Zwart2007}. 
It remains unknown whether the answer is positive or negative
for
bounded $C_0$-semigroups on Hilbert spaces.
The estimate 
$\|e^{-tA^{-1}}\| = O(\log t)$ as $t \to \infty$  
was first established for 
exponentially stable $C_0$-semigroups on Hilbert spaces in \cite{Zwart2007SF}. It was later extended to bounded $C_0$-semigroups $(e^{-tA})_{t \geq 0}$ on Hilbert spaces with $0 \in \varrho(A)$ in \cite{Batty2021}.
It was  shown in \cite{Wakaiki2024JEE} that 
the same class of $C_0$-semigroups satisfy
$\sup_{t \geq 0}\|e^{-tA^{-1}}A^{-\alpha}\| < \infty$
for all $\alpha >0$.
The generator $-A$ of a sectorially
bounded holomorphic 
$C_0$-semigroup
on a Banach space satisfies
\begin{equation}
	\label{eq:inv_estimate_holomorphic}
	\big\|A^{-\alpha}e^{-tA^{-1}}\big\| = 
	O\left( \frac{1}{t^{\alpha}} \right)\quad \text{as $t \to \infty$}
\end{equation}
for all $\alpha >0$ if $A^{-1}$ exists and is densely defined;
see, e.g., \cite[Corollary~10.1]{Batty2023}.
For an exponentially stable $C_0$-semigroup $(e^{-tA})_{t \geq 0}$
on a Banach space, 
\begin{equation}
	\label{eq:inv_estimate_exp_stable_Banach}
	\big\|e^{-tA^{-1}}A^{-m}\big\| =  O\left( \frac{1}{t^{m/2-1/4}}\right)\quad \text{as $t \to \infty$}
\end{equation}
holds,
where $m \in \mathbb{N}$. The estimate 
\eqref{eq:inv_estimate_exp_stable_Banach} 
with $m=1$ was initially
derived  in
\cite{Zwart2007} and subsequently generalized to all
$m \in \mathbb{N}$ in
\cite{deLaubenfels2009}.
In the Hilbert space setting,
the estimate \eqref{eq:inv_estimate_exp_stable_Banach} 
was improved to
\begin{equation}
	\label{eq:inv_estimate_exp_stable}
	\big\|e^{-tA^{-1}}A^{-\alpha}\big\| = O\left( \frac{1}{t^{\alpha/2}} \right)\quad \text{as $t \to \infty$}
\end{equation}
for $\alpha >0$ in
\cite{Wakaiki2023IEOT}.

We present a norm estimate for $(e^{-tA^{-1}})_{t\geq 0}$
when $(e^{-tA})_{t\geq 0}$ is 
an exponentially stable $C_0$-semigroup in the Crandall--Pazy class 
on a Hilbert space.
The following theorem is analogous to Theorem~\ref{thm:VA_bound_nonanalytic} and bridges the gap between
\eqref{eq:inv_estimate_holomorphic} and 
\eqref{eq:inv_estimate_exp_stable}.
\begin{theorem}
	\label{thm:CP_inv}
	Let $-A$ be the generator of 
	an exponentially stable $C_0$-semigroup $(e^{-tA})_{t \geq 0}$
	in the Crandall--Pazy class with parameter $\beta \in (0,1]$
	on a Hilbert space $H$.
	Then, for all $\alpha,\varepsilon>0$,
	\begin{equation}
		\label{eq:inv_bound_nonanalytic}
		\big\|
		e^{-tA^{-1}}
		A^{-\alpha}
		\big\| 
		= O
		\left( \frac{L_{\alpha,\beta, \varepsilon}(t)}{ t^{\alpha/(2-\beta)}}\right)
		\quad \text{as $t \to \infty$},
	\end{equation}
	where $L_{\alpha,\beta, \varepsilon}$ is as in \eqref{eq:Labe}.
\end{theorem}

Theorem~\ref{thm:CP_inv} is also proved via
the $\mathcal{B}$-calculus. 
Let $t \geq 0$ and $\alpha > 0$. Define
\begin{equation}
	\label{eq:fta_def}
	f_{t,\alpha}(z) \coloneqq \frac{e^{-t/(z+1)}}{(z+1)^\alpha},\quad z \in \mathbb{C}_+.
\end{equation}
Then $f_{t,\alpha} \in \mathcal{B}$ for all $t,\alpha >0$, and 
$\|f_{t,\alpha}\|_{\mathcal{B}_0} = O(t^{-\alpha/2})$ as $t \to \infty$ 
for 
all $\alpha >0$; see \cite[Lemma~2.3]{Wakaiki2023IEOT}.
Here we use the following estimate for $\vertiii{f_{t,\alpha}}$ 
with $\alpha > q > 0$.

\begin{proposition}
	\label{prop:fta_norm}
	Let $\alpha > q > 0$. Then  
	the function $f_{t,\alpha}$ defined by \eqref{eq:fta_def}
	satisfies
	$\sup_{t \geq 0}
	\vertiii{f_{t,\alpha}} < \infty$ and
	\begin{align}
		\label{eq:f_t_B0_norm}
		\vertiii{f_{t,\alpha}} 
		= 		\begin{cases}
			O\left(\dfrac{1}{t^{\alpha-q}}\right), & q<\alpha < 2q, \vspace{6pt}\\
			O\left(\dfrac{\log t}{t^{q}}\right), & \alpha = 2q, \vspace{6pt}\\
			O\left(\dfrac{1}{t^{\alpha/2}}\right), & \alpha >2q
		\end{cases}
	\end{align}
	as $t \to \infty$.
\end{proposition}
\begin{proof}
	Let $\alpha > q > 0$  and $\beta > q + 1$. Define 
	$F_{\beta}$ by
	\[
	F_\beta (t) \coloneqq 
	\int^{\infty}_0 \varphi_q(\xi) \sup_{\eta \in \mathbb{R}}
	\left|\frac{e^{-t/(\xi+i\eta+1)}}{(\xi+i\eta+1)^\beta}\right| d\xi,
	\quad t \geq 0,
	\]
	where $\varphi_q$ is as in \eqref{eq:psi_def}.
	Since
	\[
	f_{t,\alpha}'(z) = \frac{te^{-t/(z+1)} }{(z+1)^{\alpha+2}} - \frac{\alpha e^{-t/(z+1)} }{(z+1)^{\alpha+1}}
	\]
	for all $z \in \mathbb{C}_+$,
	we obtain
	\begin{equation}
		\label{eq:ftk_bound}
		\vertiii{f_{t,\alpha}} 
		\leq tF_{\alpha+2}(t) + \alpha F_{\alpha+1}(t)
	\end{equation}
	for all $t\geq  0$.
	
	We will prove that 
	\begin{equation}
		\label{eq:F_b_bound}
		\sup_{t \geq 0}F_\beta (t) < \infty
	\end{equation}
	and 
	\begin{equation}
		\label{eq:Fk_order}
		F_\beta (t) =
		\begin{cases}
			O\left(\dfrac{1}{t^{\beta-q-1}}\right), & q+1<\beta < 2q+2, \vspace{6pt}\\
			O\left(\dfrac{\log t}{t^{q+1}}\right), & \beta = 2q+2, \vspace{6pt}\\
			O\left(\dfrac{1}{t^{\beta/2}}\right), & \beta  >2q+2
		\end{cases}
	\end{equation}
	as $t \to \infty$.
	From \eqref{eq:Fk_order}, it follows that 
	\begin{align*}
		tF_{\alpha+2}(t) + \alpha F_{\alpha+1}(t) =
		\begin{cases}
			O\left(\dfrac{1}{t^{\alpha-q}}\right), & q<\alpha < 2q, \vspace{6pt}\\
			O\left(\dfrac{\log t}{t^{q}}\right), & \alpha = 2q, \vspace{6pt}\\
			O\left(\dfrac{1}{t^{\alpha/2}}\right), & \alpha >2q
		\end{cases}
	\end{align*}
	as $t \to \infty$.
	Hence, we obtain $\sup_{t \geq 0}\vertiii{f_{t,\alpha}} < \infty$ 
	and the estimate \eqref{eq:f_t_B0_norm}.

	To prove the estimates  \eqref{eq:F_b_bound} 
	and \eqref{eq:Fk_order}, we
	define 
	\[
	g_{t,\zeta ,\beta}(s) \coloneqq \frac{e^{-t\zeta /(\zeta ^2+s)}}{(\zeta ^2+s)^{\beta/2}}
	\]
	for $t>0$, $\zeta  >1$, and $s \geq 0$.
	Then, letting $\zeta = \xi+1$ and $s = \eta^2$, we obtain
	\begin{equation}
		\label{eq:Fbeta_rep}
		F_{\beta}(t)
		=
		\int^{\infty}_1 \varphi_q (\zeta -1)\sup_{s \geq 0} g_{t,\zeta ,\beta}(s)d\zeta  
	\end{equation}
	for all $t \geq  0$.
	A routine calculation shows that 
	\begin{equation}
		\label{eq:gtzb_bound}
		\sup_{s \geq 0} g_{t,\zeta ,\beta}(s) = 
		\begin{cases}
			g_{t,\zeta ,\beta}(0) = \dfrac{e^{-t/\zeta }}{\zeta ^\beta}
			\leq \dfrac{1}{\zeta^{\beta}}, & 0 \leq t \leq \dfrac{\beta \zeta }{2}, \vspace{6pt}\\
			g_{t,\zeta ,\beta}\left(\dfrac{2t\zeta }{\beta}-\zeta ^2 \right) = \dfrac{c}{(t\zeta )^{\beta/2}}, &  t \geq \dfrac{\beta \zeta }{2} > \dfrac{\beta}{2}
		\end{cases}
	\end{equation}
	for all $\zeta  > 1$,
	where $c \coloneqq  \left(\beta/(2e)\right)^{\beta/2}$.
	To prove \eqref{eq:F_b_bound}, we consider 
	the three cases $0 \leq  t \leq \beta/2$,
	$\beta/2 <  t \leq \beta$, and $t >\beta$.
	In the last case $t > \beta$, we also prove \eqref{eq:Fk_order}.
	
	If $0 \leq  t \leq \beta/2$, then \eqref{eq:gtzb_bound} yields
	\begin{align*}
		\int^{\infty}_1 \varphi_q (\zeta -1) \sup_{s \geq 0}  g_{t,\zeta ,\beta}(s)d\zeta  &\leq
		\int^{\infty}_1 \frac{\varphi_q(\zeta -1) }{\zeta ^\beta}d\zeta  \\
		&\leq
		\int^{2}_1 \frac{1}{\zeta ^\beta}d\zeta  + \int^{\infty}_2 \frac{1}{\zeta ^{\beta-q}}d\zeta \\
		&\leq 
		1 + \frac{1}{(\beta -q-1)2^{\beta-q-1}} \eqqcolon M_1.
	\end{align*}
	Combining this with \eqref{eq:Fbeta_rep},
	we obtain $F_{\beta}(t) \leq M_1$ for all $t \in [0,\beta/2]$.
	
	When $t > \beta /2$, we have from \eqref{eq:gtzb_bound}  that
	\begin{align}
		\label{eq:phi_g_bound}
		\int^{\infty}_1 \varphi_q (\zeta -1) \sup_{s \geq 0} g_{t,\zeta ,\beta}(s)d\zeta  &\leq 
		\frac{c}{t^{\beta /2}} \int^{2t/\beta}_{1} \frac{\varphi_q (\zeta -1) }{\zeta ^{\beta/2}}d\zeta  + 
		\int^{\infty}_{2t/\beta }  \frac{\varphi_q (\zeta -1)}{\zeta ^\beta } d\zeta .
	\end{align}
	If $\beta/2 < t \leq \beta$, then $1 < 2t/\beta \leq 2$.
	Hence,
	\[
	\int^{2t/\beta}_{1} \frac{\varphi_q (\zeta -1) }{\zeta ^{\beta/2}}d\zeta  =
	\int^{2t/\beta}_{1} \frac{1}{\zeta ^{\beta/2}}d\zeta \leq 1
	\]
	and
	\[
	\int^{\infty}_{2t/\beta }  \frac{\varphi_q (\zeta -1)}{\zeta ^\beta } d\zeta  \leq
	\int^{2}_{2t/\beta } \frac{1}{\zeta ^\beta } d\zeta  + 
	\int^{\infty}_{2 } \frac{1}{\zeta ^{\beta-q} } d\zeta  \leq M_1.
	\]
	From \eqref{eq:Fbeta_rep} and \eqref{eq:phi_g_bound}, it follows that 
	\[
	F_{\beta}(t) \leq 
	\frac{c}{t^{\beta /2}} + M_1
	\]
	for all $t \in (\beta/2,\beta]$.
	
	Finally, we consider the case $t > \beta$. Since $2t/\beta > 2$,
	we have that 
	\begin{align*}
		\int^{2t/\beta}_{1} \frac{\varphi_q (\zeta -1) }{\zeta ^{\beta/2}}d\zeta  \leq 1   + 
		\int^{2t/\beta}_2\frac{1}{\zeta ^{\beta/2-q}}d\zeta .
	\end{align*}
	The second term of the right-hand side satisfies
	\[
	\int^{2t/\beta}_2\frac{1}{\zeta ^{\beta/2-q}}d\zeta  =
	\begin{cases}
		\log t - \log \beta,& \beta = 2q+2, \\
		\dfrac{2}{\beta - 2q - 2}
		\left(
		\left(
		\dfrac{1}{2}
		\right)^{\beta/2-q-1} - 
		\left(
		\dfrac{\beta}{2t}
		\right)^{\beta/2-q-1}
		\right),& \beta \not= 2q+2.
	\end{cases}
	\]
	Moreover,
	\begin{align*}
		\int^{\infty}_{2t/\beta } \frac{\varphi_q (\zeta -1) }{\zeta ^\beta } d\zeta  \leq 
		\int^{\infty}_{2t/\beta } \frac{1}{\zeta ^{\beta-q} } d\zeta  
		= \frac{1}{\beta -q-1} \left(
		\dfrac{\beta}{2t}
		\right)^{\beta-q-1}.
	\end{align*}
	Using \eqref{eq:Fbeta_rep} and \eqref{eq:phi_g_bound},
	we obtain $\sup_{t > \beta }F_{\beta}(t) < \infty$ and 
	\eqref{eq:Fk_order}.
\end{proof}

We are now in a position to prove Theorem~\ref{thm:CP_inv}.
In the proof, 
the function $f_{t,\alpha}$ defined by \eqref{eq:fta_def} is decomposed as
$
f_{t,\alpha} = h_t w_{\alpha},
$
where
\begin{equation}
	\label{eq:h_r_def}
	h_t(z) \coloneqq e^{-t/(z+1)}\quad \text{and} \quad  
	w_\alpha(z) \coloneqq \frac{1}{(z+1)^{\alpha}},\quad z \in \mathbb{C}_+.
\end{equation}
We have that 
$h_t\in \mathcal{LM}$,
as shown in \cite[Example 2.12]{Batty2021}.
Moreover,
$w_{\alpha} \in \mathcal{LM}$; see Example~\ref{ex:r_alpha_Ds}.

\begin{proof}[Proof of Theorem~\ref{thm:CP_inv}.]
	There exist $K \geq 1$ and $c >0$ such that 
	$
	\|e^{-tA} \| \leq Ke^{-c t}
	$
	for all $t \geq 0$.
	We may assume that $c = 2$ by replacing $A$ with $(2/c)A$
	and $t$ with $(c/2) t$. 
	Define $B\coloneqq A-1$. Then 
	$-B$ also generates an exponentially stable
	$C_0$-semigroup  in 
	the Crandall--Pazy class with parameter $\beta$ on $H$.

	Let $t\geq 0$ and $0 < q < 1/2$.
	Define $h_t$ and $w_{2q}$ by \eqref{eq:h_r_def}.
	Then
	$
	h_t(B)= e^{-tA^{-1}}
	$
	as shown in the proof of \cite[Corollary 5.7]{Batty2021},
	and $w_{2q}(B) = A^{-2q}$.
	Hence,
	the function $f_{t,2q}$ defined by \eqref{eq:fta_def} 
	satisfies
	\begin{equation}
		\label{eq:f_t_2q_fc}
		f_{t,2q}(B)  = h_t(B)w_{2q}(B) = e^{-tA^{-1}}A^{-2q}.
	\end{equation}
	Proposition~\ref{prop:fta_norm} yields
	$\sup_{t \geq 0} \vertiii{f_{t,2q}} < \infty$
	and 
	\begin{equation}
		\label{eq:inv_A_bound1}
		\vertiii{f_{t,2q}} = O
		\left(
		\frac{\log t}{t^q}
		\right)
	\end{equation}
	as $t\to \infty$.
	By Proposition~\ref{prop:AbqfA_bound},
	$B^{\beta q}f_{t,2q}(B)\in \mathcal{L}(X)$, and there exists $M_1>0$ such that 
	\begin{equation}
		\label{eq:inv_A_bound2}
		\|B^{\beta q} f_{t,2q}(B) \| \leq M_1 \vertiii{f_{t,2q}}.
	\end{equation}
	Since $A^{\beta q}B^{-\beta q} \in \mathcal{L}(H)$,
	it follows that $M_2 \coloneqq \|A^{\beta q}B^{-\beta q}\| < \infty$.
	We have from \eqref{eq:f_t_2q_fc} that
	\begin{equation}
		\label{eq:inv_A_bound3}
		\big\|e^{-tA^{-1}}A^{-(2 - \beta) q}\big\| 
		\leq M_2\|B^{\beta q}f_{t,2q}(B) \|.	
	\end{equation}
	Combining the estimates 
	\eqref{eq:inv_A_bound1}--\eqref{eq:inv_A_bound3},
	we obtain
	\begin{equation}
		\label{eq:inv_A_bound_log_t}
		\big\|e^{-tA^{-1}}A^{-(2-\beta)q}\big\| =
		O
		\left( \frac{\log t}{ t^q}\right)
	\end{equation}
	as $t \to \infty$. 
	
	Let $0 < \alpha <(2-\beta)/2$, and set $q \coloneqq \alpha/(2-
	\beta)$.
	Then $0 < q < 1/2$ and $(2-\beta)q = \alpha$.
	From the definition \eqref{eq:Labe} of $L_{\alpha,\beta,\varepsilon}$,
	we also have that 
	\[
	\frac{\log (t+1)}{t^q} = \frac{
		L_{\alpha,\beta,\varepsilon}(t)}{t^{\alpha/(2-\beta)}}
	\]
	for all $t >0$.
	Therefore, \eqref{eq:inv_A_bound_log_t} yields the desired 
	estimate \eqref{eq:inv_bound_nonanalytic}
	in the case $0< \alpha <  (2-\beta)/2$.
	The case $\alpha \geq   (2-\beta)/2$ 
	follows as in the proof of Theorem \ref{thm:VA_bound_nonanalytic}.
\end{proof}

\begin{remark}
		To prove Theorem~\ref{thm:CP_inv},
		we used the 
		estimate \eqref{eq:f_t_B0_norm}
		in the case $\alpha = 2	q$.
		The case $\alpha \neq 2q$ does not lead to
		sharper decay estimates for inverse generators, as in Remark~\ref{rem:other_a} for
		the Crank--Nicolson scheme.
\end{remark}

\subsection{Example}
\label{sec:example}
The following simple example indicates that 
the decay rate $L_{\alpha,\beta, \varepsilon}(\tau)\tau^{-\alpha/(2-\beta)}$ given in Theorems~\ref{thm:VA_bound_nonanalytic}
and \ref{thm:CP_inv} cannot in general be replaced by any better rates than
$\tau^{-\alpha/(2-\beta)}$.

Let $\gamma \geq 1$ 
and set $\lambda_k \coloneqq k + i k^{\gamma}$
for $k \in \mathbb{N}$.
Define an operator $A$ on $\ell^2(\mathbb{N})$ by
\[
(Ax)_{k \in \mathbb{N}} \coloneqq \big((\lambda_k+1) x_k\big)_{k \in \mathbb{N}},\quad 
x = (x_k)_{k \in \mathbb{N}}
\]
with domain 
\[
D(A) \coloneqq 
\big\{
(x_k
)_{k \in \mathbb{N}} \in \ell^2(\mathbb{N}):
\big( (\lambda_k+1) x_k\big)_{k \in \mathbb{N}} \in \ell^2(\mathbb{N})
\big\}.
\]
Then $-A$ generates an exponentially stable $C_0$-semigroup
$(e^{-tA})_{t \geq 0}$
on $\ell^2(\mathbb{N})$.
Moreover, $(e^{-tA})_{t \geq 0}$ is in the Crandall--Pazy class with parameter $\beta \coloneqq 
1/\gamma \in (0,1]$,
but not with any larger parameter.
Indeed, let $k \in \mathbb{N}$ and $t>0$.
We have
\[
|(\lambda_k+1)e^{-(\lambda_k+1)t}| =
\sqrt{(k+1)^2 + k^{2\gamma}} e^{-(k+1)t} \leq 
\sqrt{5} k^{\gamma}e^{-kt}.
\]
Define $f_t(s) \coloneqq s^{\gamma} e^{-st}$ for $s \geq 0$.
Then
\[
\sup_{s \geq 0} f_t(s) = f_t\left(
\frac{\gamma}{t}
\right) = \left(
\frac{\gamma}{et}
\right)^{\gamma} .
\]
Therefore, $e^{-tA} H \subset D(A)$, and furthermore,
\[
\|Ae^{-tA}\| =\sup_{k \in \mathbb{N}} |(\lambda_k+1)e^{-(\lambda_k+1)t}|
\leq \sqrt{5} \left(
\frac{\gamma}{et}
\right)^{\gamma},
\]
which implies that $(e^{-tA})_{t \geq 0}$ is in the Crandall--Pazy class with parameter $\beta=1/\gamma$.
To show that 
the parameter $\beta$ cannot be replaced with a larger constant,
let $n \in \mathbb{N}$ and $t = 1/n$. Then
\[
\sup_{k \in \mathbb{N}} \sqrt{(k+1)^2 + k^{2\gamma}} e^{-(k+1)t}
\geq \sqrt{(n+1)^2 + n^{2\gamma}} e^{-(n+1)/n} \geq  e^{-(n+1)/n} t^{-\gamma}.
\]
Hence,
\[
\limsup_{t \to 0+} t^{1/\beta}\|Ae^{-tA}\| \geq e^{-1}.
\]
This implies that 
$(e^{-tA})_{t \geq 0}$ is not in the Crandall--Pazy class with a parameter larger than $\beta$. 

\subsubsection{Crank--Nicolson scheme}
Let $\alpha \geq 0$, and assume that $2\gamma-1 \in \mathbb{N}$.
For all $n \in \mathbb{N}_0$,
\begin{equation}
	\label{eq:V_A_example}
	\|V_1(A)^n(1+A)^{-\alpha}\| =
	\sup_{k \in \mathbb{N}} 
	\left| \frac{\lambda_k}{\lambda_k+2 } \right|^n
	\,
	\frac{1}{|\lambda_k+2|^{\alpha}}.
\end{equation}
For each $k \in \mathbb{N}$,
\[
\left| \frac{\lambda_k}{\lambda_k+2 } \right| =
\left(
1 - \frac{4(k+1)}{(k+2)^2+k^{2\gamma}}
\right)^{1/2} =
\left(
1 - \frac{4(1+1/k)}{(k+2)^2/k+k^{2\gamma-1}}
\right)^{1/2}.
\]
Letting $m \coloneqq 2\gamma - 1 \in \mathbb{N}$,
we have that for all $n \in \mathbb{N}$,
\[
\sup_{k \in \mathbb{N}} 
\left| \frac{\lambda_k}{\lambda_k+2 } \right|^{n^m}
\,
\frac{1}{|\lambda_k+2|^{\alpha}} \geq 
\left(
1 - \frac{4(1+1/n)}{(n+2)^2/n+n^{m}}
\right)^{n^m/2} \,
\frac{1}{\big((n+2)^2+n^{2\gamma} \big)^{\alpha/2}}.
\]
Since
\[
\lim_{n \to \infty}
\left(
1 - \frac{4(1+1/n)}{(n+2)^2/n+n^{m}}
\right)^{n^m/2} =
\begin{cases}
	e^{-1}, & m=1, \\
	e^{-2}, & m \geq 2,
\end{cases}
\]
there exists $M>0$  such that
\begin{equation}
	\label{eq:sup_lambda_example}
	\sup_{k \in \mathbb{N}} 
	\left| \frac{\lambda_k}{\lambda_k+2 } \right|^{n^m}
	\,
	\frac{1}{|\lambda_k+2|^{\alpha}} \geq 
	\frac{M}{n^{\alpha \gamma}}
\end{equation}
for all $n \in \mathbb{N}$. 
If $\alpha = \ell(2 - 1/\gamma)$ for some $\ell \in \mathbb{N}_0$,
then
we see from 
\eqref{eq:V_A_example} and \eqref{eq:sup_lambda_example} 
that 
\[
\big\|V_1(A)^{n^m}(1+A)^{-\alpha}\big\| \geq 
\frac{M}{(n^{ m})^{\ell}}
\]
for all $n \in \mathbb{N}$.
This yields
\[
\limsup_{n \to \infty} n^\ell
\big\|V_1(A)^{n}(1+A)^{-\alpha}\big\| \geq M.
\]
Thus, if $\beta = 1/\gamma = 
2/(m+1)$ for some $m\in \mathbb{N}$  and
if $\alpha = \ell (2-\beta)$ for some $\ell \in
\mathbb{N}_0$, then
\[
\limsup_{n \to \infty} n^{\alpha/(2-\beta)}
\|V_1(A)^{n}A^{-\alpha}\| \geq \frac{M}{\| A^{\alpha}(1+A)^{-\alpha}\|}.
\]

\subsubsection{Inverse generator}
Let $\alpha \geq 0$. 
For all $t >0$, 
\[
\big\|e^{-tA^{-1}} A^{-\alpha} \big\| 
=\sup_{k \in \mathbb{N}} 
\frac{\big|e^{-t/(\lambda_k+1)}\big|}{|\lambda_k+1|^{\alpha}}
= \sup_{k \in \mathbb{N}}
\frac{e^{-(k+1)t/((k+1)^2+k^{2\gamma})}}{\big((k+1)^2+k^{2\gamma}\big)^{\alpha/2}}.
\]
Since
\[
e^{-(k+1)t/((k+1)^2+k^{2\gamma})} \geq e^{-2t/k^{2\gamma - 1}}\quad \text{and} \quad 
\frac{1}{(k+1)^2+k^{2\gamma}}
\geq \frac{1}{5 k^{2\gamma}}
\]
for all $k \in \mathbb{N}$,
we have that 
\[
\big\|e^{-tA^{-1}} A^{-\alpha} \big\| 
\geq 
\frac{1}{5^{\alpha/2}}
\sup_{k \in \mathbb{N}}
\frac{e^{-2t/k^{2\gamma - 1}}}{ k^{\alpha \gamma }}.
\]
Define 
\[
g_t(s) \coloneqq \frac{e^{-2t/s^{2\gamma - 1}}}{s^{\alpha \gamma }}
\]
for $s \geq 1$ and $t >0$. 
If $t \geq \alpha \gamma / (4\gamma - 2)$, then
\[
\sup_{s \geq 1} g_t(s) = 
g_t\left(\left(
\frac{(4\gamma -2)t}{\alpha \gamma}
\right)^{1/(2\gamma - 1)} \right) =
\left( 
\frac{\alpha \gamma}{e(4\gamma - 2)t}
\right)^{\alpha \gamma / (2\gamma - 1)}.
\]
Therefore, if we define 
$t_k \coloneqq \alpha \gamma k / (4\gamma - 2)$ for
$k \in \mathbb{N}$, then
\[
\big\|e^{-t_k A^{-1}} A^{-\alpha} \big\|  \geq \frac{1}{5^{\alpha/2}}
\left( 
\frac{\alpha \gamma}{e(4\gamma - 2)t_k}
\right)^{\alpha\gamma  / (2\gamma- 1)}.
\]
Thus,
\[
\limsup_{t \to \infty}	t^{\alpha/ (2 - \beta)} \big\|e^{-t A^{-1}} A^{-\alpha} \big\| \geq 
\frac{1}{5^{\alpha/2}}
\left( 
\frac{\alpha \gamma}{e(4\gamma - 2)}
\right)^{\alpha / (2 - \beta)},
\]
where $\beta = 1/\gamma$.

\section{Decay estimates when $-A^{-1}$
	generates a bounded semigroup}
\label{sec:Decay_estimate2}
The purpose of this section is to improve the estimates 
obtained in Theorems~\ref{thm:VA_bound_nonanalytic} and
\ref{thm:CP_inv}.
We show that the logarithmic term $L_{\alpha,\beta, \varepsilon}$ can be removed
under the assumption that $-A^{-1}$ generates a bounded $C_0$-semigroup.
We also present some related decay estimates 
obtained by the same argument
and show how certain decay estimates for $\|V_1(A)A^{-1}\|$
as $n \to \infty$
can be transferred to growth estimates for $\|Ae^{-tA}\|$
as $t \to 0+$.
\subsection{Decay estimates for Crank--Nicolson schemes and 
	inverse generators}	
The following theorems give improved estimates for 
rates of decay under the assumption on the semigroup-generation property
of $-A^{-1}$.
\begin{theorem}
	\label{thm:CN_CP_Ainv}
	Let $-A$ be the generator of 
	an exponentially stable $C_0$-semigroup $(e^{-tA})_{t \geq 0}$
	in the Crandall--Pazy class with parameter $\beta \in (0,1]$
	on a Hilbert space $H$.
	Assume that $-A^{-1}$ generates
	a  bounded $C_0$-semigroup 
	$(e^{-tA^{-1}})_{t \geq 0}$
	on  $H$.
	Then, for each $\alpha\geq 0$ and 
	$0< \omega_{\min} \leq \omega_{\max} < \infty$,
	there exists a constant $M>0$ such that
	\begin{equation}
		\label{eq:CN_estimate_Ainv}
		\left\|
		\left(
		\prod_{k=1}^n V_{\omega_k}(A) \right)
		A^{-\alpha}
		\right\| \leq 
		\frac{M}{n^{\alpha/(2-\beta)}}
	\end{equation}
	for all $n \in \mathbb{N}$ and $(\omega_k)_{k=1}^{\infty}
	\in \mathcal{S}(\omega_{\min},\omega_{\max})$.
\end{theorem}

\begin{theorem}
	\label{thm:CP_inv_Ainv}
	Let $-A$ be the generator of a bounded 
	$C_0$-semigroup $(e^{-tA})_{t \geq 0}$ 
	on a Hilbert space $H$ such that $0 \in \varrho(A)$.
	Assume that
	$-A^{-1}$ generates a bounded $C_0$-semigroup 
	$(e^{-tA^{-1}})_{t \geq 0}$ on $H$. Then
	the following statements 
	are equivalent for fixed $\alpha >0$ and $\beta \in (0,1]$:
	\begin{enumerate}
		\renewcommand{\labelenumi}{(\roman{enumi})}
		\item 
		$(e^{-tA})_{t \geq 0}$ is exponentially stable and 
		in the Crandall--Pazy class with parameter $\beta$.
		\item
		$\|e^{-tA^{-1}}A^{-\alpha}\| = O(t^{-\alpha/(2-\beta)})$ as $t \to \infty$.
	\end{enumerate}
\end{theorem}

First, we prove Theorem~\ref{thm:CP_inv_Ainv} in
Section~\ref{sec:proof_of_inv_gen}.
Next, we give the proof of 
Theorem~\ref{thm:CN_CP_Ainv} in Section~\ref{sec:proof_of_CN}, where 
Theorem~\ref{thm:CP_inv_Ainv} is used.
In this section, we often use 
the following Gearhart--Pr\"uss theorem for bounded $C_0$-semigroups; see, e.g., \cite{Delloro2022} for an elementary proof.
\begin{theorem}
	\label{thm:GP}
	Let $-A$ be the generator of a bounded $C_0$-semigroup $(e^{-tA})_{t \geq 0}$
	on a Hilbert space.
	Then $(e^{-tA})_{t \geq 0}$ is exponentially stable if and only if
	$\sigma(A) \cap i \mathbb{R} = \emptyset$ 
	and $\sup_{\eta \in \mathbb{R}}
	\|(i\eta+A)^{-1}\| < \infty$.
\end{theorem}

\subsubsection{Proof of Theorem~\ref{thm:CP_inv_Ainv}}
\label{sec:proof_of_inv_gen}
The following 
proposition will be useful in the proof of Theorem~\ref{thm:CP_inv_Ainv}.
Note that 
a sectorial operator $A$ on a reflexive Banach space $X$
is injective if and only if the range of $A$ is dense  in $X$;
see, e.g., \cite[Proposition~2.1.1.h)]{Haase2006}.
\begin{proposition}
	\label{prop:CP_inv_bounded}
	Let $-A$ be the generator of a $C_0$-semigroup $(e^{-tA})_{t \geq 0}$ 
	on a Hilbert space $H$.
	Assume that $A$ is injective and that
	$-A^{-1}$ generates a bounded $C_0$-semigroup 
	$(e^{-tA^{-1}})_{t \geq 0}$ on $H$.
	Then the following statements are equivalent for 
	fixed $\beta \in (0,1]$ and $\lambda \in \varrho(-A)$:
	\begin{enumerate}
		\renewcommand{\labelenumi}{(\roman{enumi})}
		\item  
		$(e^{-tA})_{t \geq 0}$  is in the Crandall--Pazy class with 
		parameter $\beta$, $\sigma(A) \cap i \mathbb{R} \subset \{ 0 \}$, 
		and $\|(i \eta + A)^{-1}\| = O(|\eta|^{-2})$ as $\eta \to 0$.
		
		\item
		$\displaystyle \|e^{-tA^{-1}}(\lambda +A)^{-1}\| = O(t^{-1/(2-\beta)})$ as $t \to \infty$.
	\end{enumerate}
\end{proposition}
\begin{proof}
	Let $B$ be an injective operator on a Banach space $X$.
	A routine calculation shows that 
	for all $\lambda \in \varrho(B) \setminus \{ 0\}$, 
	\begin{equation}
		\label{eq:A_inv_resol_set}
		\frac{1}{\lambda} 
		\in \varrho(B^{-1})
	\end{equation}
	and
	\begin{equation}
		\label{eq:A_inv_resol}
		\left(
		\frac{1}{\lambda} -B^{-1}
		\right)^{-1} = 
		\lambda - \lambda^2 (\lambda -B)^{-1}.
	\end{equation}
	
	Proof of (i) $\Rightarrow$ (ii).
	From \eqref{eq:A_inv_resol_set} with $B = -A$  and the assumption
	$\sigma(A) \cap i \mathbb{R} \subset \{ 0\}$,
	we have that
	$\sigma(A^{-1}) \cap i \mathbb{R} \subset \{0 \}$.
	First, we show that the estimates
	\begin{equation}
		\label{eq:resol_inv_0}
		\|(i \eta + A^{-1})^{-1}\| = 
		O\left(\frac{1}{|\eta|^{2-\beta}} \right)\quad \text{as $\eta \to 0$}
	\end{equation}
	and 
	\begin{equation}
		\label{eq:resol_inv_infty}
		\|(i \eta + A^{-1})^{-1}\| = 
		O(1)\quad \text{as $|\eta| \to \infty$}
	\end{equation} 
	hold.
	Since $(e^{-tA})_{t \geq 0}$  is in the Crandall--Pazy class with 
	parameter $\beta$,
	Theorem~\ref{thm:decay_growth} gives
	\begin{equation}
		\label{eq:resol_estimate_inv}
		\|(i \eta + A)^{-1} \| = O(|\eta|^{-\beta}) \quad \text{as $|\eta| \to \infty$}.
	\end{equation}
	By \eqref{eq:A_inv_resol} in the case when
	$B = -A$ and $\lambda = 1/(i\eta)$ for $\eta \in \mathbb{R} \setminus \{0\}$,
	\[
	(i\eta +A^{-1})^{-1} = 
	\frac{1}{i \eta} + \frac{1}{\eta^2}
	\left(
	\frac{1}{i \eta} + A
	\right)^{-1}.
	\]
	Noting that 
	\eqref{eq:resol_estimate_inv} can be written as
	\[
	\left\|
	\left ( \frac{1}{i\eta} + A \right)^{-1} 
	\right\| = O(|\eta|^{\beta}) \quad \text{as $\eta \to 0$},
	\]
	we obtain the estimate \eqref{eq:resol_inv_0}.
	Similarly, 
	the assumption that $\|(i \eta + A)^{-1}\| = O(|\eta|^{-2})$ as $\eta \to 0$
	yields the estimate \eqref{eq:resol_inv_infty}.
	
	Next, we show that statement
	(ii) is true, by using the estimates~\eqref{eq:resol_inv_0}
	and \eqref{eq:resol_inv_infty}.
	Suppose that $0 \notin \sigma(A^{-1})$.
	Then
	$\sup_{\eta \in \mathbb{R}} \| (i\eta + A^{-1})^{-1}\| < \infty$
	by \eqref{eq:resol_inv_infty}.
	The Gearhart--Pr\"uss theorem  shows that 
	$(e^{-tA^{-1}})_{t \geq 0}$ is exponentially stable, and therefore
	statement
	(ii) is true.
	
	Suppose that $0 \in \sigma(A^{-1})$, and
	let $\lambda \in \varrho(-A)$.
	Since
	\begin{equation}
		\label{eq:omega_A_transform}
		(\lambda +A)^{-1}=
		A^{-1}(1+\lambda A^{-1})^{-1},
	\end{equation}
	we have that 
	\begin{align}
		\label{eq:inv_lambda_A}
		\big\|e^{-tA^{-1}} (\lambda+A)^{-1} \big\| \leq 
		\big \|
		(1+A^{-1}) (1+\lambda A^{-1})^{-1}
		\big\| \,
		\big\|e^{-tA^{-1}} A^{-1}(1 +A^{-1})^{-1} \big\|.
	\end{align}
	Moreover,
	\cite[Theorem 7.6]{Batty2016} shows that 
	if the estimates \eqref{eq:resol_inv_0} and 
	\eqref{eq:resol_inv_infty} hold, then
	\begin{equation}
		\label{eq:Ainv_1_Ainv}
		\big\|e^{-tA^{-1}} A^{-1}(1 + A^{-1})^{-1} \big\| 
		= O\left(\frac{1}{t^{1/(2-\beta)}} \right)
		\quad \text{as $t \to \infty$}.
	\end{equation}
	From \eqref{eq:inv_lambda_A} and \eqref{eq:Ainv_1_Ainv},
	we conclude that statement~(ii) is true.

	Proof of (ii) $\Rightarrow$ (i).
	Let $\lambda \in \varrho(-A)$.
	From \eqref{eq:omega_A_transform}, we have 
	\[
	A^{-1} x= 
	(\lambda +A)^{-1}(1+\lambda A^{-1})x
	\]
	for all $x \in D(A^{-1})$,
	and hence
	\begin{align*}
		\big\|e^{-tA^{-1}} A^{-1}(1 + A^{-1})^{-1} \big\|  &\leq 
		\big\|
		(1 + \lambda A^{-1}) (1 + A^{-1})^{-1}
		\big\|\,
		\big\|e^{-tA^{-1}} (\lambda+ A)^{-1} \big\|.
	\end{align*}
	This, together with statement~(ii), yields \eqref{eq:Ainv_1_Ainv}.
	Therefore,
	\cite[Theorem~6.10]{Batty2016} shows that 
	$\sigma(A^{-1}) \cap i \mathbb{R} 
	\subset \{ 0\}$ and that
	the estimates \eqref{eq:resol_inv_0}
	and \eqref{eq:resol_inv_infty} hold.
	Using
	\eqref{eq:A_inv_resol_set} and \eqref{eq:A_inv_resol} with $B = -A^{-1}$
	as in the proof of the implication (i) $\Rightarrow$ (ii), we derive 
	$\sigma(A) \cap i \mathbb{R} \subset \{0 \}$,
	\begin{equation}
		\label{eq:resol_inf_estimate}
		\|
		(
		i\eta + A
		)^{-1}\| = 
		O\left( \frac{1}{|\eta|^{\beta}} \right)\quad \text{as $|\eta| \to \infty$},
	\end{equation}
	and
	\begin{equation}
		\|(i \eta + A)^{-1}\| = O\left( \frac{1}{|\eta|^2} \right)\quad \text{as $\eta \to 0$}.
	\end{equation}
	By the estimate \eqref{eq:resol_inf_estimate} and 
	Theorem~\ref{thm:decay_growth},
	the $C_0$-semigroup 
	$(e^{-tA})_{t \geq 0}$  is in the Crandall--Pazy class with 
	parameter $\beta$.
\end{proof}

\begin{proof}[Proof of Theorem~\ref{thm:CP_inv_Ainv}.]
	Let $\alpha >0$ be given.
	By \cite[Lemma~4.2]{Batty2016}, 
	statement (ii) holds if and only if
	\begin{equation}
		\label{eq:inv_gen_a_1}
		\big\|e^{-tA^{-1}}A^{-1}\big\| = O
		\left( \frac{1}{t^{1/(2-\beta)}}\right) 
	\end{equation}
	as $t \to \infty$.

	Proof of (i) $\Rightarrow$ (ii).
	If $(e^{-tA})_{t \geq 0}$ is exponentially stable, then 
	$\sigma(A) \cap i \mathbb{R} = \emptyset$ and 
	$\sup_{\eta \in \mathbb{R}}\|(i \eta + A)^{-1}\| < \infty$.
	The estimate
	\eqref{eq:inv_gen_a_1}
	follows immediately from
	the implication (i)
	$\Rightarrow$ (ii) in Proposition~\ref{prop:CP_inv_bounded}.
	
	Proof of (ii) $\Rightarrow$ (i).
	By 
	the implication (ii)
	$\Rightarrow$ (i) in Proposition~\ref{prop:CP_inv_bounded},
	the $C_0$-semigroup 
	$(e^{-tA})_{t \geq 0}$  is in the Crandall--Pazy class
	and $\sigma(A) \cap i \mathbb{R} \subset \{ 0 \}$.
	It remains to prove that
	$(e^{-tA})_{t \geq 0}$ is exponentially stable.
	From the assumption $0 \in \varrho(A)$, we obtain $\sigma(A) \cap i \mathbb{R}
	= \emptyset$.
	Since 
	$\|(i \eta + A)^{-1} \| = O(|\eta|^{-\beta})$ as $|\eta| \to \infty$
	by Theorem~\ref{thm:decay_growth}, it follows that 
	$\sup_{\eta \in \mathbb{R}} \|(i \eta + A)^{-1}\| < \infty$.
	Thus,
	$(e^{-tA})_{t \geq 0}$ is exponentially stable by 
	the Gearhart--Pr\"uss theorem.
\end{proof}

\subsubsection{Proof of Theorem~\ref{thm:CN_CP_Ainv}}
\label{sec:proof_of_CN}
For $r \in (0,1)$, we define $\xi_r >0$  by
\begin{equation}
	\label{eq:xi_def}
	\xi_r \coloneqq \frac{1-r^2}{2(r^2+1)}.
\end{equation}
In the proofs of \cite[Lemmas~2.1 and 2.2]{Piskarev2007},
the following result was derived from the approach based on 
Lyapunov equations.
\begin{lemma}
	\label{lem:PZ_lemma}
	Let $-A$ be the generator of a bounded $C_0$-semigroup
	$(e^{-tA})_{t \geq 0}$ on 
	a Hilbert space $H$. 
	Assume 
	that $A$ is injective and that $-A^{-1}$ generates a bounded $C_0$-semigroup $(e^{-tA^{-1}})_{t \geq 0}$ on  $H$. Let $0 < \omega_{\min} \leq \omega_{\max} < \infty$, and
	define
	$R(r) \in \mathcal{L}(H)$ by 
	\[
	R(r) \coloneqq 2\omega_{\max}P\left(
	\omega_{\min } \xi_r
	\right) + 
	\frac{2}{\omega_{\min}}
	Q\left(\frac{\xi_r}{\omega_{\max}}\right),\quad r \in (0,1),
	\]
	where $\xi_r>0$ is as in \eqref{eq:xi_def}
	and $P(\xi), Q(\xi) \in \mathcal{L}(H)$ are defined by
	\begin{align}
		P(\xi) x &\coloneqq 
		\int^{\infty}_0 e^{-2\xi t} (e^{-tA})^* e^{-tA}x dt, \label{eq:P_def}\\
		Q(\xi) x &\coloneqq  
		\int^{\infty}_0 e^{-2\xi t} (e^{-tA^{-1}})^* e^{-tA^{-1}}x dt
		\label{eq:Q_def}
	\end{align}
	for $\xi >0$ and $x \in H$.
	Then
	there exists a constant $K_0 >0$ such that 
	\[
	\left|(n+1)r^n \left\langle y, \left( \prod_{k=1}^n V_{\omega_k}(A)\right)x
	\right\rangle
	\right| \leq 
	\frac{K_0 \|y\| }{\sqrt{1-r}} \sqrt{\langle 
		x, R(r)x
		\rangle}
	\] 
	for all $x, y \in H$, $n \in \mathbb{N}$, $r \in (0,1)$, and $(\omega_k)_{k=1}^{\infty}
	\in \mathcal{S}(\omega_{\min},\omega_{\max})$.
\end{lemma}

Now we examine $\langle
x, P(\xi)x
\rangle$  and $\langle
x, Q(\xi)x
\rangle$
for $\xi >0$, where $P(\xi), Q(\xi) \in \mathcal{L}(H)$ are defined by \eqref{eq:P_def} and \eqref{eq:Q_def}, respectively.
If $(e^{-tA})_{t \geq 0}$ is exponentially stable, then
there exist $K\geq 1$ and $c>0$ such that 
$\| e^{-tA}\| \leq Ke^{-ct}$ for all $t \geq 0$.
Therefore,
\begin{equation}
	\label{eq:P_estimate}
	\langle x, P(\xi) x \rangle  \leq K^2 \|x\|^2
	\int_0^{\infty} e^{-2(\xi+c) t} dt \leq 
	\frac{K^2 \|x\|^2}{2c}
\end{equation}
for all $\xi >0$ and $x \in H$. The following lemma gives
an estimate for 
$\langle x, Q(\xi) x \rangle $
when $(e^{-tA})_{t \geq 0}$
is in the Crandall--Pazy class.
\begin{lemma}
	\label{eq:Q_bound_exp}
	Let $-A$ be the generator of 
	an exponentially stable $C_0$-semigroup $(e^{-tA})_{t \geq 0}$
	in the Crandall--Pazy class with parameter $\beta \in (0,1]$
	on a Hilbert space $H$.
	Assume that $-A^{-1}$ generates
	a  bounded $C_0$-semigroup 
	$(e^{-tA^{-1}})_{t \geq 0}$
	on  $H$, and define
	$Q(\xi) \in \mathcal{L}(H)$ by \eqref{eq:Q_def} for $\xi >0$.
	For each $0\leq  \alpha< (2 - \beta)/2$,
	there exists a constant $K_0>0$ such that 
	\begin{equation}
		\label{eq:Q_bound1}
		\sup_{0< \xi < 1}
		\xi^{1-2\alpha/(2-\beta)}
		\langle
		x, Q(\xi)x
		\rangle
		\leq 
		K_0 \|A^{\alpha}x\|^2
	\end{equation}
	for all $x \in D(A^{\alpha})$.
\end{lemma}
\begin{proof}
	Let $0 \leq \alpha < (2-\beta)/2$.
	Define
	$\gamma \coloneqq 2\alpha/(2-\beta)$ and 
	$
	K_1 \coloneqq 
	\sup_{t \geq 0}\|e^{-tA^{-1}}\|$.
	By Theorem~\ref{thm:CP_inv_Ainv},
	there exist $K_2, t_0 >0$ such that 
	for all $t \geq t_0$,
	\[
	\big\|e^{-tA^{-1}}A^{-\alpha}\big\| \leq \frac{K_2}{t^{\gamma/2}}.
	\] 
	Hence,
	\begin{align*}
		\langle x, Q(\xi) x\rangle 
		&=
		\int^{t_0}_0 e^{-2\xi t} \big\|e^{-tA^{-1}}x\big\|^2 dt + 
		\int^{\infty}_{t_0} e^{-2\xi t}\big\|e^{-tA^{-1}}x\big\|^2 dt \\
		&\leq 
		t_0 K_1^2 \|x\|^2 + 
		K_2^2 \|A^{\alpha}x\|^2 \int^{\infty}_{t_0} \frac{e^{-2\xi t} }{t^{\gamma }} dt
	\end{align*}
	for all $x \in D(A^{\alpha})$ and $\xi >0$. 
	Since $0 \leq \gamma  < 1$, we have
	\[
	\int^{\infty}_{t_0} \frac{e^{-2\xi t} }{t^{\gamma }} dt \leq 
	\int^{\infty}_0
	\frac{e^{-2\xi t} }{t^{\gamma}} dt =
	\frac{\Gamma(1-\gamma)}{(2\xi)^{1-\gamma}}.
	\]
	Thus, there exists $K_0>0$ such that 
	\eqref{eq:Q_bound1} holds for all $x \in D(A^{\alpha})$.
\end{proof}

\begin{proof}[Proof of Theorem~\ref{thm:CN_CP_Ainv}]
	Let
	the operator $R(r) \in \mathcal{L}(H)$ be as in Lemma~\ref{lem:PZ_lemma}.
	By this lemma,
	there exists $K_1>0$ such that 
	\begin{equation}
		\label{eq:xn_bound}
		\left\| \left( \prod_{k=1}^n V_{\omega_k}(A) \right)
		x
		\right\| \leq  \frac{K_1\sqrt{\langle x, R(r)x \rangle } }{(n+1)r^n \sqrt{1-r}}
	\end{equation}
	for all $x \in H$, $n \in \mathbb{N}$, $r \in (0,1)$,
	and $(\omega_k)_{k=1}^{\infty}
	\in \mathcal{S}(\omega_{\min},\omega_{\max})$.
	
	Let $0 \leq \alpha < (2-\beta)/2$.
	By \eqref{eq:P_estimate} and \eqref{eq:Q_bound1}, there exist
	constants $K_2,K_3 >0$ and $r_0\in(0,1)$, depending on $\omega_{\min}$
	and $\omega_{\max}$, such that 
	\begin{equation}
		\label{xRx_bound}
		\langle x, R(r)x \rangle
		\leq
		K_2 \|x\|^2 + 
		\frac{K_3 }{\xi_r^{1-2\alpha/(2-\beta)}}\|A^{\alpha}x\|^2
	\end{equation}
	for all $x \in D(A^\alpha)$ and  $r \in (r_0,1)$,
	where $\xi_r>0$ is as in \eqref{eq:xi_def}.
	If $r = n/(n+1)$ for $n \in \mathbb{N}$, 
	then
	\[
	\xi_r = \frac{1-r^2}{2(r^2+1)} =  \frac{2n+1}{2(2n^2+2n+1)}
	\]
	and
	\[
	\frac{1}{(n+1)r^n \sqrt{1-r}} = \frac{(1+1/n)^n}{\sqrt{n+1}}.
	\]
	By the 
	estimates \eqref{eq:xn_bound} and 
	\eqref{xRx_bound},
	there exist $M_1 >0$ and $n_0 \in \mathbb{N}$ such that 
	\[
	\left\| \left( \prod_{k=1}^n V_{\omega_k}(A) \right)
	x
	\right\|\leq \frac{M_1 }{n^{\alpha/(2-\beta)}}
	\|A^{\alpha}x\|
	\]
	for all $x \in D(A^{\alpha})$, $n \geq  n_0$,
	and $(\omega_k)_{k=1}^{\infty}
	\in \mathcal{S}(\omega_{\min},\omega_{\max})$.
	We also have that  for all $1 \leq n < n_0$,
	\[
	\left\|
	\prod_{k=1}^n
	V_{\omega_k}(A)	
	\right\| \leq 
	\max\big\{
	1,\,
	\max\{
	\|V_{\omega}(A)\|^{n_0}: \omega_{\min} \leq \omega \leq \omega_{\max}
	\}
	\big\}.
	\]
	Therefore, we obtain the desired estimate \eqref{eq:CN_estimate_Ainv}
	in the case $0 \leq \alpha < (2-\beta)/2$. The case $\alpha \geq (2-\beta)/2$
	follows 
	as in the proof of Theorem~\ref{thm:VA_bound_nonanalytic}.
\end{proof}

\subsection{Related estimates}
The argument in the proof of Proposition~\ref{prop:CP_inv_bounded}
can be applied 
to bounded $C_0$-semigroups with polynomial decay 
and 
exponentially stable $C_0$-semigroups. The following results
are analogous to the results on Cayley transforms
developed in \cite{Pritchard2024}. 
\begin{proposition}
	\label{prop:decay_poly_exp}
	Let $-A$ be the generator of a bounded $C_0$-semigroup $(e^{-tA})_{t \geq 0}$ 
	on a Hilbert space $H$ such that $0 \in \varrho(A)$.
	Assume that
	$-A^{-1}$ generates a bounded $C_0$-semigroup 
	$(e^{-tA^{-1}})_{t \geq 0}$ on $H$.
	\begin{enumerate}
		\renewcommand{\labelenumi}{\alph{enumi})}
		\item The following statements are equivalent for fixed $\alpha,\beta >0$:
		\begin{enumerate}
			\renewcommand{\labelenumii}{(\roman{enumii})}
			\item  
			$\|e^{-tA}A^{-1}\| = O(t^{-1/\beta})$ as $t \to \infty$.
			\item
			$\|e^{-tA^{-1}}A^{-\alpha}\| = O(t^{-\alpha/(2+\beta)})$ as $t \to \infty$.
		\end{enumerate}
		\item The following statements are equivalent for a fixed $\alpha >0$:
		\begin{enumerate}
			\renewcommand{\labelenumii}{(\roman{enumii})}
			\item  
			$(e^{-tA})_{t \geq 0}$ is exponentially stable.
			\item
			$\|e^{-tA^{-1}}A^{-\alpha}\| = O(t^{-\alpha/2})$ as $t \to \infty$.
		\end{enumerate}
	\end{enumerate}
\end{proposition}
\begin{proof}
	Let $(e^{-tA})_{t \geq 0}$ be a bounded $C_0$-semigroup
	on a Hilbert space $H$, and let $\beta >0$.
	By \cite[Theorem~1.1]{Batty2008} and 
	\cite[Theorem~2.4]{Borichev2010}, 
	$\|e^{-tA}(1+A)^{-1}\| = O(t^{-1/\beta})$ as $t \to \infty$ if and only if
	$\sigma(A) \cap i \mathbb{R} =
	\emptyset$ and 
	$\|(i\eta + A)^{-1}\| = O(|\eta|^{\beta})$ as $|\eta| \to \infty$.
	Similarly, the Gearhart--Pr\"uss theorem 
	shows that $(e^{-tA})_{t \geq 0}$ is exponentially stable
	if and only if $\sigma(A) \cap i \mathbb{R} =
	\emptyset$ and $\sup_{\eta \in \mathbb{R}} \|(i\eta + A)^{-1}\| < \infty$.
	We obtain the desired results in the case $\alpha = 1$ 
	by
	combining these  equivalences  with the argument used in the proof of 
	Proposition~\ref{prop:CP_inv_bounded}.
	The case $\alpha\not=1$ follows from
	the case $\alpha = 1$ and \cite[Lemma~4.2]{Batty2016}.
\end{proof}

Let $-A$ be the generator of a bounded $C_0$-semigroup
on a Hilbert space such that $\|e^{-tA}(1+A)^{-1}\| = 
O(t^{-1/\beta})$ as $t \to \infty$ for some $\beta >0$. 
Then
$\sigma(A) \cap i \mathbb{R} =\emptyset$ by \cite[Theorem~1.1]{Batty2008}.
It was shown in \cite[Theorem~4]{Wakaiki2024JEE} that 
\[
\|e^{-tA^{-1}}A^{-1}\| = O
\left( \frac{\log t}{t^{1/(2+\beta)}}
\right)\quad \text{as $t \to \infty$}.
\] 
From Proposition~\ref{prop:decay_poly_exp}.a),
we see that the logarithmic term can be omitted
under the additional assumption that
$-A^{-1}$ generates a bounded $C_0$-semigroup.
On the other hand, the implication (i) $\Rightarrow$ (ii) in Proposition~\ref{prop:decay_poly_exp}.b) was proved
without the assumption on $-A^{-1}$ via the $\mathcal{B}$-calculus in \cite[Theorem~2.4]{Wakaiki2023IEOT}.

From Proposition~\ref{prop:decay_poly_exp}.a), 
we obtain the following decay estimate for the 
Crank--Nicolson scheme.
Since it follows from the same argument as 
in Theorem~\ref{thm:CN_CP_Ainv},
we omit the proof.
\begin{proposition}
	\label{prop:CN_polynomial_Ainv}
	Let $-A$ be the generator of 
	a bounded $C_0$-semigroup $(e^{-tA})_{t \geq 0}$
	on a Hilbert space $H$ such that 
	$\|e^{-tA}(1+A)^{-1}\| = O(t^{-1/\beta})$ as $t \to \infty$.
	Assume that $-A^{-1}$ generates
	a  bounded $C_0$-semigroup 
	$(e^{-tA^{-1}})_{t \geq 0}$
	on  $H$.
	Then, for each $\alpha\geq 0$ and 
	$0< \omega_{\min} \leq \omega_{\max} < \infty$,
	there exists a constant $M>0$ such that
	\[
	\left\|
	\left(
	\prod_{k=1}^n V_{\omega_k}(A) \right)
	A^{-\alpha}
	\right\| \leq 
	\frac{M}{n^{\alpha/(2+\beta)}}
	\]
	for all $n \in \mathbb{N}$ and $(\omega_k)_{k=1}^{\infty}
	\in \mathcal{S}(\omega_{\min},\omega_{\max})$.
\end{proposition}

Proposition~\ref{prop:CN_polynomial_Ainv} is an extension
of \cite[Proposition~5.(a)]{Wakaiki2024JEE}, which was restricted to the case where $A$ is a normal
operator.
Proposition~\ref{prop:decay_poly_exp}.b) also
yields a similar decay estimate for  
the Crank--Nicolson scheme. However,
a stronger result, which does not require the 
assumption
on the semigroup-generation property of 
$-A^{-1}$, was obtained via the $\mathcal{B}$-calculus in \cite[Theorem~2]{Wakaiki2024JEE}.

Finally, we show
an equivalence result on Cayley transforms for the Crandall--Pazy class of $C_0$-semigroups, which is analogous to the result derived in \cite{Pritchard2024} for 
bounded $C_0$-semigroups with polynomial decay. 

\begin{theorem}
	\label{thm:relation_poly_decay}
	Let $-A$ be the generator of a $C_0$-semigroup $(e^{-tA})_{t \geq 0}$
	on a Hilbert space such that 
	$\sigma(A) \cap i \mathbb{R} = \emptyset$.
	Let $\omega \in \varrho(-A) \cap (0,\infty)$ and
	assume that 
	$\sup_{n \in \mathbb{N}_0} \|V_{\omega}(A)^n\| < \infty$.
	Then the following statements are equivalent for a fixed $\beta \in (0,1]$:
	\begin{enumerate}
		\renewcommand{\labelenumi}{(\roman{enumi})}
		\item$(e^{-tA})_{t \geq 0}$ is in the Crandall--Pazy class with parameter $\beta$.
		\item $\|V_{\omega}(A)^n(\omega + A)^{-1}\| = O(n^{-1/(2-\beta)})$ as $n \to \infty$.
	\end{enumerate}
\end{theorem}

\begin{proof}
	Define $B \coloneqq A/\omega$. Then $V_{\omega}(A) = V_1(B)$. It is enough
	to show the equivalence of  (i) and (ii), replacing $A$ with $B$ and $\omega$
	with $1$. Routine calculations show that for all $\eta \in \mathbb{R}$, 
	\begin{equation}
		\label{eq:spectrum_CT}
		i \eta \in \sigma(B) \quad \Leftrightarrow \quad \frac{i \eta - 1}{i \eta + 1}
		\in \sigma\big(V_1(B)\big)
	\end{equation}
	and that if $i\eta \in \sigma(B)$, then
	\begin{equation}
		\label{eq:eta_V_resol}
		\left(
		\frac{i \eta - 1}{i\eta + 1} - V_1(B)
		\right)^{-1} = \frac{1}{2}(i\eta + 1) \left(-1 + (i\eta + 1) (i \eta - B)^{-1}\right).
	\end{equation}
	We write $\mathbb{T} \coloneqq \{ z \in \mathbb{C}: |z| = 1 \}$.  
	From \eqref{eq:spectrum_CT} and the assumption that $\sigma(A) \cap i \mathbb{R} = \emptyset$, we have that 
	$\sigma (V_1(B)) \cap \mathbb{T} \subset \{1 \}$.
	
	Proof of (i) $\Rightarrow$ (ii).
	If $\sigma (V_1(B)) \cap \mathbb{T} = \emptyset$, then
	the spectrum radius of $V_1(B)$ is less than one, and hence
	statement~(ii) immediately follows.
	
	Suppose that $\sigma (V_1(B)) \cap \mathbb{T} = \{1 \}$.
	Since $\| (i \eta - B)^{-1}  \| = O (|\eta|^{-\beta})$ 
	as $|\eta| \to \infty$ by Theorem~\ref{thm:decay_growth},
	we have from \eqref{eq:eta_V_resol} that
	\begin{equation}
		\label{eq:resol_continuous}
		\left\|
		\left(
		\frac{i \eta - 1}{i\eta + 1} - V_1(B)
		\right)^{-1}
		\right\| = O(|\eta|^{2-\beta})
	\end{equation}
	as $|\eta| \to \infty$.
	For $\eta \in \mathbb{R} \setminus \{ 0\}$,
	define $\theta \in (-\pi, \pi)$ by
	\[
	\theta \coloneqq \arg \left( \frac{i \eta - 1}{i\eta + 1} \right).
	\]
	Then
	\[
	e^{i \theta} = \frac{i \eta - 1}{i\eta + 1}.
	\]
	Since
	\[
	| \theta | = \pi - 2|\arctan \eta|,
	\]
	there exist $C_1,C_2,\delta  >0$ such that 
	for all $\eta \in \mathbb{R}$ satisfying $|\eta| \geq \delta$,
	\begin{equation}
		\label{eq:arg_bounds}
		\frac{C_1}{|\eta|}
		\leq
		| \theta | \leq 
		\frac{C_2}{|\eta|}.
	\end{equation}
	Using \eqref{eq:resol_continuous} and 
	\eqref{eq:arg_bounds}, we obtain
	\begin{equation}
		\label{eq:resol_discrete}
		\big\|\big(e^{i \theta} - V_1(B)\big)^{-1} \big\| = O
		\left(
		\frac{1}{
			|\theta|^{2-\beta}}
		\right)
	\end{equation}
	as $\theta \to 0$.
	Hence, $\|V_1(B)^n(1+B)^{-1}\| = O(n^{-1/(2-\beta)})$ as $n \to \infty$ 
	by \cite[Theorem~3.10]{Seifert2016}.

	Proof of (ii) $\Rightarrow$ (i).
	From \cite[Theorem~3.10]{Seifert2016},
	we obtain \eqref{eq:resol_discrete} in the case $\sigma (V_1(B)) \cap \mathbb{T} = \{1 \}$.
	Also, \eqref{eq:resol_discrete} is satisfied in the case
	$\sigma (V_1(B)) \cap \mathbb{T} = \emptyset$, since
	the spectrum radius of $V_1(B)$ is less than one.
	Hence, we derive
	\eqref{eq:resol_continuous} via \eqref{eq:arg_bounds}, and \eqref{eq:eta_V_resol} yields
	\[
	\left\|
	(i\eta + 1)^2 (i \eta - B)^{-1}
	\right\| \leq 
	2
	\left\|
	\left(
	\frac{i \eta - 1}{i\eta + 1} - V_1(B)
	\right)^{-1}
	\right\| + |i\eta + 1| = O(|\eta|^{2-\beta})
	\]
	as $|\eta| \to \infty$. This implies that $\|(i\eta - B)^{-1} \| = O(|\eta|^{-\beta})$
	as $|\eta| \to \infty$. Thus, $(e^{-tB})_{t \geq 0}$ is in the Crandall--Pazy class with parameter $\beta$ by Theorem~\ref{thm:decay_growth}.
\end{proof}

\appendix
\renewcommand\thetheorem{\Alph{section}.\arabic{theorem}}
\section{Proof of 
	Theorem~\ref{thm:CN_holomorphic} via
	the $\mathcal{D}$-calculus}
In this appendix, we give an alternative proof of Theorem~\ref{thm:CN_holomorphic} 
based on the $\mathcal{D}$-calculus.
\subsection{Background material on the $\mathcal{D}$-calculus}
We begin by recalling some preliminaries on the $\mathcal{D}$-calculus;
see \cite{Batty2023} for further details.
For $s > -1$, let $\mathcal{D}_s$ be the space of 
all holomorphic functions $f$ on $\mathbb{C}_+$ such that 
\[
\|f\|_{\mathcal{D}_{s,0}} \coloneqq
\int_0^{\infty}\xi^s
\int_{-\infty}^{\infty}
\frac{|f'(\xi+i \eta)|}{(\xi^2 + \eta^2)^{(s+1)/2}} d\xi d\eta < \infty.
\]
For all $s > -1$, we have
\[
\|f\|_{\mathcal{D}_{s,0}} =
\int_{-\pi/2}^{\pi/2}\cos^s\theta
\int_{0}^{\infty} |f'(r e^{i\theta})|
dr  d\theta.
\]
If $\sigma \geq s > -1$, then
$\mathcal{D}_s \subset \mathcal{D}_{\sigma}$
and 
$\|f\|_{\mathcal{D}_{\sigma,0}}  \leq \|f\|_{\mathcal{D}_{s,0}} $.
If $f \in \mathcal{D}_s$ for some $s > -1$, then
\[
f(\infty) \coloneqq \lim_{z \to \infty,\, z \in \Sigma_{\theta}} f(z)
\]
exists in $\mathbb{C}$ for all $\theta \in (0,\pi/2)$. 
For each $s > -1$, the space $\mathcal{D}_s$ equipped with the norm
\[
\|f\|_{\mathcal{D}_s} \coloneqq |f(\infty)| + \|f\|_{\mathcal{D}_{s,0}},\quad 
f \in \mathcal{D}_s,
\]
is a Banach space. 
If $s > 0$, then 
$\mathcal{LM}
\subset \mathcal{D}_s
$ and $\|f\|_{\mathcal{D}_s} \leq C \|f\|_{\textrm{HP}}$ for all 
$f \in \mathcal{LM}$ and some $C>0$.

Let $A \in \Sect(\pi/2-)$, and define
\[
M_A \coloneqq \sup_{z \in \mathbb{C}_+} \|z(z+A)^{-1}\|.
\]
Let $s > -1$ and $f \in \mathcal{D}_s$. 
We define a linear operator $f_{\mathcal{D}_s}(A)$
on $X$ by
\begin{align*}
	f_{\mathcal{D}_s}(A) \coloneqq 
	f(\infty) -
	\frac{2^s}{\pi}
	\int^{\infty}_0 \xi^s \int_{-\infty}^{\infty}
	f'(\xi+i\eta)(A+\xi-i\eta)^{-(s+1)} d\eta d\xi,
\end{align*}
where the integral is absolutely convergent in the operator norm.
The map $f \mapsto f_{\mathcal{D}_s}(A)$
is bounded from $\mathcal{D}_s$ to $\mathcal{L}(X)$. Indeed, we have
\begin{equation}
	\label{eq:fDs_bound}
	\|f_{\mathcal{D}_s}(A)\| \leq |f(\infty)| + \frac{2^s M_A^{\lceil s+ 1 \rceil}}{\pi}
	\|f\|_{\mathcal{D}_{s,0}},
\end{equation}
where $\lceil a \rceil$ is the smallest integer 
greater than or equal to $a \in \mathbb{R}$.
Moreover, $f_{\mathcal{D}_\sigma}(A) = f_{\mathcal{D}_s}(A)$ for $\sigma \geq s$.

Define
\[
\mathcal{D}_{\infty} \coloneqq \bigcup_{s > -1} \mathcal{D}_s.
\]
Then $\mathcal{D}_{\infty}$ is an algebra. For $A \in \Sect(\pi/2-)$,
we derive the following
algebra homomorphism:
\[
\Psi_A\colon \mathcal{D}_{\infty} \mapsto \mathcal{L}(X),\quad
\Psi_A(f) = f_{\mathcal{D}_s}(A),\quad 
f \in \mathcal{D}_s,\, s > -1.
\]
In addition, $\Psi_A$ is the unique algebra homomorphism
from $\mathcal{D}_{\infty}$ to $\mathcal{L}(X)$
that satisfies the following two properties:
\begin{enumerate}
	\renewcommand{\labelenumi}{\alph{enumi})}
	\item $\Psi_A((\lambda+ \cdot \,)^{-1}) = (\lambda + A)^{-1} $
	for all $\lambda \in \mathbb{C}_+$.
	\item 
	For each $s>-1$, there exists a constant $C_s(A)>0$
	such that for all $f \in \mathcal{D}_s$,
	\[
	\|\Psi_A(f)\| \leq |f(\infty)| + C_s(A) \|f\|_{\mathcal{D}_s}.
	\]
\end{enumerate} 
We call $ \Psi_A$  
the  {\em $\mathcal{D}$-calculus} for $A$.
If
$f \in \mathcal{LM}$, then $\Psi_A(f)= \Pi_A(f)$, where $\Pi_A$ is 
the Hillle-Phillips calculus defined as in Section~\ref{sec:preliminaries_B_calc}.
In the appendix, $f(A)$ refers to the $\mathcal{D}$-calculus.

\subsection{Decay estimate for Crank--Nicolson schemes}
\label{sec:CN_scheme}

Let $n \in \mathbb{N}$, $\alpha \geq 0$, and $0<c <\omega_{\min}\leq \omega_{\max} < \infty$. 
Define
\begin{equation}
	\label{eq:fnaw_def}
	f_{n,\alpha,(\omega_k)}(z) \coloneqq
	\frac{1}{(z+c+\omega_{\max})^{\alpha}} 
	\prod_{k=1}^n
	\frac{z+c-\omega_k}{z+c+\omega_k},\quad z \in \mathbb{C}_+,
\end{equation}
where $(\omega_k)_{k\in \mathbb{N}} \in \mathcal{S}(\omega_{\min},\omega_{\max})$.
Note that the function $f_{n,\alpha,(\omega_k)}$ defined here 
is slightly different from the one defined in Section~\ref{sec:CP_CN_decay}; see \eqref{eq:fnaw_def_CP}. The difference lies in the use of  $\omega_{\max}$ instead of
$\omega_{\min}$.

To prove Theorem~\ref{thm:CN_holomorphic} via the $\mathcal{D}$-calculus,
we employ the following estimate for $f_{n,\alpha,(\omega_k)} $.
\begin{proposition}
	\label{prop:fnaw_bound}
	Let $s>\alpha\geq 0$ and $0 < c < \omega_{\min} \leq  \omega_{\max} < \infty$.
	There exists a constant $M >0$  such that 
	the function $f_{n,\alpha,(\omega_k)}$ defined by \eqref{eq:fnaw_def}
	satisfies
	\begin{equation}
		\label{eq:f_norm}
		\big\|
		f_{n,\alpha,(\omega_k)}
		\big\|_{\mathcal{D}_{s,0}}
		\leq \frac{M}{n^{\alpha}}
	\end{equation}
	for all $n \in \mathbb{N}$ and 
	$(\omega_k)_{k=1}^{\infty} \in \mathcal{S}(\omega_{\min},\omega_{\max})$.
\end{proposition}

We need two technical lemmas for the proof of 
Proposition~\ref{prop:fnaw_bound}.
\begin{lemma}
	\label{lem:hol1}
	Let $a,b \in \mathbb{R}$ satisfy $|b| \leq a$, and
	let $0 \leq \delta \leq 1$.
	Then,
	for all $r >0$,
	\[
	\frac{
		r^2+2\delta br + b^2
	}{r^2+2\delta ar + a^2} 
	\leq 
	\left( 
	\frac{r+\delta |b|}{r+\delta a}
	\right)^2.
	\]
\end{lemma}
\begin{proof}
	Let $r >0$. It is enough to show that 
	\begin{equation}
		\label{eq:sec_func_bound}
		(r+\delta |b|)^2
		(r^2+2\delta ar + a^2) \geq
		(r+\delta a)^2
		(r^2+2\delta br + b^2).
	\end{equation}
	A routine calculation shows that
	\begin{align*}
		(r+\delta |b|)^2
		(r^2+2\delta ar + a^2) -
		(r+\delta a)^2
		(r^2+2\delta br + b^2) 
		=
		c_1r + c_2 r^2  +  c_3 r^3,
	\end{align*}
	where
	\begin{align*}
		c_1 &\coloneqq 2\delta a \big( (|b| - \delta^2 b)a-(1-\delta^2)b^2\big), \\
		c_2 &\coloneqq (1-\delta^2) (a^2-b^2) + 4\delta^2a(|b| - b), \\
		c_3 &\coloneqq 2\delta (|b| - b).
	\end{align*}
	From $0 \leq |b| \leq a$ and $0 \leq \delta \leq 1$, we 
	immediately see that $c_2$ and $c_3$ are non-negative.
	Since 
	\[
	|b| - \delta^2b \geq (1-\delta^2)|b|,
	\]
	we also have that
	\[
	c_1
	\geq 
	2\delta (1-\delta^2) a|b| (a-|b|) \geq 0.
	\]
	Thus, the estimate \eqref{eq:sec_func_bound} holds.
\end{proof}

\begin{lemma}
	\label{lem:tc_int_bound}
	For each $\alpha >0$, there exists a constant $M>0$
	such that 
	\begin{equation}
		\label{eq:tc_int_bound}
		\int_0^{\infty} \frac{(r+d)^n}{(r+c)^{n+\alpha+1}} dr \leq 
		\frac{M}{(c- d)^{\alpha}n^{\alpha}}
	\end{equation}
	for all $n \in \mathbb{N}$ and $c > d \geq 0$.
\end{lemma}

\begin{proof}
	Let $\alpha >0$.
	We have
	\begin{equation}
		\label{eq:tc_int_equiv}
		\int_0^{\infty} \frac{(r+d)^n}{(r+c)^{n+\alpha+1}}dr =
		\frac{1}{(c-d)^{\alpha}} \int_{d/(c-d)}^{\infty}
		\frac{r^n}{(r+1)^{n+\alpha+1}} dr.
	\end{equation}
	Moreover,
	\begin{equation}
		\label{eq:lemma_beta_func}
		\int_0^{\infty} \frac{r^n}{(r+1)^{n+\alpha+1}} dr
		=  \int_0^1 
		s^n(1-s)^{\alpha-1} ds = {\mathrm B}(n+1,\alpha),
	\end{equation}
	where ${\mathrm B}$ is the beta function.
	We use the following well-known 
	estimate on the beta function:
	\begin{equation}
		\label{eq:Beta_asymp}
		{\mathrm B}(n,\alpha) = O\left(
		\frac{1}{n^{\alpha}}
		\right)
	\end{equation}
	as $n \to \infty$, which can be obtained from the expression of
	the beta function by the gamma function $\Gamma$
	\[
	{\mathrm B}(s,t) = \frac{\Gamma(s)\Gamma(t)}{\Gamma(s+t)}
	\]
	for $s,t >0$
	(see \cite[item~5.12.1]{Olver2010})
	and Stirling's formula for the gamma function
	\[
	\left|\frac{\Gamma(s)}{
		\sqrt{2\pi/s}\left(
		s/e
		\right)^s
	} - 1 \right|=
	O\left( \frac{1}{s} \right)
	\]
	as $s \to \infty$ (see \cite[item~5.11.3]{Olver2010}.
	Sharper estimates can be found in 
	\cite[Theorem~5]{Gordon1994}). 
	Combining \eqref{eq:tc_int_equiv}--\eqref{eq:Beta_asymp}, 
	we obtain the desired estimate \eqref{eq:tc_int_bound}.
\end{proof}

\begin{proof}[Proof of Proposition~\ref{prop:fnaw_bound}.]
	Let $s > \alpha \geq 0$ and $0 < c < \omega_{\min} \leq 
	\omega_{\max} < \infty$.
	We will give an upper bound of
	\[
	\big\|
	f_{n,\alpha,(\omega_k)}
	\big \|_{\mathcal{D}_{s,0}} =
	\int_{-\pi/2}^{\pi/2} \cos^s\theta
	\int_0^{\infty}
	\big| f_{n,\alpha,(\omega_k)}'(re^{i\theta})\big| 
	dr d\theta
	\]
	for $(\omega_k)_{k\in \mathbb{N}} \in \mathcal{S}( \omega_{\min},\omega_{\max})$.
	The derivative $ f_{n,\alpha,(\omega_k)}'$ is given by
	\begin{align*}
		f_{n,\alpha,(\omega_k)}'(z) 
		&=\frac{-\alpha}{(z+c+ \omega_{\max})^{\alpha+1}}
		\prod_{k=1}^n
		\frac{z+c-\omega_k}{z+c+\omega_k}  +
		\frac{2}{(z+c+ \omega_{\max})^{\alpha}}
		\sum_{\ell=1}^n
		\frac{\omega_{\ell}}{(z+c+\omega_{\ell})^2}
		\prod_{k=1,\,k\not=\ell }^n
		\frac{z+c-\omega_k}{z+c+\omega_k}.
	\end{align*}
	Let $-\pi/2 < \theta < \pi/2$, and 
	set $\delta = \delta (\theta) \coloneqq \cos \theta \in (0,1]$. 
	We have from Lemma~\ref{lem:hol1} that
	\[
	\left|\frac{re^{i\theta }+c-\omega}{re^{i\theta }+c+\omega} \right|^2
	= \frac{r^2+2\delta (c-\omega) r + (c-\omega)^2}{r^2+2\delta (c+\omega)r + (c+\omega)^2} \leq 
	\left( 
	\frac{r+\delta |c-\omega|}{r+\delta (c+\omega)}
	\right)^2
	\]
	for all $r,\omega >0$.
	If $c<\omega \leq \omega_{\max}$, then
	\[
	\frac{r+\delta |c-\omega|}{r+\delta (c+\omega)} 
	\leq \frac{r+\delta (-c + \omega_{\max})}{r+\delta (c+\omega_{\max})}
	\]
	for all $r >0$.
	We also have that 
	\begin{align*}
		\left|
		\frac{1}{\big(re^{i\theta}+ c+\omega_{\max}\big)^{\alpha+1}}
		\right|
		&\leq
		\frac{1}{\big(r+\delta(c+\omega_{\max})\big)^{\alpha+1}}.
	\end{align*}
	If $\omega_{\min} \leq \omega_{\ell} \leq \omega_{\max}$,
	then 
	\begin{align*}
		\frac{\omega_{\ell}}{r+ \delta(c+\omega_{\ell})}
		\leq 
		\frac{\omega_{\max}}{r+ \delta(c+\omega_{\max})}
		\quad \text{and} \quad 
		\frac{r+ \delta(c+\omega_{\max})}{r+ \delta(c+\omega_{\ell})}
		\leq \frac{c+\omega_{\max}}{c+\omega_{\min}}
	\end{align*}
	for all $r >0$,
	and hence
	\begin{align*}
		\left|
		\frac{1}{(re^{i\theta}+ c+\omega_{\max})^{\alpha}}
		\cdot	\frac{\omega_{\ell}}{(re^{i\theta}+c+\omega_{\ell})^2}
		\right|
		&\leq 
		\frac{1}{\big(r+\delta(c+\omega_{\max})\big)^{\alpha}}
		\cdot \frac{\omega_{\ell}}{\big(r+\delta(c+\omega_{\ell})\big)^{2}} \\
		&\leq\frac{C}{\big(r+\delta( c+\omega_{\max})\big)^{\alpha+2}},
	\end{align*}
	where
	\[
	C =
	C(\omega_{\min},\omega_{\max}) \coloneqq
	\frac{\omega_{\max}(c+\omega_{\max})}{c+\omega_{\min}}.
	\]
	From these estimates, we obtain
	\[
	\big| f_{n,\alpha,(\omega_k)}'(re^{i\theta}) \big| \leq 
	\alpha g_{n,\alpha, \omega_{\max}}(r) + 2C n h_{n,\alpha, \omega_{\max}}(r),
	\]
	where
	\begin{align*}
		g_{n,\alpha, \omega_{\max}}(r) &\coloneqq
		\frac{\big(r+\delta (-c+\omega_{\max})\big)^{n}}{\big(r+\delta (c+\omega_{\max})\big)^{n+\alpha+1}},\\
		h_{n,\alpha, \omega_{\max}}(r) &\coloneqq
		\frac{\big(r+\delta (-c+\omega_{\max})\big)^{n-1}}{\big(r+\delta (c+\omega_{\max})\big)^{n+\alpha+1}}.
	\end{align*}
	By Lemma~\ref{lem:tc_int_bound}, 
	there exist constants $ M_g, M_h >0$, depending only on $\alpha$, 
	such that 
	\[
	\int_0^{\infty} \big| f_{n,\alpha,(\omega_k)}'(re^{i\theta}) \big| dr
	\leq \frac{\alpha  M_g}{(2 c \delta )^{\alpha} n^{\alpha}}
	+\frac{2 C M_h}
	{(2 c \delta )^{\alpha+1} n^{\alpha}}
	\]
	for all $n \in \mathbb{N}$ and 
	$(\omega_k)_{k=1}^{\infty} \in \mathcal{S}( \omega_{\min},\omega_{\max})$.
	The beta function ${\mathrm B}$ satisfies
	\[
	{\mathrm B}\left(
	\frac{\sigma+1}{2},\frac{1}{2}
	\right) = 
	\int_{-\pi/2}^{\pi/2} \cos^\sigma \theta d\theta
	\]
	for all $\sigma > -1$;
	see, e.g., \cite[item~5.12.2]{Olver2010}.
	Recalling that $\delta =  \cos \theta$ and $s > \alpha$, 
	we derive
	\[
	\big\|
	f_{n,\alpha,(\omega_k)}
	\big\|_{\mathcal{D}_{s,0}} 
	\leq 
	{\mathrm B}\left(
	\frac{s-\alpha+1}{2},\frac{1}{2}
	\right)
	\frac{\alpha  M_g}{(2c)^{\alpha}n^{\alpha}} + 
	{\mathrm B}\left(
	\frac{s-\alpha}{2},\frac{1}{2}
	\right)
	\frac{2CM_h}{(2c)^{\alpha+1}n^{\alpha}}
	\eqqcolon
	\frac{ M}{n^{\alpha}} 
	\]
	for all $n \in \mathbb{N}$ and 
	$(\omega_k)_{k=1}^{\infty} \in \mathcal{S}( \omega_{\min},\omega_{\max})$.
\end{proof}

Now we are in a position to prove Theorem~\ref{thm:CN_holomorphic} via the 
estimate~\eqref{eq:fDs_bound}.
\begin{proof}[Proof of Theorem~\ref{thm:CN_holomorphic}.]
	By assumption, there exists $c\in(0,\omega_{\min})$ 
	such that $B \coloneqq A-c \in \text{Sect}(\pi/2-)$. Let $s > \alpha \geq 0$.
	Using the function $f_{n,\alpha,(\omega_k)}$ defined by \eqref{eq:fnaw_def},
	we obtain
	\begin{align*}
		f_{n,\alpha,(\omega_k)}(B) &=
		\left(
		\prod_{k=1}^n
		V_{\omega_k}(A)
		\right) (A+ \omega_{\max} )^{-\alpha}.
	\end{align*}
	The estimate \eqref{eq:fDs_bound} yields
	\[
	\big\|f_{n,\alpha,(\omega_k)}(B) \big\| \leq \frac{2^s M_B^{\lceil s+ 1\rceil}}{\pi} \big\|f_{n,\alpha,(\omega_k)}\big\|_{\mathcal{D}_{s,0}},
	\]
	where $M_B \coloneqq \sup_{z \in \mathbb{C}_+} \|z(z+B)^{-1}\|$.
	We also have that 
	\[
	\left\|\left(
	\prod_{k=1}^n
	V_{\omega_k}(A)
	\right)A^{-\alpha}
	\right\| \leq  \|(\omega_{\max}A^{-1} + 1)^{\alpha} \|\,
	\left\|\left(
	\prod_{k=1}^n
	V_{\omega_k}(A)
	\right) (A+ \omega_{\max} )^{-\alpha}\right\|.
	\]
	By Proposition~\ref{prop:fnaw_bound},
	there exists $M>0$ 
	such that the desired estimate \eqref{eq:VA_bound} holds
	for all $n \in \mathbb{N}$ and 
	$(\omega_k)_{k \in \mathbb{N}} \in \mathcal{S}(\omega_{\min},\omega_{\max})$.
\end{proof}
We see from a simple example that
the estimate given in
Proposition~\ref{prop:fnaw_bound} cannot  in general be improved.
\begin{example}
		Let $n \in \mathbb{N}$, $\alpha >0$, and $s>-1$.
		Define
		\[
		f_{n,\alpha}(z) \coloneqq \frac{z^n}{(z+1)^{n+\alpha}},\quad 
		z \in \mathbb{C}_+,
		\]
		which coincides with $f_{n,\alpha,(\omega_k)}$ defined by \eqref{eq:fnaw_def}
		if $c = \omega_{\min} = \omega_{\max} = 1/2$.
		Using \cite[Theorem~4.8.(iii) and Lemma~4.13.(ii)]{Batty2023},
		we obtain
		\begin{equation}
			\label{eq:fn_estimate1}
			\|f_{n,\alpha}\|_{\mathcal{D}_{s,0}} \geq 
			\frac{\pi}{2^{(3s+4)/2}}
			\sup_{\xi > 0}f_{n,\alpha}(\xi).
		\end{equation}
		Moreover,
		\begin{equation}
			\label{eq:fn_estimate2}
			\sup_{\xi > 0}f_{n,\alpha}(\xi) \geq 
			f_{n,\alpha}(
			n
			) = \left(
			1+ \frac{1}{n}
			\right)^{-n} \frac{1}{
				(
				n+1
				)^{\alpha}} .
		\end{equation}
		Combining \eqref{eq:fn_estimate1} and \eqref{eq:fn_estimate2},
		we conclude that 
		\[
		\|f_{n,\alpha}\|_{\mathcal{D}_{s,0}} \geq \frac{M}{n^{\alpha}}
		\]
		for all $n \in\mathbb{N}$ and some $M >0$.
\end{example}

\section*{Acknowledgments}
This work was supported in part by JSPS KAKENHI Grant Number
24K06866.
The author would like to thank the anonymous referee for numerous helpful comments and suggestions.

\end{document}